\documentclass[11pt,dfraft]{amsart}
\usepackage{graphicx,color}
\usepackage{latexsym}
\usepackage{graphicx}
\usepackage{amsthm,amsmath, amssymb,amsfonts,esint}

\usepackage{amsmath,amssymb,amsfonts,amsthm}
\usepackage{layout,textcomp}

\theoremstyle{plain}

\theoremstyle{plain}
\newtheorem{theorem}{Theorem} [section]

\newtheorem{lemma}[theorem]{Lemma}
\newtheorem{proposition}[theorem]{Proposition}

\theoremstyle{definition}

\theoremstyle{remark}
\newtheorem{remark}[theorem]{Remark}

\DeclareMathOperator{\tr}{tr}

\DeclareMathOperator{\opRic}{Ric}

\numberwithin{theorem}{section}
\numberwithin{equation}{section}
\numberwithin{figure}{section}

\newcommand{\Rc}{\text{\rm Rc}}
\newcommand{\Sph}{\mathbb{S}}

\newcommand{\minn}{{m\in\mathbb{N}}}
\newcommand{\Cal}{\mathcal}
\newcommand{\opId}{\text{\rm{Id}}}
\newcommand{\opSupp}{\text{\rm{supp}}}
\newcommand{\opKer}{\text{\rm{Ker}}}
\newcommand{\opdVol}{dV}
\newcommand{\opcat}{\text{\rm cat}}
\newcommand{\opsgn}{\text{\rm sgn}}
\newcommand{\minter}{\cap}
\newcommand{\oprob}{\text{\rm{rob}}}
\newcommand{\opIm}{\text{\rm{Im}}}
\newcommand{\opSym}{\text{\rm{Symm}}}
\newcommand{\oploc}{\text{\rm{loc}}}

\newcommand{\opNull}{\text{\rm{Null}}}
\newcommand{\opCodim}{\text{\rm{Codim}}}
\newcommand{\opDim}{\text{\rm{Dim}}}

\newcommand{\opdisk}{\text{\rm{disk}}}

\def\makeop#1{\global\expandafter\def\csname op#1\endcsname{{\text{{\rm{#1}}}}}}%
\def\makeopsmall#1{\global\expandafter\def\csname op#1\endcsname{{\text{{\rm{\lowercase{#1}}}}}}}%
\def\munion{\mathop{\cup}}%
\def\minter{\mathop{\cap}}%
\makeop{Ext}%
\makeop{Int}%
\makeop{Dist}%
\makeop{Diam}%
\makeop{Length}%
\def\ninn{{n\in\Bbb{N}}}%
\def\minn{{m\in\Bbb{N}}}%
\def\mliminf{\mathop{{\text{LimInf}}}}%
\makeop{Dim}%
\makeop{Coker}%
\makeop{Tr}%
\makeop{Adj}%
\makeop{Det}%
\makeop{End}%
\makeop{Lin}%
\makeop{Sym}%
\makeop{Mult}%
\makeop{dx}%
\makeop{dy}%
\makeop{dz}%
\makeop{dt}%
\makeop{dArea}%
\makeop{supp}%
\makeop{Hess}%
\makeop{Lip}%
\makeop{Re}%
\makeop{Im}%
\makeop{Arg}%
\makeop{Log}%
\newcommand{\tildeExp}{\text{\rm E}}%

\makeopsmall{Cos}%
\makeopsmall{Sin}%
\makeopsmall{Tan}%
\makeopsmall{Sec}%
\makeopsmall{Cosec}%
\makeopsmall{Cot}%
\makeopsmall{ArcCos}%
\makeopsmall{ArcSin}%
\makeopsmall{ArcTan}%
\makeopsmall{ArcSec}%
\makeopsmall{ArcCosec}%
\makeopsmall{ArcCot}%
\makeopsmall{Cosh}%
\makeopsmall{Sinh}%
\makeopsmall{Tanh}%
\makeopsmall{ArcCosh}%
\makeopsmall{ArcSinh}%
\makeopsmall{ArcTanh}%
\makeop{Vol}%
\makeop{Area}%
\makeop{Riem}%
\makeop{Ric}%
\makeop{Scal}%
\makeop{Euc}%
\makeop{Imm}%
\makeop{Emb}%
\makeop{Id}%
\makeop{Ad}%
\makeop{O}%
\makeop{SO}%
\makeop{SL}%
\makeop{GL}%
\makeop{Conf}%
\makeop{Homeo}%
\makeop{Diff}%
\makeop{Isom}%
\makeop{Ind}%
\makeop{Sig}%
\makeop{Spec}%
\makeop{Conv}%
\makeop{Max}%
\makeop{Min}%
\makeop{Mod}%
\makeop{Deg}%
\makeop{N}%
\newcommand{\loc}{\text{\rm loc}}%

\begin{document}

\title[Free boundary minimal annuli in convex three-manifolds]{Free boundary minimal annuli in convex three-manifolds}
\author[Davi M{{a}}ximo]{Davi M{{a}}ximo}
\address{Department of Mathematics, Stanford University, Stanford, CA 94305 and Mathematical Sciences Research Institute, Berkeley, CA 94720}
\email{maximo@math.stanford.edu}
\author[Ivaldo Nunes]{Ivaldo Nunes}
\address{Universidade Federal do Maranh\~{a}o, S\~{a}o Lu\'{i}s, MA - Brazil}
\email{ivaldo82@impa.br}
\author[Graham Smith]{Graham Smith}
\address{Universidade Federal do Rio de Janeiro, Rio de Janeiro, RJ - Brazil}

\date{}
\begin{abstract}
We prove the existence of free boundary minimal annuli inside suitably convex subsets of three-dimensional Riemannian manifolds with nonnegative Ricci curvature $-$ including strictly convex domains of the Euclidean space $\mathbb{R}^3$.
\end{abstract}

\maketitle

%%%%%%%%%%%%%%%%%%%%%%%%%%%%%%%%%%%%%%%%%%%%%%%%
%%%%%%%%%%%%Introduction%%%%%%%%%%%%%%%%%%%%
%%%%%%%%%%%%%%%%%%%%%%%%%%%%%%%%%%%%%%%%%%%%%%
\section{Introduction}
\subsection{Definitions and results} %In this paper, we prove the existence of free boundary minimal annuli inside convex subsets of three-dimensional manifolds of nonnegative Ricci curvature. We begin with some definitions.
Let $M$ be a compact three-dimensional manifold with smooth boundary and $g$ a Riemannian metric over $M$. We say that a smooth compact surface $\Sigma$ in $M$ with $\partial\Sigma\subseteq\partial M$ is {\it free boundary minimal} with respect to the metric $g$ whenever it has zero mean curvature, and $T\Sigma$ is orthogonal to $T\partial M$ at every point of $\partial\Sigma$.

Free boundary minimal surfaces are precisely the critical points of the area functional for surfaces in $M$ with boundary in $\partial M$. These surfaces were already studied in the nineteenth century, notably with Schwarz's work on Gergonne's problem (c.f., for example, \cite{DHT}), and have since attracted the interest of numerous mathematicians, including Courant \cite{Cou}, Lewy \cite{Lew}, Meeks and Yau \cite{MY}, Smyth \cite{Smy}, Nitsche \cite{Ni1}, Ros \cite{Ros1}, and Fraser and Schoen \cite{FrSc1,FrSc2}, to name but a few.

The problem of existence of free boundary minimal disks in domains of $\mathbb{R}^3$ diffeomorphic to the three-ball was studied in the mid-eighties by Struwe \cite{Stru}, using the $\alpha$-pertubed method of Sacks-Uhlenbeck for parametric surfaces, and by Gr\"uter and Jost \cite{GrueterJost}, using several ingredients from geometric measure theory, including the min-max theory of Almgren-Pitts. In particular, Gr\"uter and Jost showed the existence of properly embedded free boundary minimal disks inside strictly convex subsets of $\Bbb{R}^3$. In both cases, the techniques used leave open the problem of existence of free boundary minimal surfaces of non-trivial prescribed topology. We prove existence for the case of annuli:
\begin{theorem}\label{babythm}
If $K\subseteq\Bbb{R}^3$ is a compact, strictly convex subset of $\Bbb{R}^3$ with smooth boundary, then there exists a properly embedded free boundary minimal annulus $\Sigma$ in $K$.
\end{theorem}
\begin{remark}
In fact, the techniques of this paper also recover the result \cite{GrueterJost} of Gr\"uter and Jost (c.f. Remark \ref{RmkFinalRemark}).
\end{remark}

We actually prove a more general existence result for free boundary minimal annuli inside suitably convex subsets of three-manifolds with nonnegative Ricci curvature, of which Theorem \ref{babythm} is an immediate consequence. Existence results for free boundary minimal surfaces in general Riemannian manifolds have appeared in the literature before. Recently, in \cite{Li}, using Almgren-Pitts' min-max theory, Li proved a general existence result for properly embedded free boundary minimal surfaces in arbitrary three-manifolds with boundary. This result assumes no curvature conditions on the boundary and, in addition, using recent ideas from \cite{DLP} of De Lellis and Pellandini, provides genus bounds for the resulting surfaces. In particular, whenever the ambient manifold is diffeomorphic to the three-ball, Li's result implies the existence of an oriented free boundary minimal surface of genus zero, but it gives no information on the number of connected components of the boundary. We refer the interested reader to the introduction of \cite{Li} for a discussion on other existence results for free boundary minimal surfaces. Our general result can be stated as follows:

\begin{theorem}\label{ThmMainTheorem}
If $(M,g)$ is a smooth, compact, functionally strictly convex Riemannian three-manifold of nonnegative Ricci curvature, then there exists a properly embedded annulus $\Sigma\subseteq M$ which is free boundary minimal with respect to $g$.
\end{theorem}

We clarify the notion of convexity used here. $(M,g)$ is said to be {\it functionally strictly convex} whenever there exists a smooth function $f:M\rightarrow[0,1]$ which is strictly convex with respect to the metric $g$ and whose restriction to $\partial M$ is constant and equal to $1$ (recall that $f$ is said to be {\it strictly convex} with respect to a given metric whenever its Hessian is everywhere positive definite). Functional strict convexity may be thought of as a barrier condition in the sense of PDEs. In addition, if $M$ is an open subset of $\mathbb{R}^3$ with smooth boundary, and if $\delta$ is the Euclidean metric over $\mathbb{R}^3$, then $(M,\delta)$ is functionally strictly convex if and only if it is strictly convex in the usual sense. It follows that Theorem \ref{babythm} is an immediate consequence of Theorem \ref{ThmMainTheorem}.

In more general manifolds, functional strict convexity trivially implies strict convexity in the usual sense although the converse does not in general hold. The interest of this concept follows from the observation (c.f. Proposition \ref{ThmConnectedness}, below) that the space of functionally strictly convex manifolds is connected, which is a necessary prerequisite for the degree theoretic techniques of this paper to be applied. Although other connected spaces of manifolds with locally strictly convex boundary can be constructed (using, for example, \cite{Huisken}), we feel the condition of functional strict convexity is the simplest.

\subsection{Idea of the proof} Theorem \ref{ThmMainTheorem} is proven using a differential topological technique inspired by the work \cite{WhiteII} of White. We reason as follows. Let $\Sigma$ be a compact oriented surface with boundary. Let $\mathcal{E}$ be the space of equivalence classes $[e]$ of embeddings $e:\Sigma\rightarrow M$ modulo reparametrisation. Let $(g_x)_{x\in X}$ be a smooth family of Riemannian metrics with {\it positive} Ricci curvature parametrised by a compact, connected, finite-dimensional manifold $X$ (possibly with non-trivial boundary). Let $\mathcal{Z}(X)\subseteq X\times\mathcal{E}$ to be the set of all pairs $(x,[e])$ such that $e$ is free boundary minimal with respect to $g_x$, and let $\Pi:\mathcal{Z}(X)\rightarrow X$ be the projection onto the first factor. $\Pi$ is trivially continuous, and, by the compactness result of \cite{FraserLi}, $\Pi$ is proper.

If $\Cal{Z}(X)$ were a finite-dimensional differential manifold with the same dimension as $X$ and if, moreover, $\Pi$ were to map $\partial\Cal{Z}(X)$ into $\partial X$, then it would follow from classical differential topology (c.f.~\cite{GuillemanPollack}) that $\Pi$ would have a well-defined ${\mathbb Z}_2$-valued mapping degree. If, in addition, both $X$ and $\Cal{Z}(X)$ were shown to be orientable, then this degree could be taken to be integer-valued. Furthermore, this mapping degree would be independent of $X$, and since knowing $\Pi^{-1}(Y)$ for any subset $Y$ of $X$ amounts to knowing the space of free boundary minimal embeddings for any given metric, it would then yield the sort of existence result that we require. We show that, although $\Cal{Z}(X)$ might not necessarily have the aforementioned properties, $X$ may be embedded into a higher dimensional manifold $\tilde{X}$ for which these properties do indeed hold. The proof of Theorem \ref{ThmMainTheorem} for metrics with positive Ricci curvature follows by showing this degree to be non-zero when $\Sigma$ is topologically an annulus. From it, by a perturbative analysis, we finish the proof to include metrics with nonnegative Ricci curvature.

\subsection{Overview of the paper} The reader familiar with the work \cite{WhiteII} of White will notice both similarities and differences to his approach. The key observation in the current setting is that the Jacobi operator $\text{\rm J}:=(\text{\rm J}^h,\text{\rm J}^\theta)$, which measures the perturbations of the mean curvature and of the boundary angle resulting from a normal perturbation of the embedding, actually defines a {\it Fredholm mapping of Fredholm index zero} (Proposition \ref{PropEllipticRegularityOfJacobiOperatorII}). This brings free boundary problems within the scope of White's analysis with minimal technical modifications. We have nonetheless chosen to further adapt White's ideas in two respects, which, although not strictly necessary in the current context, will be of use, we believe, for future applications. First, we have chosen a non-variational approach, treating free boundary minimal surfaces as zeroes vector fields over infinite-dimensional manifolds rather than as critical points of functionals. This allows one to study not only free boundary minimal surfaces (which are variational), but also other, non-variational, notions of curvature such as, for example, extrinsic curvature. Second, wheras White studies the problem by constructing infinite dimensional Banach manifolds of solutions, we focus instead on finite dimensional sections of the solution space. This allows one to treat a larger class of functionals over the solution space (such as, for example, the weakly smooth functionals introduced by the third author in \cite{SmiRos}). Finally, the explicit calculation of the degree carried out in Section $6$ requires considerable modifications of White's argument in order to adapt it to the very different geometrical setting studied here.

The paper is structured as follows. We underline that we have preferred to sacrifice brevity in the interests of clarity and of obtaining a relatively self-contained text.

\subsubsection{Section 2} We construct the framework to be used throughout the paper. We introduce the space $\mathcal{E}$ of reparametrisation equivalence classes of embeddings, $e$, of a given surface, $\Sigma$, into $M$ such that $e(\partial\Sigma)\subseteq\partial M$. For any finite dimensional family, $X:=(g_x)_{x\in X}$, of metrics, we define the solution space $\mathcal{Z}(X)$ as outlined above, and we define $\Pi:\mathcal{Z}(X)\longrightarrow X$ to be the projection onto the first factor. At this stage, we are only interested in $\mathcal{E}$ and $\mathcal{Z}(X)$ as topological spaces with the obvious topologies: more sophisticated structures will be introduced in Section $3$. It follows that $\Pi$ is continuous and, by recent work of Fraser and Li \cite{FraserLi}, $\Pi$ is also proper. The formal construction of a $\mathbb{Z}$-valued mapping degree of $\Pi$ and its explicit calculation in certain cases constitute the main aims of this paper.

The remainder of Section $2$ is devoted to studying the infinitesimal theory of extremal embeddings. In Section \ref{JacobiOperators}, we calculate the Jacobi operator $\text{\rm J}:=(\text{\rm J}^h,\text{\rm J}^\theta)$ of an embedding, where $\text{\rm J}^h$ is the usual Jacobi operator of mean curvature, and $\text{\rm J}^\theta$ measures the perturbation of the boundary angle arising from a normal perturbation of the embedding. In Section \ref{PerturbationOperators}, we calculate the perturbation operator $\text{\rm P}:=(\text{\rm P}^h,\text{\rm P}^\theta)$ of an embedding, which measures the perturbations of mean curvature and of the boundary angle arising from perturbations of the ambient metric. In Section \ref{SubHeadHoelderSpaces} we review the general theory of elliptic operators, and in Section \ref{TheEllipticTheoryOfJacobiOperators} we show that $\text{\rm J}$ defines a Fredholm mapping of Fredholm index zero. As indicated above, this key observation allows us to extend the degree theory of \cite{WhiteII} to the current context with minimal technical difficulty.

\subsubsection{Section 3} We introduce the local theory of extremal embeddings. In Section \ref{TheLocalSolutionSpace} we introduce ``graph charts'' which map open subsets of $\mathcal{E}$ homeomorphically onto open subsets of $C^\infty(\Sigma)$. Viewing these charts as coordinate charts, we treat $\mathcal{E}$ formally as an infinite dimensional manifold. Within a given graph chart, we define the mean curvature and boundary angle functionals, $H$ and $\Theta$ respectively. The zero-set of the pair $(H,\Theta)$ coincides over each chart with the solution space $\mathcal{Z}(X)$. This makes $\Cal{Z}(X)$ amenable to standard functional analytic techniques. In Section \ref{BanachSpaces}, we review the theories of H\"older spaces and of smooth maps over Banach spaces. In Section \ref{Conjugations}, we study the relationship between the functionals $H$ and $\Theta$ and the perturbation and Jacobi operators introduced in Sections \ref{JacobiOperators} and \ref{PerturbationOperators}.

It is important to note the care required in carrying out this construction as, in contrast to the usual theory of differential manifolds, the transition maps between graph charts are not smooth. Fortunately, this does not present a serious problem in the current context, since it follows from elliptic regularity, as we shall see in Section $4$, that the restrictions of the transition maps to the solution space are indeed smooth, justifying the differential manifold formalism used.

\subsubsection{Section 4} We show how to extend $X$ so that $\Cal{Z}(X)$ carries the structure of a smooth compact oriented finite-dimensional differential manifold, possibly with boundary. In Section \ref{SubheadExtensionsAndSurjectivity}, we extend $X$ so that the functional $(H,\Theta)$ defined over each chart in Section \ref{TheLocalSolutionSpace} has surjective derivative at every point of $\mathcal{Z}(X)$. In Section \ref{Smoothness}, we use ellipticity together with the standard theory of smooth functionals over Banach spaces to show that $\mathcal{Z}(X)$ then restricts to a smooth, finite-dimensional submanifold of every graph chart and that the transition maps are smooth, thus furnishing $\mathcal{Z}(X)$ with the structure of a finite dimensional differential manifold. Finally, in Section \ref{TheOrientationOfTheSolutionSpace}, we recall general results of functional analysis which allow us to furnish $\mathcal{Z}(X)$ with a canonical orientation form, from which it immediately follows that $\Pi$ has a well-defined, integer-valued mapping degree, as desired.

\subsubsection{Section 5} In order to calculate the mapping degree of $\Pi$, we should count algebraically the number of extremal embeddings for some generic, admissable metric $g$. The problem is that generic metrics are hard to find explicitely. In particular, in the case at hand, the natural candidate, being the Euclidean metric in a closed ball, is clearly not generic. Indeed, generic metrics are characterised by having finitely many extremal embeddings all of which are non-degenerate, but in the Euclidean case, the action of the rotation group yields a non-trivial continuum of extremal embeddings out of every extremal embedding.

In this section, we study the technique used to calculate the degree in the case where the metric $g$ admits non-degenerate families of free boundary minimal embeddings. These are smooth families with the property that the Jacobi operator of each element of the family has kernel of dimension equal to that of the family itself. In Section \ref{NonDegenerateFamilies}, we show that if $[e]$ lies in a non-degenerate family, then for any infinitesimal perturbation $\delta g$ of the metric, there exists a (more or less) unique infinitesimal perturbation $\delta e$ of $e$ such that the mean curvature of $e+\delta e$ lies in a fixed, finite-dimensional space which we identify with the cotangent space of the family at $[e]$. In Section \ref{GlobalSectionsOverNonDegenerateFamilies}, by perturbing the whole family we therefore obtain a smooth section of the cotangent bundle of this family whose zeroes correspond to free boundary minimal embeddings for the perturbed metric. In Sections \ref{NonDegenerateSections} and \ref{DeterminingTheIndex}, we show moreover how to choose the metric perturbation in such a manner that this section has non-degenerate zeroes, which in turn correspond to free boundary minimal embeddings with non-degenerate Jacobi operators. In short, upon perturbing the metric, we transform a non-degenerate family into a finite set of non-degenerate free boundary minimal embeddings corresponding to the zeroes of a generic section of the cotangent space of this family thus allowing us to determine its contribution to the degree.

\subsubsection{Section 6} We apply the degree theory to the current setting in order to prove Theorem \ref{ThmMainTheorem}. Since, for topological reasons, the theory is developed for metrics of positive Ricci curvature, in Section \ref{RotationallyInvariantExtremalSurfaces} we use perturbation techniques to study rotationally symmetric free boundary minimal surfaces inside closed, strictly convex, geodesic balls in the three-dimensional sphere $\Sph^3(t)$. In Sections \ref{SectionNonDegenerateFamiliesOfDisks} and \ref{NonDegenerateFamiliesOfCatenoids}, by determining the dimensions of the kernels of the Jacobi operators of rotationally symmetric surfaces, we show that they define non-degenerate families of free boundary minimal embeddings, so that the results of Section $5$ may be applied in order to calculate their contribution to the mapping degree. In Section \ref{CalculatingTheDegree}, we adapt White's symmetry argument (c.f. \cite{WhiteII}) to the current context, showing that even though there may exist other extremal embeddings, their contribution to the mapping degree is zero. Finally, combining these results yields the mapping degree and the proof Theorem \ref{ThmMainTheorem}.

\subsection{Acknowledgements} The authors would like to thank Fernando Cod\'{a} Marques, Harold Rosenberg and Richard Schoen for interesting and useful discussions.  During the preparation of this work, the first author was partly supported by the National Science Foundation grant DMS-0932078, while in residence at the Mathematical Sciences Research Institute during the Fall of 2013, the second author was partly supported by a CNPq post-doctoral fellowship at the IMPA, Rio de Janeiro, and the third author was partly supported by a FAPERJ post-doctoral fellowship also at the IMPA, Rio de Janeiro.

%%%%%%%%%%%%%%%%%%%%%%%%%%%%%%%%%%%%%%%%%
% THE GLOBAL AND INFINITESIMAL THEORIES %
%%%%%%%%%%%%%%%%%%%%%%%%%%%%%%%%%%%%%%%%%

\section{The Global and Infinitesimal Theories}\label{TheGlobalAndInfinitesimalTheory}
\subsection{The solution space} Let $M$ be a compact  three-manifold with boundary and let $\Sigma$ be a compact surface with boundary. Throughout the sequel, we will assume that all manifolds are smooth and oriented. We denote by $\hat{\mathcal{E}}$ the space of all  proper embeddings $e:\Sigma\rightarrow M$ with the properties that $e(\partial \Sigma)\subseteq\partial M$ and $e(\partial\Sigma)=e(\Sigma)\cap\partial M$. We furnish this space with the topology of $C^\infty$ convergence. We say that two embeddings $e,e'\in\hat{\mathcal{E}}$ are {\it equivalent} whenever there exists an orientation-preserving diffeomorphism $\alpha:\Sigma\rightarrow \Sigma$ such that $e'=e\circ\alpha$. We denote by $\mathcal{E}$ the space of equivalence classes $[e]$ of elements $e$ of $\hat{\mathcal{E}}$ furnished with the quotient topology.

A metric $g$ over $M$ is said to be {\it admissable} whenever it has positive Ricci curvature and there exists a smooth function $f:M\rightarrow[0,1]$ which is strictly convex with respect to $g$ and whose restriction to $\partial M$ is constant and equal to $1$.  The following observation is key to developing a degree theory for free boundary minimal surfaces:
\begin{proposition}\label{ThmConnectedness}
The space of admissable metrics over $M$ is connected.
\end{proposition}
\begin{proof} Let $g$ be an admissable metric over $M$. %Observe that since $\partial M$ is locally strictly convex, the shortest curve in $M$ joining any two interior points is a geodesic which does not touch the boundary.
Since $f$ is strictly convex, it then follows that $f$ has a unique global minimum in $M$. Let $p_0\in M$ be this minimum, and without loss of generality, assume that $f(x_0)=0$. For all $\epsilon>0$, we define $M_\epsilon=f^{-1}([0,\epsilon])$. Let $\chi\in C_0^\infty(M)$ be a nonnegative function equal to $1$ near $p_0$, and let $(\Phi_t)_{t\in[0,\infty[}$ be the gradient flow of the vector field $X$ by $X=-(1-\chi)\nabla f/\|\nabla f\|^2$. Upon choosing the support of $\chi$ in a sufficiently small neighbourhood of $p_0$, $\Phi_{1-\epsilon}$ maps $M$ diffeomorphically into $M_\epsilon$. In other words, $(M,g)$ lies in the same connected component as $(M_\epsilon,g)$ for all $\epsilon>0$.

Let $d:M\rightarrow\Bbb{R}$ be the distance in $M$ to $p_0$. Let $r>0$ be such that the closure of $B_r(p_0)$ is contained in the interior of $M$ and $d^2$ is strictly convex over this ball. For all $t\in[0,1]$, we define $f_t = (1-t)f + td^2$. If $\epsilon\in]0,r^2[$ is chosen such that $f(x)>\epsilon$ for all $x\in\partial B_r(p_0)$. Then, for all $t$, $f_t^{-1}([0,\epsilon])$ is a strictly convex subset of $M$ contained inside $B_r(p_0)$. Now let $(\Phi_t)_{t\in[0,1]}$ be the gradient flow of the vector field $X_t = -(1-\chi)(\partial_tf_t)\nabla f_t/\|\nabla f_t\|^2$. Again, upon choosing the support of $\chi$ in a sufficiently small neighbourhood of $p_0$, $\Phi_t$ maps $M_\epsilon$ diffeomorphically into $f_t^{-1}([0,\epsilon])$. In other words, $(M,g)$ lies in the same connected component as $(B_{\sqrt{\epsilon}}(p_0), g)$. Moreover, upon rescaling $g$, we may suppose that $\epsilon=1$.

We now identify the tangent space to $M$ at $p_0$ with $\Bbb{R}^3$. Upon pulling back through the exponential map, we view $g$ as a metric over $\Bbb{R}^3$. For $t\in[0,1]$, we define the metric $g_t$ by $g_t(x)=g(tx)$. Observe that $g_0$ is the Euclidean metric. Without loss of generality, we may suppose that the metric $g_t$ is sufficiently close to the Euclidean metric that the function $h:=\|x\|^2$ is strictly convex with respect to this metric. Denote ${g}_{t,s}:=e^{-2sh} g_t$ and let $\Rc^{t,s}$ be the Ricci-curvature tensor of this metric. Then:
$$
\left.\frac{\partial}{\partial_s}\right|_{\tiny s=0}\hspace{-15pt }{\Rc}^{t,s} = (n-2)\opHess\, h + \Delta h g_t.
$$
Since $h$ is strictly convex with respect to $g_t$, and since $g_t$ has positive Ricci curvature for $t>0$, there exists $\epsilon>0$ such that for all $(t,s)\in[0,1]\times[0,\epsilon]$ such that $(t,s)\neq (0,0)$, $g_{t,s}$ also has positive Ricci curvature. In particular, $(M,g)$ lies in the same connected component as $(B_1(0),g_{0,s})$, for all small $s>0$, and the space of admissable metrics is therefore connected, as desired.\end{proof}

With $X$ a compact, finite-dimensional manifold possibly with non-trivial boundary, let $g:X\times M\rightarrow\text{\rm{Sym}}^+\,TM$ be a smooth function with the property that $g_x:=g(x,\cdot)$ is an admissable metric for all $x\in X$. We henceforth refer to the pair $(X,g)$ simply by $X$. We define $\mathcal{Z}(X)\subseteq X\times\mathcal{E}$ to be the set of all pairs $(x,[e])$ such that $e$ is a free boundary minimal embedding with respect to the metric $g_x$.  We describe $\Cal{Z}(X)$ as the zero set of a functional. Indeed, for $(x,[e])\in X\times\mathcal{E}$ we denote by $N:\Sigma\longrightarrow TM$ the unit  normal vector field over $e$ with respect to $g_x$ which is compatible with the orientation  and we denote by $A:\Sigma\longrightarrow\text{\rm{End}}(T\Sigma)$ and $H:\Sigma\longrightarrow\Bbb{R}$ the corresponding {\it shape operator} and {\it mean curvature} respectively. That is, at each point $p\in \Sigma$:
$$
H =\tr A.
$$

\noindent We denote by $\nu$ the outward-pointing unit normal vector field over $\partial M$ with respect to $g_x$, and we denote by $\Theta:\partial \Sigma\longrightarrow\mathbb{R}$ the {\it boundary angle} that $e(\Sigma)$ makes with $\partial M$ with respect to this metric. That is,  at each $p\in\partial \Sigma$:
$$
\Theta = g\left(\nu,N\right).
$$
\begin{remark} The geometric quantities we have just defined depend on $(x,e)$. To avoid confusion, we often explicit this dependence in our notation by writing $N_{x,e}, H_{x,e}, \Theta_{x,e},$ etc.
\end{remark}

\noindent We define the {\it solution space} $\mathcal{Z}(X)\subseteq X\times\Cal{E}$ by:
$$
\mathcal{Z}(X) = \left\{ (x,[e])\in X\times\Cal{E}\ |\ H_{x,e}=0,\ \Theta_{x,e}=0\right\},
$$
and we define $\Pi:\mathcal{Z}(X)\rightarrow X$ to be the projection onto the first factor. Since both $H_{x,e}$ and $\Theta_{x,e}$ are equivariant under reparametrisation, this definition is consistent.

The main objective of this paper is to construct a ${\mathbb Z}$-valued mapping degree for the projection $\Pi$. A key element of this construction is the following compactness result:
\begin{theorem}[Fraser-Li \cite{FraserLi}]\label{ThmFraserLi}
Let $(g_m)_\minn$ be a sequence of metrics over $M$ of nonnegative Ricci curvature. Let $(e_m)_\minn:\Sigma\longrightarrow M$ be a sequence of embeddings such that, for all $m$, $e_m$ is a free boundary minimal embedding with respect to the metric $g_m$. If there exists a metric $g_\infty$ over $M$ towards which $(g_m)_\minn$ converges in the $C^\infty$ sense, and if $\partial M$ is strictly convex with respect to $g_\infty$, then there exists an embedding $e_\infty:\Sigma\rightarrow M$ and a sequence $(\alpha_m)_\minn:\Sigma\rightarrow\Sigma$ of diffeomorphisms of $\Sigma$ such that $(e_m\circ\alpha_m)_\minn$ subconverges towards $e_\infty$ in the $C^\infty$ sense. In particular, $e_\infty$ is a free boundary minimal embedding with respect to the metric $g_\infty$.
\end{theorem}
\noindent In our current framework, this is restated (in slightly weaker form) as follows:
\begin{proposition}\label{ThmCompactness}
Let $\Pi:X\times\Cal{E}\longrightarrow X$ be the projection onto the first factor. Then the restriction of $\Pi$ to $\Cal{Z}(X)$ is proper.
\end{proposition}

If $\Cal{Z}(X)$ were a finite-dimensional differential manifold with boundary of dimension equal to that of $X$ and if, moreover, $\Pi$ were to map $\partial\Cal{Z}(X)$ into $\partial X$, then it would follow from classical differential topology that $\Pi$ has a well-defined ${\mathbb Z}_2$-valued mapping degree. Furthermore, this degree would be independent of $X$, and if, in addition, both $X$ and $\Cal{Z}(X)$ were orientable, then it could be taken to be integer-valued. The main objective of Sections $3$ and $4$ below is to show that although $\Cal{Z}(X)$ does not necessarily have the aforementioned properties, $X$ may be embedded into a higher dimensional manifold $\tilde{X}$ for which these properties actually hold. This is summarised in Theorem \ref{ThmIntegerValuedDegree} of Section 4. The existence result of Theorem \ref{ThmMainTheorem} then follows upon showing this degree to be non-zero in the case treated there. To this end, we require in particular Theorem \ref{ThmNonDegenerateFamilies}, which determines how smooth, non-degenerate families of solutions contribute to the degree. Theorems \ref{ThmIntegerValuedDegree} and \ref{ThmNonDegenerateFamilies} together constitute the main results of Sections $3$, $4$ and $5$, and the first-time reader may skim the rest, passing directly to Section $6$ after completing Section $2$ without losing much understanding.

%Thus, in Section \ref{JacobiOperators}, we calculate the Jacobi operator $\text{\rm J}:=(\text{\rm J}^h,\text{\rm J}^\theta)$ of an embedding, where $\text{\rm J}^h$ is the usual Jacobi operator of mean curvature, and $\text{\rm J}^\theta$ measures the perturbation of the boundary angle arising from a normal perturbation of the embedding. In Section \ref{PerturbationOperators}, we calculate the perturbation operator $P:=(\text{\rm P}_h,\text{\rm P}_\theta)$ of an embedding, which measures the perturbations of mean curvature and of the boundary angle arising from perturbations of the ambient metric. Finally, after reviewing the general theory of elliptic operators in Section \ref{SubHeadHoelderSpaces}, we show in Section \ref{TheEllipticTheoryOfJacobiOperators} that $J$ defines a Fredholm mapping of Fredholm index $0$, which, as indicated above, allows us to extend White's degree theory of minimal embeddings to the current context with minimal technical difficulty.

We devote the remainder of this section to studying the infinitesimal theory of minimal embeddings with free boundary. Our goal is to prove that the  Jacobi operator $\text{\rm J}:=(\text{\rm J}^h,\text{\rm J}^\theta)$, which measures the perturbation of mean curvature as well as the perturbation of the boundary angle resulting from a normal perturbation of the embedding, defines a Fredholm {mapping of Fredholm index zero}.

\subsection{Jacobi operators}\label{JacobiOperators} Given $(x,[e])\in\Cal{Z}(X)$, we denote by $\text{\rm J}^h:C^\infty(\Sigma)\rightarrow C^\infty(\Sigma)$ and by $\text{\rm J}^\theta:C^\infty(\Sigma)\rightarrow C^\infty(\partial\Sigma)$ respectively the {\it Jacobi operator of mean curvature} of $e$ and the {\it Jacobi operator of the boundary angle} of $e$ with respect to $g_x$. That is, $\text{\rm J}^h$ and $\text{\rm J}^\theta$ are defined such that if $f:(-\delta,\delta)\times\Sigma\longrightarrow M$ is a smooth mapping with the properties that $e=f(0,\cdot)$, $e_t:=f(t,\cdot)$ is an embedding for all $t$, and $\frac{\partial f}{\partial t}\big|_{t=0} =\varphi\,{{N}}$ for some $\varphi\in C^\infty(\Sigma)$, then
$$
\text{\rm J}^h\varphi=\left.\frac{\partial}{\partial t}\right|_{\tiny t=0} \hspace{-15pt } H_{x,e_t},\, \  \text{\rm{and}}\, \ \text{\rm J}^\theta\varphi=\left.\frac{\partial}{\partial t}\right|_{\tiny t=0} \hspace{-15pt }\Theta_{x,e_t}.
$$

\noindent We denote by $\opRic$ the Ricci curvature tensor of $g_x$ and by $\Delta$ the Laplacian operator of $e^*g_x$ over $\Sigma$. We recall the second variation formula for the area:

\begin{lemma}\label{PropFormulaForJacobiOperatorOfMeanCurvature}
Given $(x,[e])\in\Cal{Z}(X)$,  for all $\varphi\in C^\infty(\Sigma)$:
$$
\text{\rm J}^h\varphi =-\Delta\varphi -\left(\opRic({{N}},{{N}})+\|A\|^2\right)\varphi.
$$
\end{lemma}
\begin{remark}
In particular, $\text{\rm J}^h$ is a second-order linear elliptic partial differential operator.
\end{remark}

\noindent Let $II$ denote the shape operator of $\partial M$ with respect to $g_x$ and the outward pointing normal $\nu$. Since $(x,[e])\in\Cal{Z}(X)$, along the boundary points of $\Sigma$, the vector $N$ lies in the tangent space of $\partial M$, and we define $\kappa:\partial \Sigma\longrightarrow\Bbb{R}$ by:
$$
\kappa = II(N,N).
$$
Moreover, the vector field $\nu\circ e$ coincides with the conormal to $e(\partial \Sigma)$ inside $e(\Sigma)$ with respect to $g_x$, and we therefore define the operator $\partial_{\nu}:C^\infty(\Sigma)\longrightarrow C^\infty(\partial \Sigma)$ to be the derivative in the direction of the vector field $\nu\circ e$. That is, for all $f\in C^\infty(\Sigma)$ and at each $p\in\partial \Sigma$:
$$
\partial_{\nu} f = \langle e^*\nu,df \rangle.
$$
The following result is proven in %Proposition $17$ of the Appendix  of
\cite{Ambrozio}:
\begin{proposition}\label{PropFormulaForJacobiOperatorOfBoundaryAngle}
Given $(x,[e])\in\Cal{Z}(X)$, for all $\varphi\in C^\infty(\Sigma)$:
$$
\text{\rm J}^\theta\varphi = \kappa\varphi\circ\epsilon - \partial_\nu\varphi,
$$
where $\epsilon:\partial \Sigma\rightarrow \Sigma$ is the canonical embedding.
\end{proposition}
Again, the geometric quantities we have just defined clearly depend on $(x,[e])$. To avoid confusion, we often explicit this dependence in our notation. In particular, we denote $\textrm{J}_{x,e}:=(\text{\rm J}^h_{x,e},\text{\rm J}^\theta_{x,e})$, and we refer to $\textrm{J}_{x,e}$ as the {\it Jacobi operator} of $[e]$ with respect to the metric $g_x$.

\subsection{Pertubation operators}\label{PerturbationOperators}For all $(x,[e])\in\Cal{Z}(X)$, we denote by $\textrm{P}^h_{x,e}:T_xX\rightarrow C^\infty(\Sigma)$ and by $\textrm{P}^\theta_{x,e}:T_xX\longrightarrow C^{\infty}(\partial\Sigma)$ respectively the {\it perturbation operator} of {\it mean curvature} of $e$ and the \textit{pertubation operator of the boundary angle} of $e$ with respect to changes in the metric. That is, if $\xi\in T_xX$, if $x:(-\delta,\delta)\rightarrow X$ is a smooth curve such that $x(0)=x$ and $\dot{x} (0)=\xi$, then we define:
$$
\textrm{P}^h_{x,e}\xi = \left.\frac{\partial}{\partial t}\right|_{\tiny t=0} \hspace{-15pt }H_{x_t,e},\, \  \text{\rm{and}}\, \ \text{\rm P}^\theta_{x,e}\varphi=\left.\frac{\partial}{\partial t}\right|_{\tiny t=0} \hspace{-15pt }\Theta_{x_t,e}.
$$
For all $(x,[e])\in\Cal{Z}(X)$, we denote $\textrm{P}_{x,e}:=\left(\textrm{P}^h_{x,e},\textrm{P}^\theta_{x,e}\right)$, and we refer to $\textrm{P}_{x,e}$ as the {\it perturbation operator} of $e$ with respect to changes in the metric.

It turns out only to be necessary to consider conformal perturbations of the ambient metric. Let $g:(-\delta,\delta)\times M\rightarrow\text{\rm{Sym}}^+(TM)$ be a smooth family of metrics. Denote $g_t:=g(t,\cdot)$ for all $t$ and $g(0)=g$. Let $e:\Sigma\rightarrow M$ be an embedding and let $N:\Sigma\rightarrow TM$ be the normal vector field over $e$ with respect to $g$ which is compatible with the orientation.

%Let $\theta:(-\delta,\delta)\times\partial \Sigma\rightarrow\Bbb{R}$ be such that for all $(t,p)\in(-\delta,\delta)\times\partial \Sigma$, $\theta_t(p):=\theta(t,p)$ is the boundary angle of the embedding $e$ with respect to the metric $g_t$ at the point $p$.
\begin{proposition}\label{PropVariationOfBoundaryAngle}
If $\dot{g}(0)=\varphi g$ for $\varphi\in C^\infty(M)$, then:
$$
\left.\frac{\partial}{\partial t}\right|_{\tiny t=0} \hspace{-15pt }\Theta_{g_t,e} = 0.
$$
\end{proposition}
\begin{proof}
By definition, a perturbation of the metric which is conformal up to order $1$ leaves angles invariant up to order $1$, and the result follows.
\end{proof}

The next proposition follows by direct calculation.
\begin{proposition}\label{PropVariationOfMeanCurvature}
If  $\dot{g}(0)=\varphi g$ for $\varphi\in C^\infty(M)$, then:
$$\left.\frac{\partial}{\partial t}\right|_{\tiny t=0}  \hspace{-15pt }H_{g_t,e}=d\varphi(N) - \frac{1}{2}\varphi H_{g(0),e}.
$$

\end{proposition}
% \begin{proof}We sketch the proof. For all $t$, let $\nabla_t$ be the Levi-Civita covariant derivative of $g_t$, and let ${(\Gamma_t)^k}_{ij}$ be the relative Christophel symbol of $\nabla_t$ with respect to $\nabla_0$. Differentiating the Kozhul formula (c.f., for example, Lemma $6.5$ of \cite{ChowKnopf}) yields:
% $$
% {{(\partial_t\Gamma)_0}^k}_{ij} = \frac{1}{2}(\varphi_{;i}{\delta^k}_j + \varphi_{;j}{\delta^k}_i - (g_0)^{kp}\varphi_{;p} (g_0)_{ij}).
%$$
% For all $t$, let $N_t:\Sigma\rightarrow TM$ be the unit normal vector field over $e$ with respect to $g_t$ which is compatible with the orientation. A straightforward calculation yields:
% $$
% (\partial_t\text{\sf{N}})_0 = -\frac{1}{2}(\varphi\circ e)\text{\sf{N}}_0.
% $$
% For all $t$, let $A_t:\Sigma\rightarrow\opEnd(T\Sigma)$ be the shape operator of $e$ with respect to $g_t$ and the normal direction $N_t$. Using the preceeding two relations we obtain:
% $$
% (\partial_t A)_0 = \frac{1}{2}(d\varphi\circ e)(\text{\sf{N}}_0)\opId - \frac{1}{2}(\varphi\circ e)A_0,
%$$
% and the result now follows upon taking the trace.\end{proof}

This yields the following surjectivity result:
\begin{proposition}\label{PropMetricsWhichYieldDesiredFunctions}
For all $f\in C^\infty(\Sigma)$, there exists $\varphi\in C^\infty(M)$ such that if $\dot{g}(0)=\varphi g$, then:
$$
\left.\frac{\partial}{\partial t}\right|_{\tiny t=0}  \hspace{-15pt }H_{g_t,e} = f.
$$
Moreover, for any neighbourhood $U$ of $e(\opSupp(f))$ in $M$, $\varphi$ may be chosen such that $\opSupp(\varphi)\subseteq U$.
\end{proposition}
\begin{proof}
We identify $\Sigma$ with its image $e(\Sigma)$ in $M$. We extend $\Sigma$ and $f$ smoothly beyond $\partial \Sigma$. Let $d$ be the signed distance function to $\Sigma$ in $M$ with respect to the metric $g$. Let $\pi$ be the closest point projection onto $\Sigma$ with respect to the metric $g$. There exists $\delta>0$ such that the restrictions of $d$ and $\pi$ to $d^{-1}(-\delta,\delta)\minter M$ are smooth. Let $\chi\in C_0^\infty(\Bbb{R})$ be supported in $(-\delta,\delta)$ and equal to $1$ near $0$. We define:
$$
\varphi = (\chi\circ d)(f\circ\pi).
$$
Observe that, restricted to $\Sigma$, $\varphi=0$ and $d\varphi(N)=f$. It follows then from Proposition \ref{PropVariationOfMeanCurvature} that if $\dot{g}(0)=\varphi g$, then $\left.\frac{\partial}{\partial t}\right|_{\tiny t=0} H_{g_t,e} = f$. Moreover, upon reducing $\delta$ if necessary, we may suppose that $\opSupp(\phi)\subseteq U$ and this completes the proof.
\end{proof}

Proposition \ref{PropMetricsWhichYieldDesiredFunctions} is already sufficient for the proof of Theorem \ref{PropSurjectivityForExtensions} of Section 4. However, the following refinement will prove useful:
\begin{proposition}\label{PropFindingTheRightPerturbationSpace}
Let $f_1,...,f_m\in C^\infty(\Sigma)$ be a basis for $\opKer(\text{\rm J}_{g,e})$. Let $p$ be a point in $\Sigma$ and let $U$ be a neighbourhood of $e(p)$ in $M$. Then, there exists functions $\varphi_1,...,\varphi_m\in C^\infty(M)$, all supported in $U$, such that for all $1\leqslant i,j\leqslant m$, if g(t) is a path of metrics with $\dot g(0)=\varphi_i g$, where $g(0)=g$, then:
$$
\left\langle \left.\frac{\partial}{\partial t}\right\rvert_{\tiny t=0} \hspace{-15pt }H_{g_t,e} , f_j \right\rangle  = \delta_{ij},
$$
where $\left\langle\cdot,\cdot \right\rangle$ is the $L^2$ inner product with respect to $e^*g$ over $\Sigma$.
\end{proposition}
\begin{remark}We will see in the following section that $\opKer(\textrm{J}_{g,e})$ is finite dimensional.
\end{remark}
\begin{proof}
We identify $\Sigma$ with its image $e(\Sigma)\subseteq M$. Let $r:C^\infty(\Sigma)\longrightarrow C^\infty(\Sigma\cap U)$ be the restriction mapping. For any vector $p:=(p_1,...,p_m)$ of points in $\Sigma\cap U$, we define the mapping $L_p:C^\infty(\Sigma\cap U)\rightarrow \Bbb{R}^m$ by:
$$
L_p(f) = (f(p_1),...,f(p_n)).
$$
Since $\text{\rm J}^h_{g,e}(f_k)=0$ for all $1\leqslant k\leqslant m$, and bearing in mind that $\text{\rm J}^h_{g,e}$ is a second-order elliptic linear partial-differential operator, it follows from Aronszajn's unique continuation theorem (c.f. \cite{Aronszajn}) that $r$ restricts to a linear isomorphism from $\opKer(\text{\rm J}_{g,e})$ to an $m$-dimensional subspace of $C^\infty(\Sigma\cap U)$. There therefore exists a vector $p$ such that $L_p$ defines a linear isomorphism from $\opKer(\text{\rm J}_{g,e})$ to $\Bbb{R}^m$.

Observe that, for all $1\leqslant k\leqslant m$:
$$
L_p(f)_k = \left\langle f,\delta_{p_k} \right\rangle
$$
where $\delta_{p_k}$ is the Dirac-delta distribution supported at $k$. For any vector $\psi:=(\psi_1,...,\psi_m)$ of smooth functions in $C_0^\infty(\Sigma\minter U)$, we define the mapping $L_\psi:C^\infty(\Sigma\cap U)\rightarrow\Bbb{R}^m$ such that for all $1\leqslant k\leqslant m$:
$$
L_\psi(f)_k = \left\langle f,\psi_k \right\rangle
$$
Observe that as $\psi$ converges to $(\delta_{p_1},...,\delta_{p_m})$ in the distributional sense, $L_\psi$ converges to $L_p$. There therefore exists a vector $\psi$ such that $L_\psi$ is invertible. We may suppose, moreover, that for all $1\leqslant k\leqslant m$, $\psi_k$ is supported in $\Sigma\minter U$. In addition, upon replacing each of $\psi_1,...,\psi_n$ by an appropriate linear combination of these functions if necessary, we may suppose that for all $1\leqslant i,j\leqslant m$:
$$
\left\langle \psi_i,  f_j \right\rangle= \delta_{ij}.
$$
By Proposition \ref{PropMetricsWhichYieldDesiredFunctions}, there exist functions $\varphi_1,...,\varphi_m\in C^\infty(M)$ such that for all $1\leqslant k\leqslant m$, $\opSupp(\varphi_k)\subseteq U$, and if $\dot{g}(0)=\varphi_kg$, then $\left.\frac{\partial}{\partial t}\right\rvert_{\tiny t=0} H_{g_t,e}=\psi_k$. Thus, for all $1\leqslant i,j\leqslant m$, if $(\partial_tg)_0=\varphi_p g_0$, then:
$$
\left\langle \left.\frac{\partial}{\partial t}\right\rvert_{\tiny t=0} \hspace{-15pt }H_{g_t,e} , f_j \right\rangle = \delta_{ij},
$$
as desired.
\end{proof}

\subsection{General elliptic theory}\label{SubHeadHoelderSpaces}For $\lambda\in[0,\infty]\setminus\Bbb{N}$, that is, $\lambda=k+\alpha$ where $k\in\mathbb{N}\munion\left\{\infty\right\}$ and $\alpha\in(0,1)$, and for any compact manifold $\Omega$, we denote by $C^\lambda(\Omega)$ the space of $\lambda$-times H\"older differentiable functions over $\Omega$. For $\lambda<\infty$, we denote by $\|\cdot\|_\lambda$ the $C^\lambda$-H\"older norm of $C^\lambda(\Omega)$ and we denote by $C^{*,\lambda}(\Omega)$ the closure of $C^\infty(\Omega)$ in $C^\lambda(\Omega)$. We remark that $C^{*,\lambda}(\Omega)$ is separable, but $C^{\lambda}(\Omega)$ is not (c.f. \cite{WhiteI}).

For $\varphi\in C^\infty(\partial \Omega)$, we define the {\it Robin operator} $R_\varphi:C^{*,\lambda+1}(\Omega)\longrightarrow C^{*,\lambda}(\partial \Omega)$ such that, for all $f\in C^{*,\lambda+1}(\Omega)$:
$$
R_\varphi(f) = \varphi(f\circ\epsilon) + \partial_\nu f,
$$
where $\epsilon:\partial \Omega\longrightarrow \Omega$ is the canonical embedding and $\partial_\nu f$ is the derivative of $f$ in the outward pointing conormal direction. For all $\lambda\in[0,\infty]\setminus\Bbb{N}$, we define $C^{*,\lambda+1}_{\oprob}(\Omega)$ to be the kernel of $R_\varphi$ in $C^{*,\lambda+1}(\Omega)$.

\begin{proposition}\label{PropOperatorIsSelfAdjoint}
If $\Delta$ is the Laplacian over $\Omega$, then for all $\xi,\eta\in C^{*,\lambda+2}_\oprob(\Omega)$:
$$
\int_\Omega \eta\Delta\xi\,\opdVol = \int_\Omega \xi\Delta\eta\,\opdVol,
$$
where $\opdVol$ is the volume form of $\Omega$.
\end{proposition}
\begin{proof}Since $\xi,\eta\in C^{*,\lambda+2}_\oprob(\Omega)$, we have $\eta\partial_\nu\xi - \xi\partial_\nu\eta = 0$ along $\partial\Omega$, and the result follows by Stokes' Theorem.
\end{proof}

We now recall some basic elliptic theory. Let $\Delta$ be the Laplacian of $\Omega$.

\begin{proposition}\label{PropBijectivityOverSmoothRobinSpaces}
$\opId-\Delta$ defines a bijective map from $C^\infty_\oprob(\Omega)$ into $C^\infty(\Omega)$.
\end{proposition}
%\label{ThmInvertibilityOverRobinSpaces}
\begin{proof} For all $k\in\Bbb{N}$, let $H^k(\Omega)$ be the Sobolev space of functions over $\Omega$ whose distributional derivatives up to and including order $k$ are of type $L^2$. For all $k$, by the Sobolev Trace Formula, %(c.f. Proposition $4.5$ of Section $4$ of \cite{TaylorI})
$R_\varphi$ defines a continuous linear mapping from $H^{k+2}(\Omega)$ into $H^{k+1/2}(\partial \Omega)$ from which it follows that $H^{k+2}_\oprob(\Omega):=\opKer(R_\varphi)$ is closed in $H^{k+2}(\Omega)$ and is therefore also a Hilbert space. By Exercise $3$ of Section $5.7$ of \cite{TaylorI}, for all $k\in\Bbb{N}$, $\opId-\Delta$ defines an invertible linear mapping from $H^{k+2}_\oprob(\Omega)$ into $H^k(\Omega)$. However:
$$
C^\infty_\oprob(\Omega) = \minter_{k\geqslant 0}H^{k+2}_\oprob(\Omega).
$$
In particular $\opId-\Delta$ has trivial kernel in $C^\infty_\oprob(\Omega)$ and injectivity follows. For surjectivity, choose $f\in C^\infty(\Omega)$. Since:
$$
C^\infty(\Omega) = \minter_{k\geqslant 0}H^k(\Omega),
$$
\noindent for all $k$, $f\in H^k(\Omega)$ and so there exists a unique function $g_k\in H^{k+2}_\oprob(\Omega)$ such that $(\opId-\Delta)g_k=f$. However, for $l\geqslant k$, $g_l\in H^{l+2}_\oprob(\Omega)\subseteq H^{k+2}_\oprob(\Omega)$, and so by uniqueness $g_l=g_k$. In particular:
$$
g:=g_k\in\minter_{k\geqslant 0}H^{k+2}_\oprob(\Omega) = C^\infty_\oprob(\Omega),
$$
and surjectivity follows.\end{proof}
\begin{proposition}\label{PropBijectivityOverSmoothSpaces}
$(\opId-\Delta,R_\varphi)$ defines a bijective map from $C^\infty(\Omega)$ into $C^\infty(\Omega)\times C^\infty(\partial \Omega)$.
\end{proposition}
\begin{proof}Choose $f\in\opKer(\opId-\Delta,R_\varphi)$. In particular, $f\in C^\infty_\oprob(\Omega)$ and so, by Proposition \ref{PropBijectivityOverSmoothRobinSpaces}, $f=0$ and injectivity follows. Choose $(u,v)\in C^\infty(\Omega)\times C^\infty(\partial \Omega)$. Let $g\in C^\infty(\Omega)$ be such that $R_\varphi(g)=v$. By Proposition \ref{PropBijectivityOverSmoothRobinSpaces} again, there exists $f\in C^\infty_\oprob(\Omega)$ such that:
$$
(\opId-\Delta)f = u - (\opId - \Delta)g.
$$
We see that $(\opId - \Delta,R_\varphi)(f+g)=(u,v)$ and surjectivity follows, completing the proof.\end{proof}
\begin{proposition}\label{PropOperatorIsInvertible}
$(\opId-\Delta,R_\varphi)$ defines an invertible, linear mapping from $C^{*,\lambda+2}(\Omega)$ into $C^{*,\lambda}(\Omega)\times C^{*,\lambda+1}(\partial \Omega)$.
\end{proposition}
\begin{proof} Denote $\lambda=k+\alpha$. Choose $f\in\opKer(\opId-\Delta,R_\varphi)$. As in the proof of Proposition \ref{PropBijectivityOverSmoothRobinSpaces}, since $f\in H^{k+2}_\oprob(\Omega)$, $f=0$, and it follows that $(\opId-\Delta,R_\varphi)$ is injective. By the global Schauder estimates for the oblique derivative problem (Theorem $6.30$ of \cite{GilbTrud}), there exists $C>0$ such that, for all $f\in C^{*,\lambda+2}(\Omega)$:
$$
\|f\|_{\lambda+2} \leqslant C(\|f\|_{L^\infty} + \|(\opId - \Delta)f\|_\lambda + \|R_\varphi (f)\|_{\lambda+1}),
$$
from which we deduce in the usual manner that the image of $(\opId-\Delta,R_\varphi)$ in $C^{*,\lambda}(\Omega)\times C^{*,\lambda+1}(\partial \Omega)$ is closed. However:
$$
C^\infty(\Omega)\times C^\infty(\partial \Omega) = (\opId - \Delta,R_\varphi)(C^\infty(\Omega))\subseteq (\opId - \Delta,R_\varphi)(C^{*,\lambda+2}(\Omega)),
$$
and since $C^\infty(\Omega)\times C^\infty(\partial \Omega)$ is a dense subset of $C^{*,\lambda}(\Omega)\times C^{*,\lambda+1}(\Omega)$, it follows that $\opId-\Delta$ is surjective. In particular, it is bijective, and the result now follows by the Closed Graph Theorem.\end{proof}

\subsection{The elliptic theory of Jacobi operators}\label{TheEllipticTheoryOfJacobiOperators} Fix $(x,[e])\in X\times\Cal{E}$. To simplify notation, we will drop the $(x,[e])$ dependence of the geometric quantities and operators for the remainder of this section.

We define $\textrm{L}:C^{*,\lambda+2}(\Sigma)\rightarrow C^{*,\lambda}(\Sigma)$ such that, for all $\varphi\in C^{*,\lambda+2}(\Sigma)$:
$$
\textrm{L}\varphi = -\varphi - (\opRic(N,N) + \|A\|^2)\varphi,
$$
so that, by Lemma \ref{PropFormulaForJacobiOperatorOfMeanCurvature}:
$$
\text{\rm J}^h = (\opId - \Delta) +\textrm{L}.
$$

\begin{proposition}\label{PropJacobiOperatorIsSelfAdjoint}
For all $\xi,\eta\in C^{*,\lambda+2}(\Sigma)$ such that $\text{\rm J}^\theta\xi=\text{\rm J}^\theta\eta=0$:
$$
\int_\Sigma \eta\text{\rm J}^h\xi\,\opdVol = \int_\Sigma \xi\text{\rm J}^h\eta\,\opdVol,
$$
\noindent where $\opdVol$ is the volume form of the metric $e^*g$.
\end{proposition}
\begin{proof}Trivially, for all $\xi,\eta\in C^{*,\lambda+2}(\Sigma)$:
$$
\int_\Sigma \eta \textrm{L}\xi\,\opdVol = \int_\Sigma \xi\textrm{L}\eta\,\opdVol,
$$
and the result now follows by Proposition \ref{PropOperatorIsSelfAdjoint}.
\end{proof}

\begin{proposition}\label{PropEllipticRegularityOfJacobiOperatorI}
For all $(x,[e])\in X\times\mathcal{E}$, and for all $\lambda\in[0,\infty[\setminus\Bbb{N}$, if $\varphi\in C^{*,\lambda+2}(\Sigma)$ and $\text{\rm J}\varphi\in C^\infty(\Sigma)\times C^\infty(\Sigma)$, then $\varphi\in C^\infty(\Sigma)$.
\end{proposition}
\begin{proof}Observe that:
$$
\left((\opId - \Delta)\varphi,\text{\rm J}^\theta\varphi\right) = \text{\rm J}\varphi - (\text{\rm L}\varphi,0) \in C^{*,\lambda+2}(\Sigma)\times C^{*,\lambda+3}(\partial\Sigma).
$$
Thus, by Proposition \ref{PropOperatorIsInvertible}, there exists $\varphi'\in C^{*,\lambda+4}(\Sigma)$ such that:
$$
\left((\opId - \Delta)\varphi',\text{\rm J}^\theta\varphi'\right) =\left((\opId - \Delta)\varphi,\text{\rm J}^\theta\varphi\right).
$$
By uniqueness, $\varphi=\varphi'$, and so $\varphi\in C^{*,\lambda+4}(\Sigma)$, and it follows by induction that $\varphi\in C^\infty(\Sigma)$, as desired.\end{proof}

\begin{proposition}\label{PropEllipticRegularityOfJacobiOperatorII}
For all $(x,[e])\in X\times\mathcal{E}$, the operator $\text{\rm J}$ defines a Fredholm map from $C^{*,\lambda+2}(\Sigma)$ to $C^{*,\lambda}(\Sigma)\times C^{*,\lambda+1}(\partial \Sigma)$ of Fredholm index zero. Moreover:
\begin{enumerate}
\item if we denote by $\opKer^{\lambda+2}(\text{\rm J})$ and $\opKer(\text{\rm J})$ the kernels of $\text{\rm J}$ in $C^{*,\lambda+2}(\Sigma)$ and $C^\infty(\Sigma)$ respectively, then:
$$
\opKer^{\lambda+2}(\text{\rm J}) = \opKer(\text{\rm J});\ \text{and}
$$
\item if we denote by $\opIm^{\lambda+2}(\text{\rm J})$ the image of $\text{\rm J}$ in $C^{*,\lambda}(\Sigma)\times C^{*,\lambda+1}(\partial\Sigma)$, then:
$$
\opIm^{\lambda+2}(\text{\rm J})^\perp = \left\{ (f,f\circ\epsilon)\ |\ f\in\opKer(\text{\rm J}) \right\},
$$
where the orthogonal complement is taken with respect to the $L^2$ inner-product of $e^*g$.
\end{enumerate}
\end{proposition}
\begin{proof}
Observe that $(\text{\rm L},0)$ maps $C^{*,\lambda+2}(\Sigma)$ into $C^{*,\lambda+2}(\Sigma)\times C^{*,\lambda+3}(\partial \Sigma)$. In particular, it defines a compact mapping from $C^{*,\lambda+2}(\Sigma)$ into $C^{*,\lambda}(\Sigma)\times C^{*,\lambda+1}(\partial \Sigma)$. However:
$$
\text{\rm J} = (\opId - \Delta,\text{\rm J}^\theta) + (\text{\rm L},0),
$$
Thus, by Proposition \ref{PropOperatorIsInvertible}, $\text{\rm J}$ defines a compact perturbation of an invertible mapping from $C^{*,\lambda+2}(\Sigma)$ to $C^{*,\lambda}(\Sigma)\times C^{*,\lambda+1}(\partial \Sigma)$ and is therefore Fredholm of index zero. Moreover, by Proposition \ref{PropEllipticRegularityOfJacobiOperatorI}:
$$
\opKer^{\lambda+2}(\text{\rm J}) \subseteq \opKer(\text{\rm J}).
$$
Since the reverse inclusion is trivial, these two spaces therefore coincide, and $(1)$ follows.

Denote by $\langle\cdot,\cdot\rangle$ the $L^2$ inner-product of $e^*g$. Bearing in mind Stokes' Theorem, for all $\varphi\in C^{*,\lambda+2}(\Sigma)$ and for all $\psi\in\opKer(\text{\rm J})$:
\begin{align*}
\langle \text{\rm J} \varphi,(\psi,\psi\circ\epsilon)\rangle
&= \int_\Sigma \psi\text{\rm J}^h\varphi\,\opdVol + \int_{\partial \Sigma} \psi\text{\rm J}^\theta\varphi\,\opdVol\\
&= \int_\Sigma \varphi \,\text{\rm J}^h\psi\,\opdVol + \int_{\partial \Sigma} \varphi \,\text{\rm J}^\theta\psi\,\opdVol\\
&= 0.
\end{align*}
It follows that $
\left\{(f,f\circ\epsilon)\ |\ f\in\opKer(\text{\rm J})\right\} \subseteq \opIm^{\lambda+2}(\text{\rm J})^\perp.$
However, since $\text{\rm J}$ is Fredholm of index zero, the dimension of the orthogonal complement of $\opIm^{\lambda+2}(\text{\rm J})$ cannot exceed that of $\opKer(\text{\rm J})$. Thus:
$$
\left\{(f,f\circ\epsilon)\ |\ f\in\opKer(\text{\rm J})\right\} = \opIm^{\lambda+2}(\text{\rm J})^\perp,
$$
and this completes the proof.\end{proof}

%%%%%%%%%%%%%%%%%%%%
% THE LOCAL THEORY %
%%%%%%%%%%%%%%%%%%%%

\section{The Local Theory}

\subsection{Local charts I: the smooth case}\label{TheLocalSolutionSpace} Let $Y$ be a compact neighbourhood in $X$. Let $e:Y\times \Sigma\rightarrow M$ be a smooth function such that, for all $y\in Y$, $e_y:=e(y,\cdot)$ is an element of $\hat{\mathcal{E}}$ with the property that $e_y(\Sigma)$ meets $\partial M$ orthogonally along $\partial \Sigma$ with respect to $g_y$. We refer to the triplet $(Y,g,e)$ simply by $Y$. The following result is useful for constructing local charts of the space of embeddings with boundary in $\partial M$:
\begin{theorem}
There exists a neighbourhood $U$ of the zero section in $TM$, and a smooth mapping $\tildeExp:U\rightarrow M$ with the following properties:
\begin{itemize}
\item[(1)] If $X_p$ is a vertical vector over the point $0_p\in TM$, then:
$$
D\tildeExp(0_p)(X_p) = X_p;
$$
\item[(2)] If $X_p\in U\minter T_p\partial M$, then:
$$
\tildeExp(X_p) \in \partial M.
$$
\end{itemize}
\end{theorem}
\begin{remark}
We henceforth refer to $\tildeExp$ as the {\it modified exponential map}.
\end{remark}
\begin{proof}
It sufficies to let $\tildeExp:U\longrightarrow M$ be the exponential map of a Riemannian metric on $M$ with respect to which $\partial M$ is totally geodesic.
\end{proof}

Let $N:Y\times \Sigma\rightarrow M$ be such that, for all $y\in Y$, $N_y:=N(y,\cdot)$ is the unit, normal vector-field over $e_y$ with respect to $g_y$ which is compatible with the orientation. Define $\hat{\Phi}_Y:Y\times C^\infty(\Sigma)\rightarrow C^\infty(\Sigma,M)$ such that, for all $y\in Y$, for all $f\in C^\infty(\Sigma)$ and for all $p\in \Sigma$:
$$
\hat{\Phi}_Y(y,f)(p) = \tildeExp(f(p)N_y(p)).
$$
\begin{proposition}\label{PropImageIsEmbeddingForSmallNorm}
There exists $r>0$ such that for all $y\in Y$, if $\|f\|_{L^\infty}<r$, then $\hat{\Phi}_Y(y,f)$ is an element of $\hat{\mathcal{E}}$.
\end{proposition}
\begin{proof}
By definition of $\tildeExp$, for all $y\in Y$, for all $f\in C^\infty(\Sigma)$ and for all $p\in \Sigma$, $\widehat{\Phi}_Y(y,f)(p)\in M$. For all $y\in Y$ and for all $p\in\partial \Sigma$, since $e_y(\Sigma)$ meets $\partial M$ orthogonally with respect to the metric $g_y$, $N_y(p)$ is tangent to $\partial M$ at $e_y(p)$. Therefore, for all $f\in C^\infty(\Sigma)$, the vector $f(p)N_y(p)$ is also tangent to $\partial M$ at $e_y(p)$, and so, by definition of $\tildeExp$, $\hat{\Phi}_Y(y,f)(p)\in\partial M$. We consider the mapping $F:Y\times \Sigma\times\Bbb{R}\rightarrow M$ given by:
$$
F(y,p,t) = \tildeExp(tN_y(p)).
$$
For all $y$, we denote $F_y:=F(y,\cdot,\cdot)$. By definition of $\tildeExp$, for all $y\in Y$ and for all $p\in \Sigma$, $DF_y$ is bijective at $(p,0)$. Since $e_y$ is an embedding for all $y\in Y$, there exists $r>0$ such that, for all $y\in Y$, the restriction of $F_y$ to $\Sigma\times(-r,r)$ is also an embedding. For $f\in C^\infty(\Sigma)$, we define $\hat{f}\in C^\infty(\Sigma,\Sigma\times\Bbb{R})$ by:
$$
\hat{f}(p) = (p,f(p)).
$$
If $\|f\|_{L^\infty}<r$, then $\hat{f}$ trivially defines an embedding of $\Sigma$ into $\Sigma\times(-r,r)$, and so, for all $y\in Y$, $\hat{\Phi}_Y(y,f)=F_y\circ\hat{f}$ defines an embedding of $\Sigma$ into $M$. We conclude that for all $y\in Y$ and for $\|f\|_{L^\infty}<r$, $\hat{\Phi}_Y(y,f)$ is an element of $\hat{\mathcal{E}}$, as desired.\end{proof}

We define $\mathcal{U}_Y\subseteq Y\times C^\infty(\Sigma)$ by:
$$
\mathcal{U}_Y = \left\{(y,f)\ |\ \|f\|_{L^\infty}<r\right\},
$$
where $r$ is as in Proposition \ref{PropImageIsEmbeddingForSmallNorm}. We define $\Phi_Y:\mathcal{U}_Y\rightarrow\mathcal{E}$ and $\Psi_Y:\mathcal{U}_Y\rightarrow Y\times\mathcal{E}$ such that for all $(y,f)\in\mathcal{U}_Y$:
$$
\Phi_Y(y,f) = [\hat{\Phi}_Y(y,f)],\qquad \Psi_Y(y,f) = (y,[\hat{\Phi}_Y(y,f)]).
$$

\begin{proposition}\label{PropGraphChartIsInjective}
$\Psi_Y$ is injective.
\end{proposition}
\begin{proof} Let $(y,f),(y',f')\in\mathcal{U}_Y$ be such that $\Psi_Y(y,f)=\Psi_Y(y',f')$. In particular, $y=y'$ and $\Phi_Y(y,f)=\Phi_Y(y',f')$. There therefore exists an orientation-preserving diffeomorphism $\alpha$ of $\Sigma$ such that $\hat{\Phi}_Y(y,f')=\hat{\Phi}_Y(y,f)\circ\alpha$. Let $r$ be as in Proposition \ref{PropImageIsEmbeddingForSmallNorm} and define $\hat{f},\hat{f}^\prime:\Sigma\longrightarrow \Sigma\times(-r,r)$ by $\hat{f}(p) = (p,f(p))$ and $\hat{f}^\prime(p) = (p,f'(p))$. Define $F_y:\Sigma\times(-r,r)\rightarrow M$ by $F_y(p,t) = \tildeExp(tN_y(p))$. By definition of $\hat{\Phi}_Y$:
$$
F_y\circ\hat{f}\circ\alpha = \hat{\Phi}_Y(y,f)\circ\alpha = \hat{\Phi}_Y(y,f^\prime) = F_y\circ\hat{f}^\prime.
$$
However, by definition of $r$, $F_y$ is an embedding, and composing the above relation with $F_y^{-1}$ yields, for all $p\in \Sigma$:
$$
(\alpha(p),(f\circ\alpha)(p)) = (\hat{f}\circ\alpha)(p) = \hat{f}^\prime(p) = (p,f'(p)).
$$
It follows that $\alpha$ coincides with the identity and $f'$ coincides with $f$, and $\Psi_Y$ is therefore injective as desired
\end{proof}

\begin{proposition}\label{PropGraphChartIsOpenMapping}
$\Psi_Y$ is an open mapping.
\end{proposition}

\begin{proof} Choose $(y,f)\in\mathcal{U}_Y$ and let $\Omega$ be a neighbourhood of $(y,f)$ in $\mathcal{U}_Y$. Denote $(y,[e])=\Psi_Y(y,f)$ and let $(y_m,[e_m])_\minn\in Y\times\mathcal{E}$ be a sequence converging to $(y,[e])$. In particular, $(y_m)_\minn$ converges to $y$. Let $r$ be as in Proposition \ref{PropImageIsEmbeddingForSmallNorm}. We define $\hat{f}:\Sigma\rightarrow \Sigma\times(-r,r)$, by $\hat{f}(p) = (p,f(p)),$ and $F:Y\times \Sigma\times (-r,r)\rightarrow M$, by $F(y,p,t) = \tildeExp(tN_y(p))$. By definition of $r$, $F_y:=F(y,\cdot,\cdot)$ is an embedding for all $y\in Y$. By definition, $[e]=[F_y\circ\hat{f}]$. Since $([e_m])_\minn$ converges to $[e]$, there exists a sequence $(\alpha_m)_\minn$ of orientation-preserving diffeomorphisms of $\Sigma$ such that $(e_m\circ\alpha_m)_\minn$ converges to $F_y\circ\hat{f}$. Bearing in mind that, in addition, $(y_m)_\minn$ converges to $y$, there exists $K\in\Bbb{N}$ such that for all $m\geqslant K$, $(e_m\circ\alpha_m)$ takes values in $F_{y_m}(\Sigma\times(-r,r))$ and that $(F_{y_m}^{-1}\circ e_m\circ\alpha_m)_{m\geqslant K}$ converges to $\hat{f}$.

Let $\pi_1:\Sigma\times\Bbb{R}\rightarrow \Sigma$ and $\pi_2:\Sigma\times\Bbb{R}\rightarrow\Bbb{R}$ be the canonical projections onto the first and second factors respectively. For all $m\geqslant K$, we denote:
$$
\beta_m = \pi_1\circ F_{y_m}^{-1}\circ e_m\circ\alpha_m,\qquad \widetilde{f}_m = \pi_2\circ F_{y_m}^{-1}\circ i_m\circ\alpha_m.
$$
Observe that $(\beta_m)_{m\geqslant K}$ converges to the identity mapping. Thus, upon increasing $K$ if necessary, we may assume that $\beta_m$ is a diffeomorphism for all $m$ and that $(\beta_m)^{-1}_{m\geqslant K}$ also converges to the identity mapping. For all $m\geqslant K$, we denote:
$$
f_m=\tilde{f}_m\circ\beta_m^{-1}.
$$
Since $(\tilde{f}_m)_{m\geqslant M}$ converges to $f$, so too does $(f_m)_\minn$. In particular, upon increasing $K$ further if necessary, we may assume that $(y_m,f_m)\in\Omega$ for all $m$. However, for all $m$, $
e_m\circ\alpha_m\circ\beta_m^{-1} = \hat{\Phi}_Y(y_m,f_m)$. In other words:
$$
(y_m,[e_m]) = (y_m,\Phi(Y)(y_m,f_m)) = \Psi(Y)(y_m,f_m).
$$
It follows that $(y_m,[e_m])\in\Psi_Y(\Omega)$ for all $m\geqslant K$, and we conclude that $\Psi_Y$ is an open mapping as desired.
\end{proof}

We denote the image $\Psi_Y(\mathcal{U}_Y) $ in $Y\times\mathcal{E}$ by  $\mathcal{V}_Y$.
By Proposition \ref{PropGraphChartIsOpenMapping}, $\mathcal{V}_Y$ is an open subset of $Y\times\mathcal{E}$. By Proposition \ref{PropGraphChartIsInjective}, $\Psi_Y$ defines a bijective mapping from $\mathcal{U}_Y$ into $\mathcal{V}_Y$, and by Proposition \ref{PropGraphChartIsOpenMapping} again, this mapping is a homeomorphism. We thus refer to the triplet $(\Psi_Y,\mathcal{U}_Y,\mathcal{V}_Y)$ as the {\it graph chart} of $X\times\mathcal{E}$ {\it over} $Y$. When only $e_0:=e(x_0)$ is a-priori given, we refer to the triplet $(\Psi_Y,\mathcal{U}_Y,\mathcal{V}_Y)$ as a {\it graph chart} of $X\times\mathcal{E}$ {\it about} $(x_0,e_0)$.

We define the {\it mean curvature function} $H_Y:\mathcal{U}_Y\rightarrow C^\infty(\Sigma)$ and the {\it boundary angle function} $\Theta_Y:\mathcal{U}_Y\rightarrow C^\infty(\partial \Sigma)$ such that, for all $(y,f)\in\mathcal{U}_Y$:
$$
H_Y(y,f)= H_{f,\hat{\Phi}_Y(y,f)},\qquad \Theta_Y(y,f) = \Theta_{f,\hat{\Phi}_Y(y,f)}.
$$
We define $\mathcal{Z}_{Y,\loc}\subseteq\mathcal{U}_Y$ by:
$$
\mathcal{Z}_{Y,\loc} = \left\{ (y,f)\ |\ H_Y(y,f)=0,\Theta_Y(y,f)=0\right\},
$$
and we call $\mathcal{Z}_{Y,\loc}$ the {\it local solution space} in the graph chart. Observe in particular that:
$$
\mathcal{Z}_{Y,\loc} = \Psi_Y^{-1}(\mathcal{Z}(Y)\minter\mathcal{V}_Y).
$$
In later sections, where no ambiguity arises, we will often suppress $Y$ and simply write $\hat{\Phi}$, $\Phi$, $\Psi$ and so on respectively for $\hat{\Phi}_Y$, $\Phi_Y$ and $\Psi_Y$ and so on.

\subsection{Local charts II: The H\"older case}\label{BanachSpaces}We consider families of H\"older spaces parametrised by $\lambda\in[0,\infty)\setminus\Bbb{N}$ (c.f. Section \ref{SubHeadHoelderSpaces}).  Let $r>0$ be as in Proposition \ref{PropImageIsEmbeddingForSmallNorm} and define $\mathcal{U}^{\lambda+1}_Y\subseteq Y\times C^{*,\lambda+1}(\Sigma)$ by:
$$
\mathcal{U}^{\lambda+1}_Y = \left\{(y,f)\ |\ \|f\|_{L^\infty} < r\right\}.
$$
We denote by $\hat{\mathcal{E}}^{\lambda+1}$ the space of all $C^{*,\lambda+1}$ embeddings $e:\Sigma\rightarrow M$ with the properties that $e(\partial \Sigma)\subseteq\partial M$ and $e(\partial\Sigma)=e(\Sigma)\cap\partial M$, and we define $\hat{\Phi}^{\lambda+1}:\mathcal{U}^{\lambda+1}_Y\longrightarrow\hat{\mathcal{E}}^{\lambda+1}$ such that for all $(x,f)\in\mathcal{U}^{\lambda+1}_Y$ and for all $p\in \Sigma$:
$$
\hat{\Phi}^{\lambda+1}_Y(y,f)(p) = \tildeExp(f(p)N_y(p)).
$$
We define the {\it mean curvature function} $H^{\lambda+2}_Y:\mathcal{U}^{\lambda+2}_Y\rightarrow C^{*,\lambda}(\Sigma)$ such that, for all $(y,f)\in\mathcal{U}^{\lambda+2}_Y$:
$$
H^{\lambda+2}_Y(y,f) = H_{y,\hat{\Phi}^{\lambda+2}_Y(y,f)}.
$$

We recall that any function that may be constructed via a finite combination of addition, multiplication, differentiation and post-composition by smooth functions defines a smooth function of Banach spaces. It follows in particular that $H^{\lambda+2}_Y$ defines a smooth function between two Banach spaces. For each $k$, we denote by $D_kH^{\lambda+2}_Y$ the partial derivative of $H^{\lambda+2}_Y$ with respect to the $k$'th component in $\mathcal{U}^{\lambda+2}_Y\subseteq Y\times C^{*,\lambda+2}(\Sigma)$. Observe, in particular, that by definition of the Jacobi operator of mean curvature:
\begin{equation}
D_2H^{\lambda+2}_Y(x,0)=\text{\rm J}^h_{x,e}.
\end{equation}

We define the {\it boundary angle function} $\Theta^{\lambda+1}_Y:\mathcal{U}^{\lambda+1}_Y\rightarrow C^{*,\lambda}(\partial \Sigma)$ such that for all $(y,f)\in\mathcal{U}^{\lambda+1}_Y$:
$$
\Theta^{\lambda+1}_Y(y,f) = \Theta_{y,\hat{\Phi}^{\lambda+1}_Y(y,f)}.
$$
Observe that $\Theta^{\lambda+1}_Y$ also defines a smooth function between two Banach spaces. For each $k$, we denote by $D_k\Theta^{\lambda+1}_Y$ the partial derivative of $\Theta^{\lambda+1}_Y$ with respect to the $k$'th component in $\mathcal{U}^{\lambda+1}_Y\subseteq Y\times C^{*,\lambda+1}(\Sigma)$. Observe, in particular, that by definition of the Jacobi operator of the boundary angle:
\begin{equation}
D_2\Theta^{\lambda+1}_Y(x,0) = \text{\rm J}^\theta_{x,e}.
\end{equation}
Finally, we denote $\mathcal{Z}^{\lambda+2}_{Y,\loc}\subseteq\mathcal{U}^{\lambda+2}_Y$ by:
$$
\mathcal{Z}^{\lambda+2}_{Y,\loc} = \left\{ (y,f)\ |\ H^{\lambda+2}_Y(y,f)=0,\ \Theta^{\lambda+2}_Y(y,f)=0\right\}.
$$

We recall the following classical result concerning the regularity of embeddings of prescribed mean curvature:

\begin{theorem}\label{ThmEllipticRegularityOfPrescribedMeanCurvature}Let $g$ be a smooth metric over $M$, let $h:M\rightarrow\Bbb{R}$ be a smooth function, and let $\Sigma\subseteq M$ be an embedded compact submanifold of $M$ of class $C^{\lambda+2}$ such that $\partial\Sigma\subseteq\partial M$ and $\partial\Sigma=\Sigma\cap\partial M$. If $\Sigma$ meets $\partial M$ orthogonally along $\partial\Sigma$ with respect to the metric $g$ and if the mean curvature of $\Sigma$ is at every point $p\in\Sigma$ equal to $h(p)$, then $\Sigma$ is smooth.
\end{theorem}

\begin{proof} This follows by applying, for example, Schauder estimates \cite{GilbTrud}.
\end{proof}

Expressed in terms of graph charts, this yields:

\begin{proposition}\label{PropEllipticRegularityOfPMCInGraphCharts}
If $(y,f)\in\mathcal{U}^{\lambda+2}_Y$ is such that $H^{\lambda+2}_Y(y,f)\in C^\infty(\Sigma)$, then $f\in C^\infty(\Sigma)$.
\end{proposition}

\begin{proof} Denote $h=H^{\lambda+2}_Y(y,f)$. Let $r$ be as in Proposition \ref{PropImageIsEmbeddingForSmallNorm}. We define $\hat{f}:\Sigma\rightarrow \Sigma\times(-r,r)$ such that for all $p\in \Sigma$, $\hat{f}(p)=(p,f(p))$. We define the mapping $F_y:\Sigma\times(-r,r)\rightarrow M$ such that for all $p\in \Sigma$ and for all $t\in(-r,r)$:
$$
F_y(p,t) = \tildeExp(tN_y(p)).
$$
Recall that $F_y$ is a diffeomorphism onto its image. Observe that $\tilde{\Phi}_Y(y,f)=F_y\circ\hat{f}$. In particular, $\hat{f}(\Sigma)$ is a $C^{\lambda+2}$ embedded submanifold of $\Sigma\times(-r,r)$ such that $\hat{f}(\partial \Sigma)$ meets $\partial \Sigma\times(-r,r)$ orthogonally along $\partial \Sigma$ with respect to the metric $F_y^*g_y$. We define $\tilde{h}:\Sigma\times(-r,r)\rightarrow\Bbb{R}$ by $\tilde{h}(p,r) = h(p)$.
Observe that for all $p\in \Sigma$, the mean curvature of $\hat{f}(\Sigma)$ at $\hat{f}(p)$ is equal to $h(p)=(\tilde{h}\circ\hat{f})(p)$. It follows from Theorem \ref{ThmEllipticRegularityOfPrescribedMeanCurvature} that $\widehat{f}(\Sigma)$ is smooth and since $\hat{f}(\Sigma)$ is the graph of $f$, $f$ is therefore also smooth, as desired.
\end{proof}
\begin{proposition}
For all $\lambda\in[0,\infty[\setminus\Bbb{N}$:
$$
\mathcal{Z}^{\lambda+2}_{Y,\loc} = \mathcal{Z}_{Y,\loc}.
$$
\end{proposition}
\begin{proof} Choose $\lambda\in[0,\infty[\setminus\Bbb{N}$. Choose $(y,f)\in\mathcal{Z}^{\lambda+2}_{Y,\loc}$ and denote $e'=\tilde{\Phi}^{\lambda+2}_Y(y,f)$. By definition of $H$ and $\Theta$, $e'$ is free boundary minimal with respect to $g_y$. By Proposition \ref{PropEllipticRegularityOfPMCInGraphCharts}, $f$ is smooth, and so $(y,f)\in\mathcal{Z}_{Y,\loc}$, from which it follows that:
$$
\mathcal{Z}^{\lambda+2}_{Y,\loc}\subseteq\mathcal{Z}_{Y,\loc}.
$$
The converse inclusion is trivial, and the result follows.\end{proof}

\subsection{Conjugations}\label{Conjugations} We finish this section by describing the relationship between functionals $H$ and $\Theta$ and the perturbation and Jacobi operators introduced in Sections \ref{JacobiOperators} and \ref{PerturbationOperators}. To this end, let $(y,f)\in\mathcal{Z}_{Y,\loc}$ and denote $e'=\hat{\Phi}_Y(y,f)$. We define the vector field $V_{y,f}$ over $e'$ such that, for all $p\in \Sigma$:
$$
V_{y,f}(p) = \partial_t\hat{\Phi}_Y(y,f+t)(p)|_{t=0}.
$$
We define the function $\lambda_{y,f}:\Sigma\rightarrow\Bbb{R}$ by:
$$
\lambda_{y,f}=g_y(N_{y,e'},V_{y,f}).
$$
Observe that for all $(y,f)\in\mathcal{Z}_{Y,\loc}$, both $V_{y,f}$ and $\lambda_{y,f}$ are smooth and, moreover, $V_{y,f}$ is at no point tangent to $e'(\Sigma)$, from which it follows that $\lambda_{y,f}$ never vanishes. The next proposition follows immediately from the definition of $\text{\rm P}$:

\begin{proposition}\label{PropPertMCIsConjugate}
For all $(y,f)\in\mathcal{Z}_{Y,\loc}=\mathcal{Z}^{\lambda+2}_{Y,\loc}$, and all $\xi_y\in T_yY$:
$$D_1H^{\lambda+2}_Y(y,f)(\xi_y) = \text{\rm P}^h_{y,e'}(\xi_y),\,\,\,\text{and}\,\,\,
D_1\Theta^{\lambda+1}_Y(y,f)(\xi_y)= \text{\rm P}^\theta_{y,e'}(\xi_y).
$$\end{proposition}
%\begin{proof} This follows from the definition of $\text{\rm P}^h_{y,e'}$.
%\end{proof}
%\begin{proposition}\label{PropPertBAIsConjugate}For all $(y,f)\in\mathcal{Z}_{Y,\loc}=\mathcal{Z}^{\lambda+1}_{Y,\loc}$, and all $\xi_y\in T_yY$:
%$$
%D_1\Theta^{\lambda+1}_Y(y,f)(\xi_y) = \text{\rm P}^\theta_{y,e'}(\xi_y).
%$$
%\end{proposition}
%\begin{proof} This follows from the definition of $\textrm{P}^\theta_{y,e'}$. \end{proof}

For the partial derivatives with respect to the second component:

\begin{proposition}\label{PropInfMCIsConjugate}
\noindent For all $(y,f)\in\mathcal{Z}_{Y,\loc}=\mathcal{Z}^{\lambda+2}_{Y,\loc}$, and all $\varphi\in C^{*,\lambda+2}(\Sigma)$:
$$
D_2H^{\lambda+2}_Y(y,f)(\varphi) = \text{\rm J}^h_{y,e'}(\lambda_{y,f}\varphi).
$$
\end{proposition}

\begin{proof} Denote $e'=\hat{\Phi}_Y(y,f)$. Let $Y'$ be a compact neighbourhood of $y$ in $Y$ and let $(\Psi_{Y^\prime},\mathcal{U}_{Y^\prime},\mathcal{V}_{Y^\prime})$ be a graph chart of $X\times \mathcal{E}$ about $(y,e')$ over $Y'$. Choose $\varphi\in C^\infty(\Sigma)$. There exists $\delta>0$ and smooth mappings $\alpha:(-\delta,\delta)\times \Sigma\rightarrow \Sigma$ and $\psi:(-\delta,\delta)\times \Sigma\rightarrow\Bbb{R}$ such that $\alpha(0,\cdot)$ coincides with the identity mapping, for all $t\in(-\delta,\delta)$, $\alpha_t:=\alpha(t,\cdot)$ is a smooth diffeomorphism of $\Sigma$ and
$\hat{\Phi}_{Y^\prime}(y,\psi_t)\circ\alpha_t = \hat{\Phi}_Y(y,f+t\varphi)$,
where $\psi_t:=\psi(t,\cdot)$. Observe that, by injectivity of $\Psi_{Y^\prime}$, $\psi_0=0$. Bearing in mind the definition of $V_{y,f}$, differentiating with respect to $t$ yields
$D_2\hat{\Phi}_Y(y,f)(\varphi) = \varphi V_{y,f}$. Likewise:
$$
D_2\tilde{\Phi}_{Y^\prime}(y,0)((\partial_t\psi)_0) = (\partial_t\psi)_0N_{y,e'}.
$$
By the chain rule, this yields $\varphi V_{y,f} = (\partial_t\psi)_0N_{y,e'} +W$,
where $W$ is tangent to $e(\Sigma)$. Taking the inner product with $N_{y,e'}$ therefore yields
$(\partial_t\psi)_0 = \varphi g_y(N_{y,e'},V_{y,f}) = \lambda_{y,f}\varphi$.
Let $H_{Y^\prime}$ be the mean curvature function in the chart $(\Psi_{Y^\prime},\mathcal{U}_{Y^\prime},\mathcal{V}_{Y^\prime})$. Observe that, for all $t$:
$$
H_{Y^\prime}(y,\psi_t)\circ\alpha_t = H_Y(x,f+t\varphi).
$$
Observe, moreover, that since $(y,f)\in\mathcal{Z}(Y)$, $H_{Y^\prime}(y,0)=H_Y(y,f)=0$. Differentiating the above relation at $t=0$ therefore yields:
\begin{eqnarray*}
D_2H_Y(x,f)(\varphi) & =&  D_2H_{Y^\prime}(y,0)((\partial_t\psi)_0)\\
&=& D_2H_{Y^\prime}(y,0)(\lambda_{y,f}\varphi)\\
&=& \text{\rm J}^h_{y,e'}(\lambda_{y,f}\varphi).
\end{eqnarray*}
Since $C^\infty(\Sigma)$ is a dense subset of $C^{*,\lambda+2}(\Sigma)$, %it follows by continuity that for all $\varphi\in C^{*,\lambda+2}(\Sigma)$:$$D_2H^{\lambda+2}_Y(x,f)(\varphi) = \text{\rm J}^h_{y,e'}(\lambda_{y,f}\varphi),$$
the result follows by continuity.
\end{proof}
\begin{proposition}\label{InfBCIsConjugate}
For all $(y,f)\in\mathcal{Z}_{Y,\loc}=\mathcal{Z}^{\lambda+1}_{Y,\loc}$, and for all $\varphi\in C^{*,\lambda+1}(\Sigma)$:
$$
D_2\Theta^{\lambda+1}_Y(y,f)(\varphi) = \text{\rm J}^\theta_{y,e'}(\lambda_{y,f}\varphi).
$$
\end{proposition}
\begin{proof} Choose $\varphi\in C^\infty(\Sigma)$. We use the same construction as in the proof of Proposition \ref{PropInfMCIsConjugate}. Let $\Theta_{Y^\prime}$ be the boundary angle function in the chart generated by $Y'$. Observe that, for all $t$, $\Theta_{Y^\prime}(y,\psi_t)\circ\alpha_t = \Theta_Y(x,f+t\varphi)$.
Observe, moreover, that since $(y,f)\in\mathcal{Z}(\Sigma)$, $\Theta_{Y^\prime}(y,0)=\Theta_Y(y,f)=0$. Differentiating this relation at $t=0$ therefore yields:
\begin{eqnarray*}
D_2\Theta_Y(x,f)(\varphi) & =&  D_2\Theta_{Y^\prime}(y,0)((\partial_t\psi)_0)\\ &=& D_2\Theta_{Y^\prime}(y,0)(\lambda_{y,f}\varphi) \\
&=& \text{\rm J}^\theta_{y,e'}(\lambda_{y,f}\varphi).
\end{eqnarray*}
Once again, since $C^\infty(\Sigma)$ is a dense subset of $C^{*,\lambda+1}(\Sigma)$, %it follows by continuity that for all $\varphi\in C^{*,\lambda+1}(\Sigma)$:$$D_2\Theta^{\lambda+1}_Y(x,f)(\varphi) = \text{\rm J}^\theta_{y,e'}(\lambda_{y,f}\varphi),$$as desired.
the result follows by continuity.
\end{proof}

%%%%%%%%%%%%%%%%%%%%%%%%%%%%%%%%%%%%%%%%%%%%%%%%%%%%
% THE DIFFERENTIAL STRUCTURE OF THE SOLUTION SPACE %
%%%%%%%%%%%%%%%%%%%%%%%%%%%%%%%%%%%%%%%%%%%%%%%%%%%%

\section{The Differential Structure of the Solution Space}
\subsection{Extensions and surjectivity}\label{SubheadExtensionsAndSurjectivity}
Let $\widetilde{X}$ be another smooth, compact, finite-dimensional manifold. Let $\widetilde{g}:\widetilde{X}\times M\rightarrow\opSym^+(TM)$ be a smooth function such that for all $x\in\widetilde{X}$, the metric $\widetilde{g}_x:=\widetilde{g}(x,\cdot)$ is admissable.

We say that $\widetilde{X}$ is an {\it extension} of $X$ whenever $X\subseteq\widetilde{X}$, and the restriction of $\widetilde{g}$ to $X$ coincides with $g$. In this section, we show the smoothness of the solution space $\Cal{Z}(\widetilde{X})$ for a suitable extension $\widetilde{X}$ of $X$. Upon furnishing $\widetilde{X}$ with a canonical orientation, we then define a canonical orientation of $\Cal{Z}(\widetilde{X})$. In particular, this yields a canonical ${\mathbb Z}$-valued mapping degree of $\Pi:\Cal{Z}(\widetilde{X})\rightarrow\tilde{X}$ which we denote by $\opDeg(\Pi)$. We will see in Sections \ref{NonDegenerateFamilies} and \ref{CalculatingTheDegree} that it is also useful to define a local degree. We therefore denote for any open subset $\Omega\subseteq\Cal{E}$:
$$
\Cal{Z}(X|\Omega) := \Cal{Z}(X)\minter(X\times\Omega),\qquad \partial_\omega\Cal{Z}(X|\Omega) :=\Cal{Z}(X)\minter(X\times\partial\Omega),
$$
Since $\Cal{Z}(X|\Omega)$ is an open subset of $\Cal{Z}(X)$, we see that $\Cal{Z}(\widetilde{X}|\Omega)$ is also smooth for a suitable extension $\widetilde{X}$ of $X$. If, in addition, $\partial_\omega\Cal{Z}(X|\Omega)=\emptyset$, then we may suppose also that $\partial_\omega\Cal{Z}(\widetilde{X}|\Omega)=\emptyset$, and, upon furnishing $\tilde{X}$ with an orientation form, we obtain as before a ${\mathbb Z}$-valued mapping degree of $\Pi:\Cal{Z}(\widetilde{X}|\Omega)\rightarrow\widetilde{X}$, which we denote by $\opDeg(\Pi|\Omega)$. We recall from Section \ref{Conjugations} that $P_{x,e}+J_{x,e}$ is conjugate to the derivative of $(H,\Theta)$ in any graph chart about $(x,e)$.

\begin{proposition}\label{PropSurjectivityIsAnOpenProperty}
If $\text{\rm P}_{x,e}+\text{\rm J}_{x,e}$ is surjective at $(x,[e])\in\Cal{Z}(X)$, then there exists a neighbourhood $W_{x,e}$ of $(x,[e])$ in $\Cal{Z}(X)$ such that $\text{\rm P}_{x',e'}+\text{\rm J}_{x',e'}$ is surjective for all $(x',[e'])\in W$.
\end{proposition}

\begin{proof}Suppose the contrary. There exists $(x,[e])\in\Cal{Z}(X)$ such that $\text{\rm P}_{x,e}+\text{\rm J}_{x,e}$ is surjective and a sequence $(x_m,[e_m])_\minn\in\Cal{Z}(X)$ which converges to $(x,[e])$ such that $\text{\rm P}_{x_m,e_m}+\text{\rm J}_{x_m,e_m}$ is not surjective. Choose $\lambda\in[0,\infty[\setminus\Bbb{N}$. By Proposition \ref{PropEllipticRegularityOfJacobiOperatorII}, $\text{\rm P}_{x,e}+\text{\rm J}_{x,e}$ defines a surjective, Fredholm map from $T_xX\times C^{*,\lambda+2}(\Sigma)$ into $C^{*,\lambda}(\Sigma)\times C^{*,\lambda+1}(\partial \Sigma)$. Observe that $(\text{\rm P}_{x_m,e_m},\text{\rm J}_{x_m,e_m})_\minn$ converges to $\text{\rm P}_{x,e}+\text{\rm J}_{x,e}$ in the operator norm. Since the property of being a surjective, Fredholm map is open, there exists $M\in\Bbb{N}$ such that for all $m\geqslant M$, $\text{\rm P}_{x_m,e_m}+\text{\rm J}_{x_m,e_m}$ also defines a surjective map from $T_{x_m}X\times C^{*,\lambda+2}(\Sigma)$ into $C^{*,\lambda}(\Sigma)\times C^{*,\lambda+1}(\partial \Sigma)$. By Propositions \ref{PropEllipticRegularityOfJacobiOperatorI} and \ref{PropEllipticRegularityOfJacobiOperatorII}, it follows that for all $m\geqslant M$, $\text{\rm P}_{x_m,e_m}+\text{\rm J}_{x_m,e_m}$ defines a surjective map from $T_{x_m}X\times C^\infty(\Sigma)$ into $C^\infty(\Sigma)\minter C^\infty(\partial \Sigma)$, and this completes the proof.
\end{proof}
\begin{theorem}\label{PropSurjectivityForExtensions}
For every open set $\Omega\subseteq\Cal{E}$ such that $\partial_\omega\Cal{Z}(X|\Omega)=\emptyset$, there exists an extension $\widetilde{X}$ of $X$ such that $\partial_\omega\Cal{Z}(\tilde{X}|\Omega)=\emptyset$ and, for all $(x,[e])\in\Cal{Z}(\tilde{X}|\Omega)$, the operator $\text{\rm P}_{x,e}+\text{\rm J}_{x,e}$ defines a surjective mapping from $T_x\widetilde{X}\times C^\infty(\Sigma)$ into $C^\infty(\Sigma)\times C^\infty(\partial \Sigma)$.
\end{theorem}
\begin{proof}
We define the mapping $\widetilde{g}:C^\infty(M)\times X\times M\rightarrow\opSym^+(TM)$ such that, for all $f\in C^\infty(M)$ and for all $x\in X$:
$$
\widetilde{g}_{f,x}:=\widetilde{g}(f,x,\cdot)=e^f g_x.
$$
Let $E$ be a finite-dimensional, linear subspace of $C^\infty(M)$ and for $r>0$, let $E_r$ be the closed ball of radius $r$ about $0$ in $E$ with respect to some metric. Observe that for sufficiently small $r$, and for all $(f,x)\in E_r\times X$, the metric $\tilde{g}_{f,x}$ is also admissable. We denote $\widetilde{X}:=E_r\times X$, and we will show that $\widetilde{X}$ has the desired properties for suitable choices of $E$ and $r$.

Choose $(x,[e])\in\Cal{Z}(X|\Omega)$. We claim that there exists a finite-dimensional subspace $E_{x,e}\subseteq C^\infty(M)$ with the property that if $E$ contains $E_{x,e}$, then:
$$
C^{\infty}(\Sigma)\times C^{\infty}(\partial \Sigma) = \opIm(\text{\rm P}_{(0,x),e}) + \opIm(\text{\rm J}_{(0,x),e}).
$$
Indeed, let $f_1,...,f_m$ be a basis of $\opKer(\text{\rm J}_{(0,x),e})$. Let $U$ be an open subset of $M$ intersecting $e(\Sigma)$ non-trivially, let $\varphi_1,...,\varphi_m$ be as in Proposition \ref{PropFindingTheRightPerturbationSpace}, and let $E_{x,e}\subseteq C^\infty(M)$ be the linear span of these functions. For $1\leqslant k\leqslant m$, we think of $\varphi_k$ as a tangent vector to $E_{x,e}$ at $0$ and we denote $\psi_k=\textrm{P}^h_{(0,x),e}(\varphi_k)$. For all $1\leqslant k\leqslant m$, by Proposition \ref{PropVariationOfBoundaryAngle}, $\textrm{P}^\theta_{(0,x),e}(\varphi_k)=0$ and so $\text{\rm P}_{(0,x),e}(\varphi_k)=(\psi_k,0)$. We denote by $F_{x,e}$ the linear span of $(\psi_1,0),...,(\psi_m,0)$ in $C^{\infty}(\Sigma)\times C^{\infty}(\partial \Sigma)$, and we claim that:
$$
C^{\infty}(\Sigma)\times C^{\infty}(\partial \Sigma) \subseteq F_{x,e} + \opIm(\text{\rm J}_{(0,x),e}).
$$
Indeed, let $\pi$ be the orthogonal projection from $C^{\infty}(\Sigma)\times C^{\infty}(\partial \Sigma)$ onto $\opIm(\text{\rm J}_{(0,x),e})$ with respect to the $L^2$ inner-product of $e^*g_x$ and denote $\pi^\perp = \opId - \pi$. By Proposition \ref{PropEllipticRegularityOfJacobiOperatorII}, $\opIm(\pi^\perp)$ is spanned by $(f_q,f_q\circ\epsilon)_{1\leqslant q\leqslant m}$, where $\epsilon:\partial \Sigma\rightarrow \Sigma$ is the canonical embedding. However, denoting by $\opdVol_{x,e}$ the volume form of $e^*g_x$, and bearing in mind the definition of $\psi_p$, for all $1\leqslant p,q\leqslant m$:
$$
\langle \pi^\perp(\psi_p,0),(f_q,f_q\circ\epsilon)\rangle = \langle (\psi_p,0),(f_q,f_q\circ\epsilon)\rangle = \int_\Sigma \psi_p f_q\,\opdVol_{x,e} = \delta_{pq}.
$$
The restriction of $\pi^\perp$ to $F_{x,e}$ therefore defines a linear isomorphism onto $\opIm(\pi^\perp)$, and so:
$$
F_{x,e}\minter \opIm(\text{\rm J}_{(0,x),e}) = F_{x,e}\minter\opKer(\pi^\perp) = \left\{0\right\}.
$$
Since the dimension of $F_{x,e}$ is equal to the codimension of $\opIm(\text{\rm J}_{(0,x),e})$ in $C^\infty(\Sigma)\times C^\infty(\partial \Sigma)$, it follows that $F_{x,e}$ and $\opIm(\text{\rm J}_{(0,x),e})$ are complementary subspaces so that:
$$
C^\infty(\Sigma)\times C^\infty(\partial \Sigma) \subseteq F_{x,e} \oplus \opIm(\text{\rm J}_{(0,x),e}),
$$
as asserted. In particular, if $E$ contains $E_{x,e}$, then $\text{\rm J}_{(0,x),e}+\text{\rm P}_{(0,x),e}$ is surjective.

We now conclude using compactness. By Proposition \ref{PropSurjectivityIsAnOpenProperty}, there exists a neighbourhood $\tilde{W}_{x,e}$ of $((0,x),[e])$ in $\Cal{Z}(E_{x,e,r}\times X|\Omega)$ such that for all $((f,x),[e'])\in\tilde{W}_{x,e}$, $\text{\rm P}_{(f,x'),e'}+\text{\rm J}_{(f,x'),e'}$ defines a surjective map from $T_{(f,x')}(E_{x,e,r}\times X)\times C^\infty(\Sigma)$ onto $C^\infty(\Sigma)\times C^\infty(\partial \Sigma)$. We consider $\Cal{Z}(X|\Omega)$ as a subset of $\Cal{Z}(E_{x,e,r}\times X)$ and we denote $W_{x,e}=\tilde{W}_{x,e}\minter\Cal{Z}(X|\Omega)$. Thus, if $E$ contains $E_{x,e}$, then for all $(x',[e'])\in W_{x,e}$, $\text{\rm P}_{(0,x'),e'}+\text{\rm J}_{(0,x'),e'}$ defines a surjective mapping from $T_{(0,x')}(E_{r}\times X)\times C^\infty(\Sigma)$ into $C^\infty(\Sigma)\times C^\infty(\partial \Sigma)$.

Since $\partial_\omega\Cal{Z}(X|\Omega)=\emptyset$, $\Cal{Z}(X|\Omega)$ is a closed subset of $\Cal{Z}(X)$. By Proposition \ref{ThmCompactness}, $\Cal{Z}(X)$ is compact and therefore so too is $\Cal{Z}(X|\Omega)$. There therefore exist finitely many points $(x_k,[e_k])_{1\leqslant k\leqslant m}$ such that:
$$
\Cal{Z}(X|\Omega)\subseteq\munion_{k=1}^m W_{x_k,e_k}.
$$
We define $E=E_{x_1,e_1}+...+E_{x_m,e_m}$ and we see that for all $(x,[e])\in\Cal{Z}(X|\Omega)$, $\text{\rm P}_{(0,x),e}+\text{\rm J}_{(0,x),e}$ defines a surjective mapping from $T_{(0,x)}(E\times X)\times C^\infty(\Sigma)$ into $C^\infty(\Sigma)\times C^\infty(\partial \Sigma)$. Finally, by compactness again, for sufficiently small $r$ we have
$\partial_\omega\Cal{Z}(\widetilde{X}|\Omega) = \Cal{Z}(E_r\times X)\minter(E_r\times X\times\partial\Omega) = \emptyset,$
and since being a surjective Fredholm map is an open property, we may also suppose that $\text{\rm P}_{x,e}+\text{\rm J}_{x,e}$ defines a surjective map from $T_x\tilde{X}\times C^\infty(\Sigma)$ into $C^\infty(\Sigma)\times C^\infty(\partial\Sigma)$ for all $(x,[e])\in\Cal{Z}(\widetilde{X}|\Omega)$, and this completes the proof.\end{proof}

\subsection{Surjectivity and smoothness}\label{Smoothness}
\begin{proposition}\label{PropLocalSolutionSpacesAreSmooth}
Let $\Omega\subseteq\Cal{E}$ be such that $\partial_\omega\Cal{Z}(X|\Omega) = \emptyset$. If $\text{\rm P}_{x,e}+\text{\rm J}_{x,e}$ is surjective for all $(x,[e])\in\Cal{Z}(X|\Omega)$, then for every compact neighbourhood $Y$ of $X$ and for every graph chart $(\Psi,\Cal{U},\Cal{V})$ of $X\times\Cal{E}$ over $Y$, $\Cal{Z}_{\oploc}\minter\Psi^{-1}(X\times\Omega)=\Cal{Z}_\oploc^{\lambda+2}\minter\Psi^{-1}(X\times\Omega)$ is a smooth, embedded submanifold of $\Cal{U}^{\lambda+2}$ with smooth boundary and of finite dimension equal to $\opDim(X)$. Moreover:
\begin{enumerate}
\item the differential structure induced over $\Cal{Z}_\oploc\minter\Psi^{-1}(X\times\Omega)$ by the canonical embedding into $\Cal{U}^{\lambda+2}$ is independent of $\lambda$; and
\item $\Pi$ defines a smooth mapping from $\Cal{Z}_\oploc\minter\Psi^{-1}(X\times\Omega)$ into $Y$ with the property that $\Pi(\partial\Cal{Z}_\oploc)\subseteq\partial Y$.
\end{enumerate}
\end{proposition}
\begin{proof}Choose $(x,[e])\in\Cal{Z}(X|\Omega)$. By hypothesis, $\text{\rm P}_{x,e}+\text{\rm J}_{x,e}$ defines a surjective map from $T_x X\times C^\infty(\Sigma)$ into $C^\infty(\Sigma)\times C^\infty(\partial \Sigma)$. Choose $\lambda\in[0,\infty[\setminus\Bbb{N}$. By Proposition \ref{PropEllipticRegularityOfJacobiOperatorII}, $\text{\rm P}_{x,e}+\text{\rm J}_{x,e}$ defines a surjective map from $T_x X\times C^{*,\lambda+2}(\Sigma)$ into $C^{*,\lambda}(\Sigma)\times C^{*,\lambda+1}(\partial \Sigma)$. We now claim that $\text{\rm P}_{x,e}+\text{\rm J}_{x,e}$ is Fredholm of index $\opDim(X)$. Indeed, let $\pi_1:T_xX\times C^{*,\lambda+2}(\Sigma)\rightarrow T_xX$ and $\pi_2:T_xX\times C^{*,\lambda+2}(\Sigma)\rightarrow C^{*,\lambda+2}(\Sigma)$ be the projections onto the first and second factors respectively. Observe that $\pi_2$ is Fredholm of index $\opDim(X)$. Since the composite of two Fredholm maps is Fredholm of index equal to the sum of the indices of each component, it follows that $\text{\rm J}_{x,e}\circ\pi_2$ is Fredholm of index $\opDim(X)$. Observe that $\pi_1$ is compact. Since the composite of a compact mapping and any other mapping is also compact, it follows that $\text{\rm P}_{x,e}\circ\pi_1$ is compact. Since a compact perturbation of a Fredholm mapping is also Fredholm of the same index, it follows that $\text{\rm P}_{x,e}+\text{\rm J}_{x,e}$ is Fredholm of index $\opDim(X)$ as asserted.

Now let $Y$ be a compact neighbourhood of $x$ in $X$, let $(\Psi,\Cal{U},\Cal{V})$ be a graph chart of $X\times\Cal{E}$ over $Y$ and let $H^{\lambda+2}:\Cal{U}^{\lambda+2}\longrightarrow C^{*,\lambda}(\Sigma)$ and $\Theta^{\lambda+1}:\Cal{U}^{\lambda+1}\longrightarrow C^{*,\lambda}(\partial \Sigma)$ be respectively the mean curvature function and the boundary angle function in this chart (c.f. Section \ref{TheLocalSolutionSpace}). By Propositions \ref{PropPertMCIsConjugate}, \ref{PropInfMCIsConjugate} and \ref{InfBCIsConjugate}, for all $(y,f)\in\Cal{Z}_\oploc^{\lambda+2}\minter\Psi^{-1}(X\times\Omega)$, the mapping $D(H^{\lambda+2},\Theta^{\lambda+2})(y,f)$ is conjugate to $\text{\rm P}_{y,e}+\text{\rm J}_{y,e}$, where $e=\hat{\Phi}(y,f)$, and therefore defines a surjective, Fredholm map of index equal to $\opDim(X)$ from $T_xX\times C^{*,\lambda+2}(\Sigma)$ into $C^{*,\lambda}(\Sigma)\times C^{*,\lambda+1}(\partial \Sigma)$. It therefore follows from the Submersion Theorem for Banach manifolds that $\Cal{Z}^{\lambda+2}_\oploc\minter\Psi^{-1}(X\times\Omega)$ is a smooth, embedded submanifold of $\Cal{U}^{\lambda+2}$ of finite dimension equal to $\opDim(X)$ and, moreover, that $\Pi(\partial\Cal{Z}^{\lambda+2}_\oploc)\subseteq\partial Y$.

It remains to show independence. However, by the preceeding discussion, for all $\mu\geqslant\lambda$, $\Cal{Z}_\oploc^{\mu+2}\minter\Psi^{-1}(X\times\Omega)$ and $\Cal{Z}_{\oploc}^{\lambda+2}\minter\Psi^{-1}(X\times\Omega)$ are smooth, embedded, submanifolds of $\Cal{U}^{\mu+2}$ and $\Cal{U}^{\lambda+2}$ respectively, both of finite dimension equal to $\opDim(X)$. Let $i_{\mu,\lambda}:Y\times C^{*,\mu+2}(\Sigma)\longrightarrow Y\times C^{*,\lambda+2}(\Sigma)$ be the canonical embeddings. The mapping $i_{\mu,\lambda}$ is smooth and injective with injective derivative at every point, and therefore restricts to a diffeomorphism from $\Cal{Z}_{\oploc}^{\mu+2}\minter\Psi^{-1}(X\times\Omega)$ to $\Cal{Z}_\oploc^{\lambda+2}\minter\Psi^{-1}(X\times\Omega)$. It follows that the differential structure induced over $\Cal{Z}_\oploc\minter\Psi^{-1}(X\times\Omega)$ by the canonical embedding into $\Cal{U}^{\lambda+2}$ is independent of $\lambda$, and this completes the proof.\end{proof}

We recall the following technical result:
\begin{proposition}\label{PropSmoothnessToSmoothness}
Let $N_1,N_2$ be smooth, finite-dimensional manifolds and suppose that $N_2$ is compact. Let $\Phi$ be a mapping from $N_1$ into $C^\infty(N_2)$, and define the function $\varphi:N_1\times N_2\rightarrow\Bbb{R}$ such that for all $(p,q)\in N_1\times N_2$:
$$
\varphi(p,q) = \Phi(p)(q).
$$
$\Phi$ defines a smooth mapping from $N_1$ into $C^{*,\lambda}(N_2)$ for all $\lambda\in[0,\infty[\setminus\Bbb{N}$ if and only if $\varphi$ is smooth.
\end{proposition}
\begin{proof}For $k\in\left\{1,2\right\}$, denote by $D_k$ the partial derivative with respect to the $k$'th component. Choose $m\in\Bbb{N}$ and $\lambda>m$. If $\Phi$ defines a smooth mapping from $N_1$ into $C^{*,\lambda}(N_2)$, then $D_1^pD_2^q\varphi$ exists and is continuous for all $p,q\in\Bbb{N}\times\left\{0,...,m\right\}$. It follows that if $\Phi$ defines a smooth mapping from $N_1$ into $C^{*,\lambda}(N_2)$ for all $\lambda\in[0,\infty[\setminus\Bbb{N}$, then $\varphi$ is smooth. The reverse implication is trivial, and this completes the proof.
\end{proof}

\begin{theorem}\label{ThmSolutionSpaceIsSmooth}
Let $\Omega\subseteq\Cal{E}$ be such that $\partial_\omega\Cal{Z}(X|\Omega)=\emptyset$. If $\text{\rm P}_{x,e}+\text{\rm J}_{x,e}$ is surjective for all $(x,[e])\in\Cal{Z}(X|\Omega)$, then $\Cal{Z}(X|\Omega)$ carries the canonical structure of a smooth, compact manifold with boundary of finite dimension equal to $\opDim(X)$. Moreover, $\Pi$ defines a smooth map from $\Cal{Z}(X|\Omega)$ to $X$ such that:
$$
\Pi(\partial\Cal{Z}(X|\Omega))\subseteq\partial X,
$$
where $\partial\Cal{Z}(X|\Omega)$ here denotes the manifold boundary of $\Cal{Z}(X|\Omega)$.
\end{theorem}
\begin{proof}Since $\partial_\omega\Cal{Z}(X|\Omega)=\emptyset$, $\Cal{Z}(X|\Omega)$ is a closed subset of $\Cal{Z}(X)$. Since $\Cal{Z}(X)$ is compact, by Proposition \ref{ThmCompactness}, so too is $\Cal{Z}(X|\Omega)$. In addition, Proposition \ref{PropLocalSolutionSpacesAreSmooth} yields an atlas of smooth charts of $\Cal{Z}(X|\Omega)$, and it thus remains to prove that the transition maps are also smooth. Choose $(x,[e])\in\Cal{Z}(X|\Omega)$. Let $Y$ be a compact neighbourhood of $x$ in $X$ and let $\tilde{e}:Y\times \Sigma\longrightarrow M$ be such that $\widetilde{e}(x)=e$ and, for all $y\in Y$, $\widetilde{e}_y:=\widetilde{e}(y,\cdot)$ is an embedding such that $\tilde{e}_y(\Sigma)$ meets $\partial M$ orthogonally along $\partial \Sigma$ with respect to $g_y$. Let $N:Y\times \Sigma\longrightarrow TM$ be such that, for all $y\in Y$, $N_y:=N(y,\cdot)$ is the unit, normal vector field over $e_y$ with respect to $g_y$ which is compatible with the orientation. We define the mapping $F:Y\times \Sigma\times \Bbb{R}\longrightarrow M$ by:
$$
F(y,p,t) =\text{\rm E}(tN_y(p)),
$$
where $\text{\rm E}$ is the modified exponential map. Let $Y'$ be another compact neighbourhood of $x$ in $X$ and define $\tilde{e}'$, $N'$ and $F'$ in the same manner. For all $y$, we denote $F_y:=F(y,\cdot,\cdot)$ and $F'_y:=F'(y,\cdot,\cdot)$.

Let $(\Psi,\Cal{U},\Cal{V})$ and $(\Psi',\Cal{U}',\Cal{V}')$ be the graph charts of $X\times\Cal{E}$ generated by $(Y,\tilde{e})$ and $(Y',\widetilde{e}^\prime)$ respectively. Denote $Z_0=\Cal{Z}_{Y,\loc}\minter\Psi^{-1}(X\times\Omega)$ and let $B:=(\eta,\varphi):Z_0\longrightarrow Y\times C^\infty(\Sigma)$ be the canonical embedding. By definition $(\eta,\varphi)$ defines a smooth mapping from $Z_0$ into $Y\times C^{*,\lambda+2}(\Sigma)$ for all $\lambda$. It follows that $\eta$ is smooth and, by Proposition \ref{PropSmoothnessToSmoothness}, the function $\tilde{\varphi}:Z_0\times \Sigma\rightarrow\Bbb{R}$ given by:
$$
\tilde{\varphi}(z,p) := \varphi(z)(p)
$$
is smooth. Observe that, for all $(z,p)\in Z_0\times \Sigma$:
$$
(\hat{\Phi}\circ B)(z)(p) = F_{\eta(z)}(p,\tilde{\varphi}(z,p)).
$$

Let $\pi_1:\Sigma\times\Bbb{R}\longrightarrow S$ and $\pi_2:\Sigma\times\Bbb{R}\rightarrow\Bbb{R}$ be the projections onto the first and second factors respectively. We define $\alpha:Z_0\times \Sigma\rightarrow S$ and $\psi:Z_0\times \Sigma\rightarrow\Bbb{R}$ such that for all $(z,p)\in Z_0\times \Sigma$:
\begin{align*}
\alpha(z,p) &= (\pi_1\circ (F'_{\eta(z)})^{-1}\circ F_{\eta(z)})(p,\tilde{\varphi}(z,p)),\\
\psi(z,p) &= (\pi_1\circ (F'_{\eta(z)})^{-1}\circ F_{\eta(z)})(p,\tilde{\varphi}(z,p)).
\end{align*}
Observe that both $\alpha$ and $\psi$ are smooth mappings. Moreover, for all $z$ sufficiently close to $z_0:=(x,0)$, $\alpha_z:=\alpha(z,\cdot)$ is a diffeomorphism. We therefore define $\beta:Z_0\times \Sigma\rightarrow \Sigma$ such that for all $z\in Z_0$, $\beta_z:=\beta(z,\cdot)=\alpha_z^{-1}$, and we see that $\beta$ is also a smooth mapping. However, for all $z\in Z_0$:
$$
((\Psi')^{-1}\circ\Psi\circ B)(z) = (\eta(z),\psi_z\circ\beta_z).
$$
Since the mapping $(z,p)\mapsto(\psi_z\circ\beta_z)(p)$ is smooth, it follows from Proposition \ref{PropSmoothnessToSmoothness} again that $((\Psi')^{-1}\circ\Psi\circ B)$ is also a smooth mapping, and the transition maps are therefore smooth as desired.\end{proof}

\subsection{Surjectivity and Orientation}\label{TheOrientationOfTheSolutionSpace}
In order to define the orientation form over $\Cal{Z}(X|\Omega)$, we briefly review some basic spectral theory. Although we restrict attention here to self-adjoint operators, the results of this section extend to the more general framework of operators of compact resolvent (c.f. \cite{Kato} and \cite{SmiRos}).

Let $E$ and $F$ be Hilbert spaces. Let $i:E\rightarrow F$ be a compact, injective mapping with dense image. Let $A:E\rightarrow F$ be a Fredholm mapping. We say that $A$ is {\it self-adjoint} whenever it has the property that for all $u,v\in E$:
$$
\langle A(u),i(v)\rangle = \langle i(u),A(v)\rangle.
$$
Observe, in particular, that $A$ has Fredholm index zero. We henceforth identify $E$ with its image $i(E)$. Let $K\subseteq E\subseteq F$ be the kernel of $A$, let $R_f\subseteq F$ be its orthogonal complement and denote $R_e:=R_f\minter E$. Observe that $R_e$ and $R_f$ are closed subspaces of $E$ and $F$ respectively. Moreover:
$$
E = K\oplus R_e,\qquad F=K\oplus R_f.
$$
By the Closed Graph Theorem, $A$ restricts to an invertible, linear mapping from $R_e$ to $R_f$. We define $B:R_f\rightarrow R_e$ to be the inverse of this restriction. We extend $B$ to an operator from $F$ into $E$ by composing with the orthogonal projection of $F$ onto $R_f$, so that $B$ then defines a self-adjoint, compact operator from $F$ to itself. By the Sturm-Liouville Theorem, the (non-zero) spectrum of $B$, which we denote by $\opSpec(B)$ is a discrete subset of $\Bbb{R}\setminus\left\{0\right\}$ and every eigenvalue has finite multiplicity. We recall that the {\it spectrum} of $A$, which we denote by $\opSpec(A)$, is defined to be the set of all $\lambda\in\Bbb{R}$ such that $A-\lambda$ is not invertible, and we see that:
$$
\opSpec(A)\setminus\left\{0\right\} = \left\{\lambda\in\Bbb{R}\setminus\left\{0\right\}\ |\ \lambda^{-1}\in\opSpec(B)\right\},
$$
from which it follows, in particular, that $\opSpec(A)$ is a discrete subset of $\Bbb{R}$, and every eigenvalue has finite multiplicity.

We define the {\it nullity} of $A$ to be the dimension of the kernel of $A$, and we denote it by $\opNull(A)$. Since $A$ is Fredholm, $\opNull(A)$ is finite. We define the {\it index} of $A$ (not to be confused with its Fredholm index) to be the sum of the multiplicities of the negative eigenvalues of $A$, and we denote it by $\opInd(A)$. That is:
$$
\opInd(A) = \sum_{\lambda\in\opSpec(A)\minter(-\infty,0)}\opMult(\lambda).
$$
When $\opInd(A)$ is finite, we define the {\it signature} of $A$, which we denote by $\opSig(A)$ by:
$$
\opSig(A) = (-1)^{\opInd(A)}.
$$

We define $\Cal{F}^+(E,F)$ to be the set of all self-adjoint, Fredholm maps $A:E\longrightarrow F$ such that, for all non-zero $v\in E$:
\begin{equation}\label{EqnOperatorBoundedBelow}
\frac{\langle Av, v\rangle}{\langle v,v\rangle} \geqslant K,
\end{equation}
for some $K\in\Bbb{R}$, where $\langle\cdot,\cdot\rangle$ is the inner-product of $F$. Observe that $\opInd(A)<\infty$ for all $A\in\Cal{F}^+(E,F)$ and $\opSig(A)$ is therefore well defined for all such $A$. Observe, moreover, that since \ref{EqnOperatorBoundedBelow} is a convex condition, $\Cal{F}^+(E,F)$ is a convex subset of the set of self-adjoint, Fredholm mappings and is therefore, in particular, locally connected.
\begin{proposition}\label{NullityIndexConstant}
Let $C\subseteq\Cal{F}^+(E,F)$ be connected. If $\opNull$ is constant over $C$, then so too is $\opInd$.
\end{proposition}
\begin{proof}By classical spectral theory (c.f. \cite{Kato}), $\opInd$ defines a lower semi-continuous function over $\Cal{F}^+(E,F)$, whilst $(\opInd + \opNull)$ defines an upper-semicontinuous function over this set. Consequently, if $\opNull$ is continuous (i.e. locally constant), then so too is $\opInd$, and the result follows.\end{proof}

Let $X$ be a vector space with orientation form $\tau$ and finite dimension equal to $n$. Let $\Cal{M}:=\Cal{M}(X,E,F)$ be the space of all pairs $(M,A)$ with the properties that:
\begin{enumerate}
\item $M:X\rightarrow F$ is a linear mapping;
\item $A:E\rightarrow F$ is an element of $\Cal{F}^+(E,F)$; and
\item $M+A$ is surjective.
\end{enumerate}
Observe that $\opKer(M+A)$ defines a continuous mapping from $\Cal{M}$ into the Grassmannian of $n$-dimensional subspaces of $X\times E$.
\begin{proposition}\label{AInvertibleGivesPiInvertible}
If $\pi:X\times E\rightarrow X$ is the projection onto the first component, then $\pi$ restricts to a linear isomorphism from $\opKer(M+A)$ into $X$ if and only if $A$ is bijective.
\end{proposition}
\begin{proof}Since $\opDim(\opKer(M+A))=\opDim(X)$, this restriction is bijective if and only if it is injective. However:
$$
\opKer(M+A)\minter\opKer(\pi) = \left\{0\right\}\times\opKer(A),
$$
from which the result follows.\end{proof}
When $A$ is invertible, we therefore define the {\it orientation form} $\sigma(M,A)$ over $\opKer(M+A)$ by:
$$
\sigma(M,A) = \opSig(A)(\pi^*\tau).
$$
We identify orientation forms that differ only by a positive factor and we obtain (c.f. Proposition $4$ of \cite{WhiteI}):
\begin{proposition}\label{PropDefinitionOfOrientationForm}
$\sigma(M,A)$ extends continuously to define an orientation form over $\opKer(M+A)$ for all $(M,A)\in\Cal{M}$.
\end{proposition}
\begin{remark}
In other words, for all $(M,A)\in\Cal{M}$, the subspace $\Cal{K}(M,A)$ carries a canonical orientation.
\end{remark}
\begin{proof}Let $K$ be the kernel of $A$. Let $R_f$ be its orthogonal complement in $F$ and denote $R_e:=R_f\minter E$. We denote each of $R_e$ and $R_f$ simply by $R$ when no ambiguity arises. Let $p_1:F\rightarrow K$ and $p_2:F\rightarrow R$ be the orthogonal projections. For convenience, we furnish $X$ with a positive-definite inner product. Let $L$ be the kernel of $p_1\circ M$ and let $S$ be its orthogonal complement. Let $q_1:X\rightarrow L$ and $q_2:X\rightarrow S$ be the orthogonal projections. With respect to the decompositions $E=K\oplus R$, $F=K\oplus R$ and $X=L\oplus S$, we denote:
$$
M = \begin{pmatrix} M_{11}& M_{12}\\ M_{21}& M_{22}\end{pmatrix},\qquad A = \begin{pmatrix} A_{11}& A_{12}\\  A_{21}& A_{22}\end{pmatrix}.
$$
By definition $M_{11}=0$ and $A_{11},A_{12},A_{21}=0$. The mapping $M_{12}$ coincides with the restriction of $p_1\circ M$ to $S$. We claim that this mapping is a linear isomorphism. Indeed, by definition of $L$ and $S$, $M_{12}$ is injective. To see that it is surjective, observe that, for all $w\in K$, by surjectivity of $M+A$, there exists $(u,v)\in X\times E$ such that $M(u) + A(v) = w$. In particular:
$$
w = p_1(w) = (p_1\circ M)(u) + (p_1\circ A)(w) = (p_1\circ M)(u),
$$
and surjectivity follows. Consequently, we identify $S$ with $K$ and assume that $M_{12}=\opId$.

Let $\pi$ and $\tilde{\pi}$ denote the canonical projections from $X\times E$ onto $X$ and $E$ respectively. We now claim that $(q_1\circ\pi,p_1\circ\tilde{\pi})$ defines a linear isomorphism from $\opKer(M+A)$ onto $L\oplus K$. Indeed, since:
$$
\opDim(L\oplus K) = \opDim(L\oplus S) = \opDim(X) = \opDim(\opKer(M+A)),
$$
it suffices to show that this mapping is injective. However, let $(u,v)\in\opKer(M+A)$ be such that $q_1(u),p_1(v)=0$. By definition, $M(u)=-A(v)$. Thus, bearing in mind that $M_{12}=\opId$, $q_2(u)=(p_1\circ M)(u) = -(p_1\circ A)(v) = 0$, from which it follows that $u=0$. Moreover, $A(v)=-M(u)=0$, and since the restriction of $A$ to $R$ is invertible, $p_2(v)=0$. Consequently, $v=0$, and $(q_1\circ\pi,p_1\circ\tilde{\pi})$ therefore defines a linear isomorphism from $\opKer(M+A)$ onto $L\oplus K$ as asserted.

By classical perturbation theory (c.f. \cite{Kato}), there exists a neighbourhood $U$ of $(M,A)$ in $\Cal{M}$ and smooth mappings $Q_e:U\rightarrow\opLin(E)$ and $Q_f:U\rightarrow\opLin(F)$ such that $Q_e(M,A),Q_f(M,A)=\opId$, and for all $(M',A')\in U$, $Q_f(M',A')$ is an isometry of $F$ whose restriction to $E$ coincides with $Q_e(M',A')$ and $Q_f(M',A')^*A'Q_e(M',A')$ preserves both $K$ and $R$. Conjugating with $Q$, we may therefore assume that for a given element $(M',A')\in U$, $A'$ preserves both $K$ and $R$.

Let $\tau_1$ and $\tau_2$ be non-zero volume forms over $L$ and $S$ respectively such that $\tau=\tau_1\wedge\tau_2$. Since we identify $S$ with $K$, we may also consider $\tau_2$ as a volume form over $K$. Observe that, over $\opKer(M',A')$, $M'\circ\pi$ coincides with $-A\circ\tilde{\pi}$. In particular, observing that $A'$ commutes with $p_1$:
$$
p_1\circ M'\circ \pi=-p_1\circ A'\circ\tilde{\pi}=-A'\circ p_1\circ\tilde{\pi}.
$$
Thus, denoting the dimension of $K$ by $k$:
\begin{align*}
\pi^*\tau &= (\pi^* q_1^*\tau_1)\wedge(\pi^* q_2^*\tau_2)\\
&=(-1)^k(\pi^* q_1^*\tau_1)\wedge ((p_1\circ M'\circ q_2)^{-1}\circ A'\circ (p_1\circ \tilde{\pi}))^*\tau_2\\
&=(-1)^k(\pi^* q_1^*\tau_1)\wedge ((M'_{12})^{-1}\circ A'_{11}\circ(p_1\circ\tilde{\pi}))^*\tau_2\\
&=(-1)^k\opDet(A'_{11})\opDet(M'_{12})^{-1}(\pi^* q_1^*\tau_1)\wedge (p_1\circ\tilde{\pi})^*\tau_2.
\end{align*}
We may suppose that the restriction of $(q_1\circ\pi,p_1\circ\tilde{\pi})$ to $\opKer(M'+A')$ is a linear isomorphism so that $\tilde{\sigma}(M',A'):=(\pi^* q_1^*\tau_1)\wedge (p_1\circ\tilde{\pi})^*\tau_2$ defines a non-zero volume form over $\opKer(M'+A')$. In addition, since $M_{12}=\opId$, we may suppose that $\opDet(M'_{12})$ is always positive, and so, over $\opKer(M',A')$:
$$
\opSig(A'_{11})\pi^*\tau \sim (-1)^k\tilde{\sigma}(M',A'),
$$
where $\sim$ denotes equivalence of volume forms up to a positive factor. Finally, we may assume that $\opSig(A'_{22})=\opSig(A_{22})$, and since:
$$
\opSig(A') = \opSig(A'_{11}) + \opSig(A'_{22}) = \opSig(A'_{11}) + \opSig(A_{22}),
$$
we conclude that over $\opKer(M'+A')$:
$$
\sigma(M,A) = \opSig(A')\pi^*\tau \sim (-1)^k\opSig(A_{22})\tilde{\sigma}(M',A').
$$
Since the right-hand side defines a continuous family of non-vanishing volume forms, we see that $\sigma$ extends continuously over a neighbourhood of every point of $\Cal{M}$, and therefore extends continuously over the whole of $\Cal{M}$, as desired.\end{proof}
\begin{proposition}\label{PropRegularValuesHaveInvertibleJacobiOperator}
Let $\Omega\subseteq\Cal{E}$ be such that $\partial_\omega\Cal{Z}(X|\Omega)=\emptyset$. If $\text{\rm P}_{x,e}+\text{\rm J}_{x,e}$ is surjective for all $(x,[e])\in\Cal{Z}(X|\Omega)$ then $(x,[e])\in\Cal{Z}(X|\Omega)$ is a regular point of the restriction of $\Pi$ to $\Cal{Z}(X|\Omega)$ if and only if $\text{\rm J}_{x,e}$ is invertible.
\end{proposition}
\begin{proof}
Choose $(x,[e])\in\Cal{Z}(X|\Omega)$. Let $Y$ be a compact neighbourhood of $x$ in $X$ and let $(\Psi,\Cal{U},\Cal{V})$ be a graph chart of $X\times\Cal{E}$ about $(x,[e])$ over $Y$. Let $H$ and $\Theta$ be the mean curvature function and the boundary angle function in this chart. Let $\Pi':Y\times C^\infty(\Sigma)\rightarrow Y$ be the projection onto the first factor. The point $(x,[e])$ is a regular point of $\Pi$ if and only if it is a regular point of $\Pi'$. However:
\begin{align*}
T_{(x,0)}\Cal{Z}_\oploc\minter\opKer(D_{(x,0)}\Pi')&=\opKer(D_{(x,0)}(H,\Theta))\minter(\left\{0\right\}\times C^\infty(\Sigma))\\
&=\opKer(\text{\rm P}_{x,e} + \text{\rm J}_{x,e})\minter(\left\{0\right\}\times C^\infty(\Sigma))\\
&=\opKer(\text{\rm J}_{x,e}).
\end{align*}
We conclude that $(x,[e])$ is a regular value of $\Pi$ if and only if $\text{\rm J}_{x,e}$ is invertible, as desired.
\end{proof}

Combining these results yields:
\begin{theorem}\label{ThmSolutionSpaceIsOriented}
Let $\Omega\subseteq\Cal{E}$ be such that $\partial_\omega\Cal{Z}(X|\Omega)=\emptyset$. If $X$ is orientable with orientation form $\tau$, and if $\text{\rm P}_{x,e}+\text{\rm J}_{x,e}$ is surjective for all $(x,[e])\in\Cal{Z}(X|\Omega)$, then $\Cal{Z}(X|\Omega)$ carries a canonical orientation $\sigma$. Moreover, $(x,[e])$ is a regular point of the restriction of $\Pi$ to $\Cal{Z}(X|\Omega)$ if and only if $\text{\rm J}_{x,e}$ is non-degenerate, and in this case:
$$
\sigma(x,[e]) \sim \opSig(\text{\rm J}_{x,e})\Pi^*\tau,
$$
\noindent where $\opSig(\text{\rm J}^h_{e})$ is defined to be the signature of the restriction of $\text{\rm J}^h_{x,e}$ to the kernel of $\text{\rm J}^\theta_{x,e}$.
\end{theorem}
\begin{proof}
By Theorem \ref{ThmSolutionSpaceIsSmooth}, $\Cal{Z}(X|\Omega)$ is a smooth manifold of finite dimension equal to $\opDim(X)$. Choose $(x,[e])\in\Cal{Z}(X|\Omega)$. Observe that $T_{(x,[e])}\Cal{Z}(X|\Omega)$ identifies canonically with $\opKer(\text{\rm P}_{x,e}+\text{\rm J}_{x,e})$. Let $H^2(\Sigma)$ be the Sobolev space of $L^2$ functions over $\Sigma$ whose distributional derivatives up to order $2$ are also of class $L^2$. By the Sobolev Trace Formula (c.f. Proposition $4.5$ of Section $4$ of \cite{TaylorI}), $\text{\rm J}^\theta_{x,e}$ maps $H^{2}(\Sigma)$ into $H^{1/2}(\partial \Sigma)$. We denote by $H^2_\oprob(\Sigma)$ the kernel of $\text{\rm J}^\theta_{x,e}$ in this space. Observe that $H^2_\oprob(\Sigma)$ embeds canonically into $L^2(\Sigma)$, and that this embedding is compact with dense image. By Proposition \ref{PropJacobiOperatorIsSelfAdjoint}, the restriction of $\text{\rm J}^h_{x,e}$ to $H^2_\oprob(\Sigma)$ is self-adjoint. The preceeding discussion therefore applies to this restriction of $\text{\rm J}^h_{x,e}$, and we define the orientation form $\sigma$ over $\opKer(\text{\rm P}_{x,e}+\text{\rm J}_{x,e})$ as in Proposition \ref{PropDefinitionOfOrientationForm}. Since this kernel is canonically identified with $T_{(x,[e])}\Cal{Z}(X|\Omega)$, $\sigma$ also defines an orientation form over this space. It follows from the definition that this construction yields a continuous family of orientation forms over $\Cal{Z}(X|\Omega)$, so that the manifold carries a canonical orientation, as desired. Finally, by Proposition \ref{PropRegularValuesHaveInvertibleJacobiOperator}, $(x,[e])\in\Cal{Z}(X|\Omega)$ is a regular value of $\Pi$ if and only if $\text{\rm J}_{x,e}$ is invertible, and so, by definition, and bearing in mind the definition following Proposition \ref{AInvertibleGivesPiInvertible}, we have
$
\sigma(x,[e]) \sim \opSig(\text{\rm J}_{x,e})\Pi^*\tau$,
as desired.
\end{proof}

The results of this section may be summarised as follows:
\begin{theorem}\label{ThmIntegerValuedDegree}
Let $\Omega\subseteq\Cal{E}$ be such that $\partial_\omega\Cal{Z}(X|\Omega)=\emptyset$. There exists an extension $\tilde{X}$ of $X$, which we may take to be orientable, such that $\partial_\omega\Cal{Z}(\tilde{X}|\Omega)=\emptyset$, $\Cal{Z}(\tilde{X}|\Omega)$ carries canonically the structure of a smooth orientable manifold of finite dimension equal to that of $\tilde{X}$, and $\Pi(\partial\Cal{Z}(\tilde{X}|\Omega))\subseteq\partial\tilde{X}$. In particular, the restriction of $\Pi$ to $\Cal{Z}(X|\Omega)$ has a well-defined ${\mathbb Z}$-valued degree. Moreover, a point $x\in\tilde{X}$ is a regular value of this restriction if and only if $\text{\rm J}_{x,e}$ is non-degenerate for all $(x,[e])\in\Cal{Z}(\left\{x\right\}|\Omega)$, and in this case:
$$
\opDeg(\Pi|\Omega) = \sum_{(x,[e])\in\Cal{Z}(\left\{x\right\}|\Omega)}\opSig(\text{\rm J}_{x,e}),
$$
where $\opSig(\text{\rm J}_{x,e})$ is defined to be the signature of the restriction of $\text{\rm J}^h_{x,e}$ to the kernel of $\text{\rm J}^\theta_{x,e}$.
\end{theorem}
\begin{proof} By Theorem \ref{PropSurjectivityForExtensions}, there exists an extension $\tilde{X}$ of $X$ with $\partial_\omega\Cal{Z}(\tilde{X}|\Omega) = \emptyset$ and such that the operator $\text{\rm P}_{x,e}+\text{\rm J}_{x,e}$ defines a surjective mapping from $T_x\tilde{X}\times C^\infty(\Sigma)$ into $C^\infty(\Sigma)\times C^\infty(\partial\Sigma)$ for all $(x,[e])\in\Cal{Z}(\tilde{X}|\Omega)$. Upon extending $\tilde{X}$ further if necessary, we may assume that $\tilde{X}$ is orientable with orientation form, $\tau$, say. By Theorem \ref{ThmSolutionSpaceIsSmooth}, $\Cal{Z}(\tilde{X}|\Omega)$ carries the structure of a smooth, compact manifold with boundary, of finite dimension equal to that of $\tilde{X}$ and moreover $\Pi(\partial\Cal{Z}(\tilde{X}|\Omega)\subseteq\partial\tilde{X}$. By Theorem \ref{ThmSolutionSpaceIsOriented}, $\Cal{Z}(\tilde{X}|\Omega)$ carries a canonical orientation form $\sigma$. Moreover, $(x,[e])\in\Cal{Z}(\tilde{X}|\Omega)$ is a regular point of the restriction of $\Pi$ to $\Cal{Z}(\tilde{X}|\Omega)$ if and only if $\text{\rm J}_{x,e}$ is non-degenerate, and in this case $\sigma\sim\opSig(\text{\rm J}_{x,e})\Pi^*\tau$, where $\sim$ here denotes equivalence of volume forms up to a positive factor. By Proposition \ref{ThmCompactness}, $\Pi$ defines a proper map from $\Cal{Z}(\tilde{X}|\Omega)$ into $\tilde{X}$, and so, by classical differential topology (c.f. \cite{GuillemanPollack}), its restriction to $\Cal{Z}(\tilde{X}|\Omega)$ has a well-defined ${\mathbb Z}$-valued degree. Moreover $x\in\tilde{X}$ is a regular value if and only if $\text{\rm J}_{x,e}$ is non-degenerate for all $(x,[e])\in\Cal{Z}(\left\{x\right\}|\Omega)$, and in this case, by definition of the degree:
$$
\opDeg(\Pi|\Omega) = \sum_{(x,[e])\in\Cal{Z}(\left\{x\right\}|\Omega)}\opSig(\text{\rm J}_{x,e}),
$$
as desired.\end{proof}

%%%%%%%%%%%%%%%%%%%%%%%%%%%
% NON-DEGENERATE FAMILIES %
%%%%%%%%%%%%%%%%%%%%%%%%%%%

\section{Non-Degenerate Families}
\subsection{Non-degenerate families}\label{NonDegenerateFamilies}

Let $Z$ be a closed, finite-dimensional manifold. Let $\Cal{F}:Z\rightarrow\Cal{E}$ be a continuous mapping. We say that $\Cal{F}$ is {\it smooth} whenever it has the property that for all $z\in\Cal{Z}$, there exists a compact neighbourhood $Z_0$ of $z$ in $Z$ and a smooth function $e:Z_0\times \Sigma\rightarrow M$ such that for all $w\in Z_0$, $e_w:=e(w,\cdot)$ is an element of $\hat{\Cal{E}}$ and $\Cal{F}(w)=[e_w]$. We refer to the pair $(Z_0,e)$ as a {\it local parametrisation} of $(Z,\Cal{F})$ about $z$. We say that $\Cal{F}$ is an {\it immersion} whenever it has the property that for all $z\in\Cal{Z}$, for every local parametrisation $(Z_0,e)$ of $(Z,\Cal{F})$ about $z$, for all $w\in Z_0$ and for all non-zero $\xi_w\in T_wZ_0$, the vector field $(D_1e)_w(\xi_w)$ is not tangent to $e_w(\Sigma)$ at at least one point, where $D_1e$ is the partial derivative of $e$ with respect to the first component in $Z_0\times \Sigma$. We say that $\Cal{F}$ is an {\it embedding} whenever it is, in addition, injective.
\begin{proposition}\label{lowerboundondimensionofkernel}
Let $g_0$ be an admissable metric over $M$ and let $\Cal{F}:Z\rightarrow\Cal{E}$ be a smooth embedding. If $\Cal{F}(z)$ is free boundary minimal with respect to $g_0$ for all $z\in Z$, then for all $z\in Z$:
$$
\opNull(\text{\rm J}_{g_0,\Cal{F}(z)}) = \opDim(\opKer(\text{\rm J}_{g_0,\Cal{F}(z)})) \geqslant \opDim(Z).
$$
\end{proposition}
\begin{proof}Let $n$ be the dimension of $Z$. Choose $z\in Z$. Observe that $\Cal{F}$ defines an $n$-dimensional family of non-trivial, free boundary minimal perturbations of $\Cal{F}(z)$, from which it follows that the derivative of $\Cal{F}$ defines an injective mapping from $T_zZ$ into $\opKer(\text{\rm J}_{g_0,\Cal{F}(z)})$. More formally, this injection is explicitely described in the proof of Proposition \ref{PropLambdaLiesInKernel} (below). In particular, $\opNull(\text{\rm J}_{g_0,\Cal{F}(z)})\geqslant n$, and the result follows.
\end{proof}
Proposition \ref{lowerboundondimensionofkernel} motivates the following definition: if $g_0$ is an admissable metric over $M$, and if $\Cal{F}(z)$ is free boundary minimal with respect to $g_0$ for all $z\in Z$, then $(Z,\Cal{F})$ is said to be a {\it non-degenerate family} whenever it has in addition the property that for all $z\in Z$:
$$
\opNull(\text{\rm J}_{g_0,\Cal{F}(z)}) = \opDim(\opKer(\text{\rm J}_{g_0,\Cal{F}(z)})) = \opDim(Z).
$$
We recall from Proposition \ref{NullityIndexConstant} that if $\opNull(\text{\rm J}_{g_0,\Cal{F}(z)})$ is constant, then so too is $\opInd(\text{\rm J}_{g_0,\Cal{F}(z)})$, and we therefore define the {\it index} of the family $Z$, which we denote by $\opInd(Z)$ to be equal to $\opInd(\text{\rm J}_{g_0,\Cal{F}(z)})$ for all $z\in Z$.

Let $(Z_0,e)$ be a local parametrisation of $(Z,\Cal{F})$. Let $X$ be another smooth, compact, finite-dimensional manifold. Let $x_0$ be an element of $X$ and let $g:X\times M\rightarrow\opSym^+(TM)$ be a smooth function such that $g_x:=g(x,\cdot)$ is admissable for all $x\in X$ and $g(x_0,\cdot)=g_0$. We extend $e$ and $g$ to functions defined over $X\times Z_0$ by setting $e$ to be constant in the $X$ direction and by setting $g$ to be constant in the $Z_0$ direction. Let $(\Psi,\Cal{U},\Cal{V})$ be the graph chart of $X\times Z_0\times\Cal{E}$ generated by $(X\times Z_0,e)$ and let $H$ and $\Theta$ be respectively the mean curvature function and the boundary angle function in this chart. We define $\Cal{K}\subseteq X\times Z_0\times C^\infty(\Sigma)$ by:
$$
\Cal{K} = \left\{(x,z,f)\ |\ f\in\opKer(\text{\rm J}_{g_0,e_z})\right\},
$$
and for all $(x,z)\in X\times Z_0$, we denote the fibre over $(x,z)$ by $\Cal{K}_{x,z}$. Observe that $\Cal{K}$ is a finite-dimensional vector bundle over $X\times Z_0$ of constant dimension equal to $\opDim(Z)$. We shall see presently that $\Cal{K}$ is smooth, and is in fact canonically isomorphic to $TZ_0$. We also define $\Cal{K}^\perp\subseteq X\times Z_0\times C^\infty(\Sigma)$ such that for all $(x,z)\in X\times Z_0$ the fibre $\Cal{K}^\perp_{x,z}$ of $\Cal{K}^\perp$ is the orthogonal complement of $\Cal{K}_{x,z}$ in $C^\infty(\Sigma)$ with respect to the $L^2$-inner-product of $e_z^*g_0$.

\begin{proposition}\label{PropLocalPerturbationSections}
There exists a compact neighbourhood $Y$ of $x_0$ in $X$ and a continuous function $F:Y\times Z_0\rightarrow C^\infty(\Sigma)$ such that $F(0,z)=0$ for all $z$ and, for all $(x,z)\in Y\times Z_0$:
\begin{enumerate}
\item $F_{x,z}:=F(x,z)$ is an element of $\Cal{K}^\perp_{x,z}$;
\item $\Theta(x,z,F_{x,z})=0$; and
\item $H(x,z,F_{x,z})$ is an element of $\Cal{K}_{x,z}=\opKer(\text{\rm J}_{g_0,e_z})$.
\end{enumerate}
Moreover:
\begin{enumerate}
\item the function $F$ is unique in the sense that if $Y'\subseteq Y$ is another compact neighbourhood of $x_0$ and if $F':Y'\times Z_0\rightarrow C^\infty(\Sigma)$ is another continuous function with the same properties, then $F'=F$; and
\item the function $f:Y\times Z_0\times S\rightarrow\Bbb{R}$ given by $f(x,z,p)=F(x,z)(p)$ is smooth.
\end{enumerate}
\end{proposition}
\begin{remark}
It is, in fact, sufficient for the proof of this result to assume that $\opNull(\text{\rm J}_{g_0,\Cal{F}(z)})$ has constant dimension.
\end{remark}
\begin{remark}
Recall that if $e:\Sigma\longrightarrow M$ is a free boundary minimal embedding which is non-degenerate in the sense that $\text{\rm J}_{g_0,e}$ is invertible, then for any infinitesimal perturbation $\delta g_0$ of $g_0$, there exists a unique infinitesimal perturbation $\delta e$ of $e$ with the property that $e+\delta e$ is free boundary minimal with respect to $g+\delta g_0$. Proposition \ref{PropLocalPerturbationSections} constitutes a generalisation of this result to the case where $\text{\rm J}_{g_0,e}$ has non-trivial kernel.
\end{remark}
\begin{proof}
Denote $n=\opDim(\opKer(\text{\rm J}_{g_0,e_z}))=\opDim(\Cal{K})$. Choose $\lambda\in[0,\infty[\setminus\Bbb{N}$. Observe that for all $z\in Z_0$, $D(H^\lambda,\Theta^\lambda)(x_0,z,0)=\text{\rm J}_{g_0,e_z}$. By Proposition \ref{PropEllipticRegularityOfJacobiOperatorII}, for all $z\in Z_0$:
$$
\opDim(\opKer(D(H^\lambda,\Theta^\lambda))) = \opDim(\opKer(\text{\rm J}_{g_0,e_z})) = n,
$$
and since $(H^\lambda,\Theta^\lambda)$ is a smooth, Fredholm mapping, it follows from the Submersion Theorem for Banach manifolds that $\Cal{K}$ defines a smooth Banach sub-bundle of $X\times Z_0\times C^{*,\lambda}(\Sigma)$ with typical fibre of dimension equal to $n$.

Let $\Cal{K}^{\lambda,\perp}\subseteq X\times Z_0\times C^{*,\lambda}(\Sigma)$ be the Banach sub-bundle whose fibre over any point $(x,z)\in X\times Z_0$ coincides with the orthogonal complement of $\Cal{K}_{x,z}$ in $C^{*,\lambda}(\Sigma)$ with respect to the $L^2$ inner-product of $e_z^*g_0$. We define $\Pi^\lambda:X\times Z_0\times C^{*,\lambda}(\Sigma)\longrightarrow\Cal{K}^{\lambda,\perp}$ such that for all $(x,z)\in X\times Z_0$, $\Pi_{x,z}^\lambda:=\Pi^\lambda(x,z,\cdot)$ is the orthogonal projection of $C^{*,\lambda}(\Sigma)$ onto $\Cal{K}^{\lambda,\perp}_{x,z}$ with respect to the $L^2$-inner-product of $e_z^*g_0$. Observe that $\Pi^\lambda$ is a smooth Banach bundle mapping.

We define $\overline{H}^{\lambda+2}:\Cal{U}^{\lambda+2}\rightarrow\Cal{K}^{\lambda,\perp}$ such that for all $(x,z,f)\in\Cal{U}^{\lambda+2}$:
$$
\overline{H}^{\lambda+2}(x,z,f) = (x,z,(\Pi_z^\lambda\circ H^{\lambda+2})(x,z,f))).
$$
Let $D_3\overline{H}^{\lambda+2}$ be the partial derivative of $\overline{H}^{\lambda+2}$ with respect to the third component in $X\times Z_0\times C^{*,\lambda+2}(\Sigma)$. Choose $z\in Z_0$. We claim that the restriction of  $D_3(\overline{H}^{\lambda+2},\Theta^{\lambda+2})(x_0,z,0)=(\Pi_z^\lambda\circ \text{\rm J}^h_{g_0,e_z},\text{\rm J}^\theta_{g_0,e_z})$ to $\Cal{K}^{\lambda+2,\perp}_{x_0,z}$ defines a linear isomorphism onto $\Cal{K}^{\lambda,\perp}_{x_0,z}\times C^{*,\lambda+1}(\partial \Sigma)$. Indeed, by definition of $\Cal{K}$, $\text{\rm J}_{g_0,e_z}$ restricts to a linear isomorphism from $\Cal{K}_{x_0,z}^{\lambda+2,\perp}$ to $\opIm^{\lambda+2}(\text{\rm J}_{g_0,e_z})$, and it thus suffices show that the restriction of $(\Pi^\lambda_{x_0,z},\opId)$ to $\opIm^{\lambda+2}(\text{\rm J}_{g,e_z})$ defines a linear isomorphism onto $\Cal{K}^{\lambda,\perp}_{x_0,z}\times C^{*,\lambda+1}(\partial \Sigma)$. However, by definition of $\Pi^\lambda$, and bearing in mind that $\text{\rm J}_{g_0,e_z}$ is Fredholm of index zero:
$$
\opDim(\opKer(\Pi_{x_0,z}^\lambda,\opId)) = n = \opCodim(\opIm^{\lambda+2}(\text{\rm J}_{g_0,e_z})).
$$
Consequently, if $\opKer(\Pi_{x_0,z}^\lambda,\opId)\minter \opIm^{\lambda+2}(\text{\rm J}_{g_0,e_z}) = \left\{0\right\}$, then:
$$
C^\lambda(\Sigma)\times C^{\lambda+1}(\partial \Sigma) = \opKer(\Pi_{x_0,z}^\lambda,\opId)\oplus\opIm^{\lambda+2}(\text{\rm J}_{g_0,e_z}),
$$
and since $(\Pi_{x_0,z}^\lambda,\opId)$ is surjective, it would follow that its restriction to $\opIm^{\lambda+2}(\text{\rm J}_{g_0,e_z})$ defines a linear isomorphism onto $\Cal{K}^{\lambda,\perp}_{x_0,z}$. It thus suffices to show that this intersection is trivial. However, let $(\psi,0)$ be an element of the intersection $\opKer(\Pi_{x_0,z}^\lambda,\opId)\minter\opIm^{\lambda+2}(\text{\rm J}_{g_0,e_z})$. In particular, there exists $\varphi\in C^{*,\lambda+2}(\Sigma)$ such that $\text{\rm J}^h_{g_0,e_z}(\varphi)=\psi$ and $\text{\rm J}^\theta_{g_0,e_z}(\varphi)=0$. Moreover, by definition of $\Pi^\lambda$, $\psi\in\opKer(\text{\rm J}_{g_0,e_z})$ and so $\text{\rm J}^\theta_{g_0,e_z}(\psi)=0$. Thus, bearing in mind Proposition \ref{PropJacobiOperatorIsSelfAdjoint}, and denoting by $\opdVol$ the volume form of $e_z^*g_0$, we have:
$$
\int_\Sigma\psi^2\,\opdVol = \int_\Sigma (\text{\rm J}^h_{g_0,e_z}\varphi)\psi\,\opdVol = \int_\Sigma \varphi(\text{\rm J}^h_{g_0,e_z}\psi)\,\opdVol = 0,
$$
and the intersection is therefore trivial, and the restriction of $D_3(\overline{H}^{\lambda+2},\Theta^{\lambda+2})$ to $\Cal{K}_{x_0,z}^{\lambda+2,\perp}$ at $(x_0,z,0)$ defines a linear isomorphism onto $\Cal{K}_{x_0,z}^{\lambda,\perp}\times C^{*,\lambda+1}(\partial \Sigma)$, as asserted.

Since $Z_0$ is compact, it follows from the implicit function theorem for Banach manifolds that there exists a compact neighbourhood $Y$ of $x_0$ and a continuous mapping $F:Y\times Z_0\rightarrow\Cal{K}^{\lambda+2,\perp}$ such that for all $z\in Z_0$, $F(x_0,z)=0$ and for all $(x,z)\in Y\times Z_0$:
$$
(\overline{H}^{\lambda+2},\Theta^{\lambda+2})(x,z,F(x,z)) =(0,0).
$$
Moreover, we may assume that $F$ is unique in the sense described above, and since any continuous mapping from $Y\times Z_0$ into $C^\infty(\Sigma)$ which satisfies $(1)$ is in particular a continuous mapping from $Y\times Z_0$ into $\Cal{K}^{\lambda+2,\perp}$ satisfying the above relation, uniqueness follows.

We now prove that $f:Y\times Z_0\times \Sigma\longrightarrow\Bbb{R}$ is smooth. We claim that $F$ defines a smooth mapping into $\Cal{K}^{\mu+2,\perp}$ for all $\mu\in[0,\infty[\setminus\Bbb{N}$. Indeed, choose $\mu\in[0,\infty[\setminus\Bbb{N}$ such that $\mu>\lambda$. By Proposition \ref{PropEllipticRegularityOfJacobiOperatorII}, $\opKer(\text{\rm J}_{g_0,e_z})\subseteq C^\infty(\Sigma)$. Thus, for all $(x,z)\in Y\times Z$:
$$
H^{\lambda+2}(x,z,F(x,z))\in C^\infty(\Sigma),
$$
and it follows by Proposition \ref{PropEllipticRegularityOfPMCInGraphCharts} that for all $(x,z)\in Y\times Z_0$:
$$
F(x,z)\in C^\infty(\Sigma) \subseteq C^{*,\mu+2}(\Sigma).
$$
Since invertibility is an open property and since $Z_0$ is compact, upon reducing $Y$ if necessary, we may suppose that $D_3\overline{H}^{\lambda+2}(x,z,f(x,z))$ defines an invertible map from $\Cal{K}^{\lambda+2,\perp}_{x,z}$ into $\Cal{K}^{\lambda,\perp}_{x,z}\times C^{*,\lambda+1}(\partial \Sigma)$ for all $(x,z)\in Y\times Z$. Then, by Propositions \ref{PropEllipticRegularityOfJacobiOperatorI} and \ref{PropEllipticRegularityOfJacobiOperatorII} that for all $\mu>\lambda$, $D_3\overline{H}^{\mu+2}(x,z,F(x,z))$ also defines an invertible map from $\Cal{K}_{x,z}^{\mu+2,\perp}$ into $\Cal{K}_{x,z}^{\mu,\perp}\times C^{*,\mu+1}(\partial \Sigma)$. Thus, by the implicit function theorem for Banach manifolds, for all $(x,z)\in Y\times Z_0$, there exists a neighbourhood, $\Omega$ of $(x,z)\in Y\times Z$ and a continuous mapping $F':\Omega\rightarrow\Cal{K}^{\mu+2,\perp}\subseteq\Cal{K}^{\lambda+2,\perp}$ such that $F'(x,z)=F(x,z)$ and for all $(x',z')\in\Omega$:
$$
(\tilde{H}^{\mu+2},\Theta^{\mu+2})(x',z',F'(x',z')) = (0,0).
$$
Since $F'$ is also unique in the sense described above, it coincides with the restriction of $F$ to $\Omega$, from which it follows that $F$ defines a smooth mapping from $\Omega$ into $C^{*,\mu+2}(\Sigma)$ as asserted. Thus, by Proposition \ref{PropSmoothnessToSmoothness}, the function $f:Y\times Z\times \Sigma\rightarrow\Bbb{R}$ given by:
$$
f(y,z,p) = F(y,z)(p)
$$
is smooth, as desired. In particular, $F$ defines a continuous mapping from $Y\times Z$ into $C^\infty(\Sigma)$. Finally, observe that for all $(x,z)\in Y\times Z_0$:
\begin{align*}
&F(x,z) \in \Cal{K}_{x,z}^\perp,\ \text{and}\\
&H^{\lambda+2}(x,z,F(x,z)) \in \Cal{K}_{x,z},
\end{align*}
and this completes the proof.\end{proof}

\subsection{Global sections over non-degenerate families}\label{GlobalSectionsOverNonDegenerateFamilies}Let $Y\subseteq X$ and $F:Y\times Z_0\longrightarrow C^\infty(\Sigma)$ be as in Proposition \ref{PropLocalPerturbationSections}. We define $\widetilde{h}:Y\times Z_0\times \Sigma\rightarrow\Bbb{R}$ such that for all $(x,z)\in Y\times Z_0$:
$$
\widetilde{h}_{x,z} := \widetilde{h}(x,z,\cdot) = H(x,z,F_{x,z}).
$$
We consider $\widetilde{h}$ as a smooth family of sections of $\Cal{K}$ over $Z_0$ parametrised by $Y$. We now show how this family is canonically identified with a family of sections of $T^*Z_0$ parametrised by $Y$, and moreover, upon reducing $Y$ if necessary, that these sections can be combined to yield a family of sections over the whole of $T^*Z$.

Define $\widetilde{e}:Y\times Z_0\times \Sigma\rightarrow M$ such that for all $(x,z)\in Y\times Z_0$:
$$
\widetilde{e}_{x,z} := \widetilde{e}(x,z,\cdot) = \Psi(x,z,F_{x,z}).
$$
Define $\widetilde{N}=Y\times Z_0\times \Sigma\longrightarrow TM$ such that for all $(x,z)\in Y\times Z_0$, $\widetilde{N}_{x,z}:=\widetilde{N}(x,z,\cdot)$ is the unit, normal vector field over $\widetilde{e}_{x,z}$ with respect to $\widetilde{g}_{x,z}$ which is compatible with the orientation. Recalling Section \ref{Conjugations}, we define $\widetilde{\lambda}:Y\times TZ_0\times \Sigma\rightarrow\Bbb{R}$ such that for all $(x,z)\in Y\times Z_0$ and for all $\xi_z\in T_zZ_0$:
$$
\widetilde{\lambda}_{x,z}(\xi_z):=\widetilde{\lambda}(x,z,\xi_z,\cdot) = \widetilde{g}_z((D_2\widetilde{e})_{x,z}(\xi_z),\widetilde{N}_{x,z}),
$$
where $D_2\widetilde{e}$ is the partial derivative of $\widetilde{e}$ with respect to the second component in $Y\times Z_0\times \Sigma$. Observe that $\widetilde{\lambda}_{x,z}$ defines a linear mapping from $T_zZ_0$ to $C^\infty(\Sigma)$. We define $\Cal{A}:Y\times Z_0\rightarrow\Bbb{R}$ such that, for all $(x,z)\in Y\times Z_0$:
$$
\Cal{A}(x,z) = \opVol(\widetilde{e}_{x,z}) = \int_\Sigma \opdVol_{x,z},
$$
where $\opdVol_{x,z}$ is the volume form of $\widetilde{e}_{x,z}^*\widetilde{g}_{x,z}$. For all $x\in Y$, we denote $\Cal{A}_x:=\Cal{A}(x,\cdot)$.
\begin{proposition}\label{PropFormulaForDerivativeOfArea}
For all $(x,z)\in Y\times Z_0$ and for all $\xi_z\in T_zZ$:
$$
d\Cal{A}_x(\xi_z) = \int_\Sigma \widetilde{h}_{x,z}\widetilde{\lambda}_{x,z}(\xi_z)\opdVol_{x,z}.
$$
\end{proposition}
\begin{proof}
This follows from the definitions of $\widetilde{h}$, $\widetilde{\lambda}$ and the first variation formula for area (c.f. Section \ref{Conjugations} and Section $1.1$ of \cite{ColdingMinicozzi}).
\end{proof}

\begin{proposition}\label{PropLambdaLiesInKernel}
Upon reducing $Y$ if necessary, for all $(x,z)\in Y\times Z_0$, the pairing:
$$
T_zZ_0\times\opKer(\text{\rm J}_{g_0,e_z})\longrightarrow\Bbb{R};(\xi_z,\varphi)\mapsto \int_\Sigma\varphi\widetilde{\lambda}_{x,z}(\xi_z)\opdVol_{x,z},
$$
is non-degenerate.
\end{proposition}
\begin{proof}
Choose $z\in Z_0$. There exists a neighbourhood $\Omega$ of $z$ in $Z_0$ and smooth mappings $\alpha:\Omega\times \Sigma\longrightarrow \Sigma$ and $\psi:\Omega\times \Sigma\rightarrow\Bbb{R}$ such that $\alpha(z,\cdot)$ coincides with the identity mapping, $\psi(z,\cdot)=0$, and for all $w\in\Omega$, $\alpha_w:=\alpha(w,\cdot)$ is a smooth diffeomorphism of $\Sigma$ and $\Psi(0,z,\psi_w)\circ\alpha_w = \tilde{e}_{x_0,w}$, where $\psi_w:=\psi(w,\cdot)$. In particular, for all $w\in\Omega$, $(H,\Theta)(x_0,z,\psi_w)=0$, and so, for all $\xi_z\in T_zZ$:
\begin{equation}\label{EqnNonDegeneratePairing}
\text{\rm J}_{g,e_z}((D_1\psi)_z(\xi_z))=D_3(H,\Theta)(x_0,z,0)((D_1\psi)_z(\xi_z))= 0.
\end{equation}
However, as in the proof of Proposition \ref{PropInfMCIsConjugate}, $(D_1\psi)_z(\xi_z)=\widetilde{\lambda}_{x_0,z}(\xi_z)$, from which it follows that $\widetilde{\lambda}_{x_0,z}$ maps $T_zZ_0$ into $\opKer(\text{\rm J}_{g_0,e_z})$. Moreover, since $\Cal{F}$ is an immersion, $\widetilde{\lambda}_{x_0,z}$ is injective for all $z\in Z_0$, and since $\Cal{F}$ is non-degenerate, $\opDim(TZ_0)=\opDim(\opKer(\text{\rm J}_{g_0,e_z}))$. It follows that this mapping is a linear isomorphism and the pairing \eqref{EqnNonDegeneratePairing} is therefore non-degenerate. Finally, since $Z_0$ is compact, upon reducing $Y$ if necessary, the pairing \eqref{EqnNonDegeneratePairing} is also non-degenerate for all $(x,z)\in Y\times Z_0$, and this completes the proof.
\end{proof}
\begin{proposition}\label{PropExtremalWheneverDAVanishes}
For all $(x,z)\in Y\times Z_0$, $\widetilde{h}_{x,z}=0$ if and only if $d\Cal{A}_x(z)=0$.
\end{proposition}
\begin{proof}
By Proposition \ref{PropFormulaForDerivativeOfArea}, if $\tilde{h}_{x,z}=0$, then $d\Cal{A}_x(\xi_z)=0$. Conversely, if $d\Cal{A}_x(z)=0$, then, for all $\xi_z\in T_zZ_0$:
$$
\int_\Sigma\widetilde{h}_{x,z}\widetilde{\lambda}_{x,z}(\xi_z)\,\opdVol_{x,z} = 0,
$$
and it follows from Proposition \ref{PropLambdaLiesInKernel} that $\widetilde{h}_{y,z}=0$, as desired.
\end{proof}
\begin{proposition}\label{PropExtensionYieldsNonDegenerateSectionII}
There exists a compact neighbourhood $Y$ of $x_0$ in $X$, a smooth mapping $\widetilde{\Cal{F}}:Y\times Z\longrightarrow\Cal{E}$ and smooth family of sections $\sigma:Y\times Z\rightarrow T^*Z$ such that:
\begin{enumerate}
\item the restriction of $\tilde{\Cal{F}}$ to $\left\{x_0\right\}\times Z$ coincides with $\Cal{F}$; and
\item for all $(y,z)\in Y\times Z$, $(y,\tilde{\Cal{F}}(y,z))$ is an element of $\Cal{Z}(Y\times Z)$ if and only if $\sigma(y,z)=0$.
\end{enumerate}
\end{proposition}
\begin{remark}Importantly, in the variational context studied here, for all $y\in Y$, $\sigma_y:=\sigma(y,\cdot)$ is the derivative of the area functional.
\end{remark}
\begin{proof}
Since $Z$ is compact, there exists a finite family $(Z_i,e_i)_{1\leqslant i\leqslant m}$ of local parametrisations of $(Z,\Cal{F})$ which covers $Z$. Choose $1\leqslant i\leqslant m$. Let $Y_i\subseteq X$ and $F_i:Y_i\times Z_i\longrightarrow C^\infty(\Sigma)$ be as in Propositions \ref{PropLocalPerturbationSections} and \ref{PropLambdaLiesInKernel}. Define $\widetilde{e}_i:Y_i\times Z_i\times \Sigma\longrightarrow M$ and $\widetilde{h}_i:Y_i\times Z_i\times \Sigma\rightarrow\Bbb{R}$ as above. Define $\widetilde{\Cal{F}}_i:Y_i\times Z_i\rightarrow\Cal{E}$ by $\widetilde{\Cal{F}}_i(x,z) = [\widetilde{e}_{i,x,z}]$, define $\Cal{A}_i:Y_i\times Z_i\longrightarrow\Bbb{R}$ by $\Cal{A}_i(x,z)=\opVol(\widetilde{e}_{x,z})$ and define $\sigma_i:Y_i\times Z_i\rightarrow T^*Z_i$ by $\sigma_i(x,z)=d\Cal{A}_{i,x}(z)$, where $\Cal{A}_{i,x}=\Cal{A}_i(x,\cdot)$.

Denote $Y=Y_1\minter...\minter Y_m$. Choose $z\in Z_i\minter Z_j$. Since $[e_{i,z}]=\Cal{F}(z)=[e_{j,z}]$, there exists a smooth, orientation-preserving diffeomorphism $\alpha:\Sigma\longrightarrow \Sigma$ such that $e_{i,z}\circ\alpha = e_{j,z}$. By uniqueness, it follows that for all $x\in Y$, $F_{i,x,z}\circ\alpha = F_{j,x,z}$. In particular $\widetilde{e}_{i,x,z}\circ\alpha =\widetilde{e}_{j,x,z}$, from which it follows that $\widetilde{\Cal{F}}_i(x,z)=\widetilde{\Cal{F}}_j(x,z)$. We thus define $\widetilde{\Cal{F}}:Y\times Z\rightarrow\Cal{E}$ such that, for all $i$ and for all $(x,z)\in Y\times Z_i$, $\widetilde{\Cal{F}}(x,z)=\widetilde{\Cal{F}}_i(x,z)$, and it follows from the above discussion that $\widetilde{\Cal{F}}$ is smooth. We define $\sigma:Y\times Z\longrightarrow T^*Z$ such that for all $i$ and for all $(x,z)\in Y\times Z_i$, $\sigma(x,z)=\sigma_i(x,z)$, and we show in a similar manner that $\sigma$ is also smooth.

Finally, choose $1\leqslant i\leqslant m$ and choose $(x,z)\in Y\times Z_i$. By definition, the mean curvature of $\widetilde{e}_{i,x,z}$ is equal to $\widetilde{h}_{i,x,z}$. In particular, $\widetilde{e}_{i,x,z}$ is free boundary minimal if and only if $\widetilde{h}_{i,x,z}$ vanishes. However, by Proposition \ref{PropExtremalWheneverDAVanishes}, $\widetilde{h}_{i,x,z}$ vanishes if and only if $d\Cal{A}_{i,x}(z)=\sigma_i(x,z)$ vanishes. $\Cal{F}(x,z)$ is therefore free boundary minimal if and only if $\sigma(x,z)$ vanishes, and this completes the proof.
\end{proof}

\subsection{Non-degenerate sections}\label{NonDegenerateSections}We briefly consider the following general result for sections of bundles over finite-dimensional manifolds. Let $N_1$ and $N_2$ be two Riemannian manifolds, let $V$ be a smooth vector bundle over $N_2$ and let $\sigma:N_1\times N_2\longrightarrow V$ be a smooth family of sections of $V$ parametrised by $N_1$. We say that $\sigma$ is {\it non-degenerate} whenever $D_1\sigma(p,q)$ defines a surjective map from $T_pN_1$ onto $V_qN_2$ for all $(p,q)\in\sigma^{-1}(\left\{0\right\})$. Non-degenerate families of sections are of interest due to the following result:
\begin{proposition}\label{PropNonDegenerateSectionsOfVectorBundles}
If $\sigma:N_1\times N_2\longrightarrow V$ is a non-degenerate family of sections, then $W:=\sigma^{-1}(\left\{0\right\})$ is a smooth, embedded submanifold of $N_1\times N_2$ of dimension equal to $\opDim(N_1)+\opDim(N_2)-\opDim(V)$. Moreover, if $N_2$ is compact, then there exists an open, dense subset $N_1^0\subseteq N_1$ such that for all $p\in N_1^0$, every zero of the section $\sigma_p:=\sigma(p,\cdot)$ is non-degenerate.
\end{proposition}
\begin{proof}
The first assertion follows from the implicit function theorem. Let $\pi:N_1\times N_2\rightarrow N_1$ be the canonical projection onto the first factor. Let $N_1^0\subseteq N_1$ be the set of regular values of the restriction of $\pi$ to $W$. By Sard's Theorem, $N_1^0$ is a dense subset of $N_1$, and by compactness of $N_2$ it is open. Choose $p\in N_1^0$. We claim that every zero of the section $\sigma_p$ is non-degenerate. Indeed, let $q$ be a zero of $\sigma_p$. Upon trivialising $V$, we consider $\sigma$ as a smooth mapping from $N_1\times N_2$ into $\Bbb{R}^m$, where $m=\opDim(V)$. We claim that $D\sigma_p(q)$ is surjective. Indeed, choose $\xi\in\Bbb{R}^m$. Since $\sigma$ is non-degenerate, there exists $\alpha\in T_pN_1$ such that $(D_1\sigma)(p,q)(\alpha)=\xi$. Since $p$ is a regular value of the restriction of $\pi$ to $W$, there exists $\beta\in T_qN_2$ such that $(-\alpha,\beta)\in T_{(p,q)}W$. In particular, $D\sigma(p,q)(-\alpha,\beta)=0$. Taking the sum of these two relations yields:
$$
D\sigma_p(q)(\beta) = (D_2\sigma)(p,q)(\beta) = (D\sigma)(p,q)((\alpha,0)+(-\alpha,\beta)) = \xi,
$$
from which it follows that $D\sigma_p(q)$ is surjective, as asserted, and we conclude that every zero of the section $\sigma_p:=\sigma(p,\cdot)$ is non-degenerate, as desired.
\end{proof}

In the present framework, we have the following result:
\begin{proposition}\label{PropExtensionYieldsNonDegenerateSection}
There exists an extension $\widetilde{X}$ of $X$ and a compact neighbourhood $Y$ of $x_0$ in $\widetilde{X}$ with the property that if $\widetilde{h}:Y\times Z_0\times S\rightarrow\Bbb{R}$ is defined as in the preceeding section, then $\widetilde{h}$ defines a non-degenerate family of sections of $\Cal{K}$ over $Z_0$ parametrised by $Y$.
\end{proposition}
\begin{proof}
We define the mapping $\widetilde{g}:C^\infty(M)\times X\times M\rightarrow\opSym^+(TM)$ such that, for all $\varphi\in C^\infty(M)$ and for all $x\in X$:
$$
\widetilde{g}_{\varphi,x} := \widetilde{g}(\varphi,x,\cdot) = e^\varphi g_x.
$$
Let $E$ be a finite-dimensional, linear subspace of $C^\infty(M)$, and for $r>0$, let $E_r$ be the closed ball of radius $r$ about $0$ in $E$ with respect to some metric. Since $X$ is compact, for sufficiently small $r$, and for all $(\varphi,x)\in E_r\times X$, the metric $g_{\varphi,x}$ is admissable. We denote $\widetilde{X}=E_r\times X$ and we will show that $\widetilde{X}$ has the desired properties for suitable choices of $E$ and $r$.

Choose $z\in Z_0$. Let $\psi_1,...,\psi_m$ be a basis of $\opKer(\text{\rm J}_{g_0,e_z})$. Let $\varphi_1,...,\varphi_m\in C^\infty(M)$ be as in Proposition \ref{PropMetricsWhichYieldDesiredFunctions} with $U=M$ and let $E_z$ be the linear span of $\varphi_1,...,\varphi_m$ in $C^\infty(M)$.

Let $Y_z\subseteq E_{z,r}\times X$ and $f_z:Y_z\times Z_0\times \Sigma\longrightarrow\Bbb{R}$ be respectively a compact neighbourhood of $(0,x_0)$ and a smooth function as in Proposition \ref{PropLocalPerturbationSections}. We define $\tilde{h}_z:Y_z\times Z_0\rightarrow C^\infty(\Sigma)$ by:
$$
\widetilde{h}_{z,(\varphi,x),w}:=\widetilde{h}_z((\varphi,x),w) = H((\varphi,x),w,f_{z,\varphi,x,w}).
$$

Let $D_1\widetilde{h}_z$ be the partial derivative of $\widetilde{h}_z$ with respect to the first component in $E_{z,r}\times X\times Z_0$. We claim that $(D_1\widetilde{h}_z)_{(0,x_0),z}$ defines a surjective mapping from $E_z$ onto $\Cal{K}_{(0,x_0),z}$. We first show that for all $1\leqslant k\leqslant m$, $(D_1f_z)_{(0,x_0),z}(\varphi_k)=0$. Indeed, by definition, $f_{z,(0,x_0),z}=0$, and for all $\varphi\in E_r$, $f_{z,\varphi,x_0,z}\in\Cal{K}^\perp_z$. It thus follows upon differentiating that for all $1\leqslant k\leqslant m$, $(D_1f_z)_{(0,x_0),z}(\varphi_k)\in\Cal{K}^\perp_z=\opKer(\text{\rm J}_{g_0,e_z})^\perp$. However, differentiating the definition of $f_z$ yields:
$$
(\Pi_z\circ \text{\rm J}^h_{g_0,e_z},\text{\rm J}^\theta_{g_0,e_z})((D_1f_z)_{(0,x_0),z}(\varphi_k)) = 0,
$$
where $\Pi_z:C^\infty(\Sigma)\longrightarrow\Cal{K}_z^\perp$ is the orthogonal projection with respect to the $L^2$-inner-product of $e_z^*g_0$. As in the proof of Proposition \ref{PropLocalPerturbationSections}, the restriction of $(\Pi_z\circ \text{\rm J}^h_{g_0,e_z},\text{\rm J}^\theta_{g_0,e_z})$ to $\Cal{K}_z^\perp$ is injective, and so, for all $1\leqslant k\leqslant m$, $(D_1f_z)_{(0,x_0),z}(\varphi_k)=0$ as asserted. It now follows from the chain rule and by definition of $(\psi_k)_{1\leqslant k\leqslant m}$ and $(\varphi_k)_{1\leqslant k\leqslant m}$, that for all $1\leqslant k\leqslant m$:
$$
(D_1\widetilde{h}_z)_{(0,x_0),z}(\varphi_k) =\psi_k.
$$
Consequently, $\opIm((D_1\tilde{h}_z)_{(0,x_0),z})=\opKer(\text{\rm J}_{g_0,e_z})=\Cal{K}_{(0,x_0),z}$ and therefore $(D_1\tilde{h}_z)_{(0,x_0),z}$ defines a surjective mapping from $E_z$ onto $\Cal{K}_{(0,x_0),z}$ as asserted. Since surjectivity is an open property, there exists a neighbourhood $W$ of $z$ in $Z_0$ such that $(D_1\tilde{h}_z)_{(0,x_0),z}$ defines a surjective mapping from $E_z$ onto $\Cal{K}_{(0,x_0),w}$ for all $w\in W$. Observe that if $E$ contains $E_z$ then, by uniqueness, the restrictions of $f$ and $\tilde{h}$ to $E_z\times X\times Z\times S$ coincide with $f_z$ and $\tilde{h}_z$ respectively, and so $(D_1\tilde{h})_{(0,x_0),z}$ therefore also defines a surjective mapping from $E_z$ onto $\Cal{K}_{(0,x_0),w}$ for all $w\in W$. By compactness of $Z_0$, there exists a finite collection $z_1,...,z_m$ of points in $Z_0$ such that:
$$
Z_0 = \munion_{k=1}^m W_{z_k}.
$$
We denote $E=E_{z_1}+...+E_{z_k}$, and we see that $(D_1\widetilde{h})_{x_0,z}$ defines a surjective mapping from $T_{x_0}\widetilde{X}$ onto $\opKer(\text{\rm J}_{g_0,e_z})$ for all $z\in Z_0$. Since surjectivity is an open property, and since $Z_0$ is compact, there exists a compact neighbourhood $Y$ of $x_0$ in $\widetilde{X}$ such that $(D_1\widetilde{h})_{x,z}$ defines a surjective mapping from $T_x\widetilde{X}$ onto $\Cal{K}_{x,z}$ for all $(x,z)\in Y\times Z_0$ and $\widetilde{h}$ therefore defines a non-degenerate family of sections of $\Cal{K}$ over $Z_0$ parametrised by $Y$, as desired.
\end{proof}

\begin{proposition}\label{}
There exists an extension $\widetilde{X}$ of $X$ and a compact neighbourhood $Y$ of $x_0$ in $\widetilde{X}$ such that if $\sigma:Y\times Z\rightarrow T^*Z$ is defined as in Proposition \ref{PropExtensionYieldsNonDegenerateSectionII}, then $\sigma$ defines a non-degenerate family of sections of $T^*Z$ over $Z$ parametrised by $Y$.
\end{proposition}
\begin{proof}
We use the notation of the proof of Proposition \ref{PropExtensionYieldsNonDegenerateSectionII}. We denote $\widetilde{X}_0=X$. For $1\leqslant i\leqslant m$, having defined $\widetilde{X}_{i-1}$, we extend it to $\widetilde{X}_i$ so that it satisfies the conclusion of Proposition \ref{PropExtensionYieldsNonDegenerateSection} with $Z_0=Z_i$. We denote $\widetilde{X}=\widetilde{X}_m$. By compactness, for $1\leqslant i\leqslant m$, there exists a compact neighbourhood $Y_i$ of $x_0$ in $\widetilde{X}$ such that $\widetilde{X}$ satisfies the conclusion of Proposition \ref{PropExtensionYieldsNonDegenerateSection} with $Y=Y_i$ and $Z_0=Z_i$. We denote $Y=Y_1\minter ...\minter Y_m$.

Choose $1\leqslant i\leqslant m$. Choose $(x,z)\in Y\times Z_i$ such that $\sigma_i(x,z)=0$. By Proposition \ref{PropExtremalWheneverDAVanishes}, $\widetilde{h}_{i,x,z}=0$. Choose $\alpha\in T^*Z_i$. By Proposition \ref{PropLambdaLiesInKernel}, there exists $\psi\in\Cal{K}_{i,x,z}$ such that for all $\xi_z\in T_zZ_i$:
$$
\alpha(\xi_z) = \int_\Sigma\psi\widetilde{\lambda}_{i,x,z}(\xi_z)\,\opdVol_{i,x,z}.
$$
However, by definition of $\tilde{X}$, there exists $\eta_x\in T_x\tilde{X}$ such that $(D_1\tilde{h}_{i,x,z})(\eta_x)=\psi$. Thus, for all $\xi_z\in TZ_i$:
\begin{align*}
D_1\sigma_{i,x,z}(\eta_x)(\xi_z) &= \int_\Sigma D_1\widetilde{h}_{i,x,z}(\eta_x)\widetilde{\lambda}_{i,x,z}(\xi_z)\opdVol_{x,z}\\
&= \int_\Sigma\psi\widetilde{\lambda}_{i,x,z}(\xi_z)\,\opdVol_{x,z}\\
&=\alpha(\xi_z),
\end{align*}
and it follows that $D_1\sigma_{i,x,z}$ is surjective. $\sigma_i$ therefore defines a non-degenerate family of sections of $T^*Z_i$ over $Z_i$ parametrised by $Y$. Since $i$ is arbitrary, it follows that $\sigma$ defines a non-degenerate family of sections of $T^*Z$ over $Z$ parametrised by $Y$, and this completes the proof.
\end{proof}

\subsection{Determining the index}\label{DeterminingTheIndex}
The following result is proven in \cite{WhiteI}: %but we include a (different) proof for the reader's convenience:
\begin{lemma}\label{LemmaWhitesLemmaForNullities}
Let $A$ be an element of $\Cal{F}^+(E,F)$. Let $K\subseteq E$ be the kernel of $A$. There exists a neighbourhood $U$ of $A$ in $\Cal{F}^+(E,F)$ such that if $A'\in U$ and if $A'$ maps $K'$ into $K$ for some $K'\subseteq E$ of dimension equal to that of $K$, then:
$$
\opNull(A') = \opNull(A'|_{K'}),\qquad \opInd(A') = \opInd(A) + \opInd(A'|_{K'}),
$$
where $A'|_{K'}$ denotes the restriction of the bilinear form $\langle A'\cdot,\cdot\rangle$ to $K'$.
\end{lemma}

\begin{proposition}\label{PropJacobiOperatorsOfVanishingPoints}
For all $(x,z)\in Y\times Z_0$ such that $\sigma(x,z)=0$ and for all $\xi_z\in T_zZ$:
$$
(\text{\rm J}^h_{(x,z),\widetilde{e}_{x,z}}\circ\widetilde{\lambda}_{x,z})(\xi_z)\in\opKer(\text{\rm J}_{g_0,e_z}),
$$
\noindent and, for all $\xi_z,\eta_z\in T_zZ_0$:
$$
D\sigma_x(z)(\xi_z,\eta_z) = \int_\Sigma(\text{\rm J}^h_{(x,z),\widetilde{e}_{x,z}}\circ\widetilde{\lambda}_{x,z})(\xi_z)\widetilde{\lambda}_{x,z}(\eta_z)\,\opdVol_{x,z},
$$
where $\opdVol_{x,z}$ is the volume form of $\widetilde{e}^*_{x,z}\widetilde{g}_{x,z}$.
\end{proposition}
\begin{proof}
Since $\sigma(x,z)=0$, by Proposition \ref{PropExtremalWheneverDAVanishes}, $\widetilde{h}_{x,z}=0$. Thus, for all $\xi_z\in T_zZ$, as in the proof of Proposition \ref{PropInfMCIsConjugate}:
$$
(D_2\widetilde{h})_{x,z}(\xi_z) = (\text{\rm J}^h_{(x,z),\widetilde{e}_{x,z}}\circ\widetilde{\lambda}_{x,z})(\xi_z),
$$
from which it follows that for all $\xi_z,\eta_z\in T_zZ_0$:
\begin{align*}
D\sigma_x(z)(\xi_z,\eta_z) &=\int_\Sigma(D_2\widetilde{h})_{x,z}(\xi_z)\widetilde{\lambda}_{x,z}(\eta_z)\,\opdVol_{x,z}\\
&= \int_\Sigma(\text{\rm J}^h_{(x,z),\widetilde{e}_{x,z}}\circ\widetilde{\lambda}_{x,z})(\xi_z)\widetilde{\lambda}_{x,z}(\eta_z)\,\opdVol_{x,z},
\end{align*}
and the second result follows. Moreover, by definition, for all $(x,z)\in Y\times Z$, $\widetilde{h}_{x,z}$ is an element of $\opKer(\text{\rm J}_{g_0,e_z})$. Thus, when $\widetilde{h}_{x,z}=0$:
$$
(\text{\rm J}^h_{(x,z),\widetilde{e}_{x,z}}\circ\lambda_{x,z})(\xi_z) = (D_2\widetilde{h})_{x,z}(\xi_z)\in\opKer(\text{\rm J}_{g_0,e_z}).
$$
The first result follows, and this completes the proof.
\end{proof}

Combining the above results yields:
\begin{theorem}\label{ThmNonDegenerateFamilies}
If $\Cal{Z}(\left\{x_0\right\})$ contains a closed, non-degenerate family $Z$, then there exists a neighbourhood $\Omega$ of $Z$ in $\Cal{E}$ such that:
$$
\Cal{Z}(\left\{x_0\right\})\minter\overline{\Omega} = Z.
$$
Moreover, for any such neighbourhood $\Omega$, there exists a compact neighbourhood $Y$ of $x_0$ in $X$ such that $\partial_\omega\Cal{Z}(Y|\Omega) = \emptyset$ and the local mapping degree of the restriction of $\Pi$ to $\Cal{Z}(Y|\Omega)$ is given by:
$$
\opDeg(\Pi|\Omega)=(-1)^{\opInd(Z_0)}\chi(Z_0),
$$
\noindent where $\opInd(Z_0)$ and $\chi(Z_0)$ are respectively the index and Euler characteristic of $Z_0$.
\end{theorem}

\begin{proof}
Let $\Cal{F}:Z\rightarrow\Cal{E}$ be the canonical embedding. By Theorem \ref{PropSurjectivityForExtensions}, there exists an extension $\widetilde{X}_1$ of $X$ such that, for all $(x,[e])\in\Cal{Z}(X)$, $\text{\rm P}_{x,e}+\text{\rm J}_{x,e}$ defines a surjective map from $T_x\widetilde{X}_1\times C^\infty(\Sigma)$ onto $C^\infty(\Sigma)\times C^\infty(\partial \Sigma)$. Let $\widetilde{X}$ be a further extension of $\widetilde{X}_1$ satisfying the conclusion of Proposition \ref{PropExtensionYieldsNonDegenerateSectionII}. By Proposition \ref{ThmCompactness}, $\Cal{Z}(X)$ is compact and so by Proposition \ref{PropSurjectivityIsAnOpenProperty}, there exists a compact neighbourhood, $\widetilde{Y}$ of $X$ in $\widetilde{X}$ such that $\text{\rm P}_{x,e}+\text{\rm J}_{x,e}$ defines a surjective mapping from $T_x\widetilde{X}\times C^\infty(\Sigma)$ onto $C^\infty(\Sigma)\times C^\infty(\partial \Sigma)$ for all $(x,[e])\in\Cal{Z}(\widetilde{Y})$. Thus, by Theorem \ref{ThmSolutionSpaceIsSmooth}, $\Cal{Z}(\widetilde{Y})$ is a smooth, compact, finite-dimensional manifold of dimension equal to $\opDim(\widetilde{X})$. Observe, in particular, that by Proposition \ref{PropSmoothnessToSmoothness}, $\Cal{F}$ defines a smooth map from $Z_0$ into $\Cal{Z}(\widetilde{Y})$ and since $(Z,\Cal{F})$ is non-degenerate, this mapping is an embedding.

Upon reducing $\tilde{Y}$ if necessary, there exists a smooth mapping $\widetilde{\Cal{F}}:\widetilde{Y}\times Z\longrightarrow\Cal{E}$ and a smooth, non-degenerate family of sections $\sigma:\widetilde{Y}\times Z\longrightarrow T^*Z$ satisfying the conclusion of Proposition \ref{PropExtensionYieldsNonDegenerateSectionII}. We define $W\subseteq \widetilde{Y}\times Z$ by $W=\sigma^{-1}(\left\{0\right\})$. Since $\sigma$ is a non-degenerate family, by Proposition \ref{PropNonDegenerateSectionsOfVectorBundles}, $W$ is a smooth, embedded submanifold of $Y\times Z$ of dimension equal to $\opDim(\widetilde{Y})=\opDim(\widetilde{X})$. We define $\tilde{G}:W\longrightarrow\tilde{Y}\times\Cal{E}$ such that for all $(y,z)\in W$:
$$
\tilde{G}(y,z) = (y,\widetilde{\Cal{F}}(y,z)).
$$
By definition, $\widetilde{G}(W)\subseteq\Cal{Z}(\widetilde{Y})$. Moreover, $\widetilde{G}$ defines a smooth mapping from $W$ into $\Cal{Z}(Y)$.

Choose $z\in Z$. We claim that $D\tilde{G}(x_0,z)$ is a linear isomorphism. Indeed, choose $(\xi_{x_0},\eta_z)\in T_{x_0}\tilde{X}\times T_zZ$ such that $D\tilde{G}(x_0,z)(\xi_{x_0},\eta_z)=0$. Let $\pi_1:\widetilde{X}\times\Cal{E}\longrightarrow\widetilde{X}$ be the projection onto the first factor. Then, bearing in mind the chain rule:
$$
\xi_{x_0} = D(\pi_1\circ\widetilde{G})(x_0,z)(\xi_{x_0},\eta_z) = 0.
$$
In particular, since the restriction of $\widetilde{\Cal{F}}$ to $\left\{x_0\right\}\times Z$ coincides with $\Cal{F}$:
$$
(0,D\Cal{F}(z)(\eta_z)) = D\widetilde{G}(x_0,z)(0,\eta_z) =0,
$$
and since $\Cal{F}$ is an embedding, it follows that $\eta_z=0$ and $D\tilde{G}(x_0,z)$ is therefore a linear isomorphism as asserted. Upon reducing $\tilde{Y}$ if necessary, we may assume that $D\tilde{G}$ is a linear isomorphism for all $(x,z)\in\tilde{Y}\times Z$. In particular, $\tilde{G}$ is an open mapping.

Observe that $\tilde{G}(W)$ is an open subset of $\Cal{Z}(\tilde{Y})$. Thus, since:
$$
Z = (\left\{x_0\right\}\times\Cal{E})\minter\tilde{G}(W),
$$
it follows that $Z$ is an isolated subset of $\Cal{Z}(\left\{x_0\right\})=(\left\{x_0\right\}\times\Cal{E})\minter\Cal{Z}(\tilde{Y})$. In particular, there exists a neighbourhood $\Omega$ of $Z$ in $\Cal{E}$ such that:
$$
Z = \Cal{Z}(\left\{x_0\right\})\minter\overline{\Omega},
$$
and the first assertion follows.

Since $\tilde{G}$ is a local diffeomorphism, since its restriction to $\left\{x_0\right\}\times Z$ coincides with $\Cal{F}$, which is a diffeomorphism, and since $Z$ is compact, upon reducing $\tilde{Y}$ further if necessary, we may assume that $\tilde{G}$ is also a diffeomorphism onto its image. By continuity, we may suppose furthermore that $G(\tilde{Y})\subseteq \tilde{Y}\times\Omega$. In particular:
$$
G(\tilde{Y})\subseteq\Cal{Z}(\tilde{Y}|\Omega).
$$
Conversely, by Proposition \ref{ThmCompactness}, upon reducing $\tilde{Y}$ yet further if necessary, we may suppose that:
$$
\Cal{Z}(\tilde{Y}|\Omega) = \Cal{Z}(\tilde{Y})\minter(\tilde{Y}\times\Omega)\subseteq\tilde{G}(W),
$$
and so:
$$
\tilde{G}(W) = \Cal{Z}(\tilde{Y}|\Omega).
$$
For all $y\in\tilde{Y}$:
$$
\Cal{Z}(\left\{y\right\}|\Omega) = \tilde{G}(\left\{(y,z)\ |\ z\in\sigma^{-1}_y(\left\{0\right\})\right\}).
$$
Since $\sigma$ is a non-degenerate family, it follows from Proposition \ref{PropNonDegenerateSectionsOfVectorBundles} that there exists an open, dense subset $\widetilde{Y}_0\subseteq \widetilde{Y}$ such that for all $y\in \tilde{Y}_0$, the zeroes of the section $d\Cal{A}_y=\sigma_y$ are non-degenerate. Choose such a $y$. We claim that $y$ is also a regular value of the restriction of $\Pi$ to $\Cal{Z}(\widetilde{Y}|\Omega)$. By Proposition \ref{PropRegularValuesHaveInvertibleJacobiOperator} it suffices to show that $\text{\rm J}_{(y,z),\tilde{e}_{y,z}}$ is invertible for all $z\in\sigma_y^{-1}(\left\{0\right\})$. However, choose $z\in\sigma_y^{-1}(\left\{0\right\})$. By Lemma \ref{LemmaWhitesLemmaForNullities} and Proposition \ref{PropJacobiOperatorsOfVanishingPoints}, upon reducing $\widetilde{Y}$ if necessary:
$$
\opNull(\text{\rm J}_{(y,z),\widetilde{e}_{y,z}}) = \opNull(\text{\rm J}_{(y,z),\widetilde{e}_{y,z}}|_{E_{y,z}}),
$$
where $E_{y,z}=\left\{\widetilde{\lambda}_{y,z}(\xi_z)\ |\ \xi_z\in T_zZ_0\right\}$. However, by Proposition \ref{PropJacobiOperatorsOfVanishingPoints}, bearing in mind that the critical points of $\Cal{A}_y$ are non-degenerate:
\begin{equation}
\opNull(\text{\rm J}_{\widetilde{g}_{y,z},\widetilde{e}_{y,z}}|_{E_{y,z}}) = \opNull(\opHess(\Cal{A}_y)(z)) = 0,
\end{equation}
and $y$ is therefore a regular value of the restriction of $\Pi$ to $\Cal{Z}(\widetilde{Y}|\Omega)$, as asserted.

By Lemma \ref{LemmaWhitesLemmaForNullities} and Proposition \ref{PropJacobiOperatorsOfVanishingPoints} again:
$$
\opInd(\text{\rm J}_{(y,z),\widetilde{e}_{y,z}}) = \opInd(\text{\rm J}_{g_0,e_z}) + \opInd(\text{\rm J}_{(y,z),\widetilde{e}_{y,z}}|_{E_{y,z}}).
$$
Thus, bearing in mind the definition of $\opInd(Z_0)$:
$$
\opInd(\text{\rm J}_{(y,z),\widetilde{e}_{y,z}}) = \opInd(Z_0) + \opInd(\opHess(\Cal{A}_y)(z)).
$$
Thus by Theorem \ref{ThmIntegerValuedDegree}, the mapping degree of the restriction of $\Pi$ to $\Cal{Z}(\tilde{Y}|\Omega)$ is given by:
\begin{align*}
\opDeg(\Pi|\Omega) & =\sum_{(y,[e])\in\Cal{Z}(\left\{y\right\}|\Omega)}\opSig(\text{\rm J}_{y,e})\\
&= \sum_{z\in\sigma_y^{-1}(\left\{0\right\})}\opSig(\text{\rm J}_{(y,z),\widetilde{e}_{y,z}})\\
&=(-1)^{\opInd(Z_0)}\sum_{z\in\sigma_y^{-1}(\left\{0\right\})}\opSig(\opHess(\Cal{A}_y)(z))\\
&=(-1)^{\opInd(Z_0)}\chi(Z_0),
\end{align*}
where $\chi(Z_0)$ is the Euler characteristic of $Z_0$ and the last equality follows from classical Morse theory.
\end{proof}

\section{Free boundary minimal surfaces inside convex domains}
%\label{CalculatingTheDegree}

\subsection{Rotationally invariant free boundary minimal surfaces}\label{RotationallyInvariantExtremalSurfaces}
Let $\delta$ be the Euclidean metric over $\Bbb{R}^3$ and let $B:=B^3\subset\mathbb{R}^3$ be the unit Euclidean three-ball. In order to apply degree theoretic techniques, it is preferrable to work with metrics of strictly positive curvature. For $-1<t<1$ and $t\neq 0$, let $\mathbb{S}^3(t)\subset \mathbb{R}^4$ be the sphere of radius $r(t)=1/\left|t\right|$ centered at $c(t)=(0,0,0,-1/t)$. For $t\neq 0$, we define $\varphi_t:B\rightarrow\mathbb{S}^3(t)$ by:
\begin{equation}\label{eq1}
\varphi_t(x) = (x,-1/t + \opsgn(t)\sqrt{t^{-2} - \|x\|^2}),
\end{equation}
where $\opsgn(t)$ is the sign of $t$. Observe that $\varphi$ extends to a smooth mapping from $]-1,1[\times B$ into $\mathbb{R}^4$ with $\varphi_0(x):=\varphi(0,x)=(x,0)$. For all $t$, denoting by $\delta$ the Euclidean metric, the induced metric $g_t=\varphi_t^*\delta$ on $B$ at the point $x\in B$ is given by:
\begin{equation}\label{eq2}
g_t(x) = \delta + \frac{t^2}{1 - t^2\|x\|^2}x\otimes x,
\end{equation}
so that, for all $-1<t<1$ and $t\neq 0$, $g_t$ is the metric of a spherical cap of radius $1/\left|t\right|$ and $g_0$ is the Euclidean metric. In particular, for all $t\in]-1,1[$, $g_t$ has positive constant sectional curvature equal to $t^2$. Observe, moreover, that $(B,g_t)$ is functionally strictly convex for all $t\in]-1,1[$.

\begin{remark}\label{foliation}
Given a unit vector $v\in\mathbb{R}^3$, we define the standard foliation $\{\mathcal{C}_s\}_{s\in(-1,1)}$ of $\partial B\setminus\{v,-v\}$ by $\mathcal{C}_s=\{w\in\partial B:\langle v,w\rangle_\delta=s\}$. For all $t\in]-1,1[$, we define the standard foliation $\{\mathcal{D}_{s,t}\}_{s\in(-1,1)}$ of $B\setminus\{v,-v\}$ so that for all $s$, $\mathcal{D}_{s,t}\subset B$ is the properly embedded disk which is totally geodesic with respect to  $g_t$ such that $\partial \mathcal{D}_{s,t}=\mathcal{C}_s$. Observe that, for all $s$, $\mathcal{D}_{s,0}=\{w\in B:\langle v,w\rangle_\delta=s\}$.
\end{remark}

For every unit vector $v$ in $\Bbb{R}^3$ and for all $\theta\in\Bbb{R}$, we define $R_{v,\theta}\in \opSO(3)$ to be the rotation about $v$ by $\theta$ radians in the positive direction (with respect to the canonical orientation of $\Bbb{R}^3$). In this section, we consider embedded surfaces in $B$ mainly as subsets of $B$ (rather than as equivalence classes of embeddings). We recall that an embedded surface $\Sigma\subseteq B$ is said to be {\it invariant by rotation} about $v$ whenever:
$$
R_{v,\theta}(\Sigma) = \Sigma,
$$
for all $\theta\in\Bbb{R}$. For $f:\Bbb{R}\rightarrow]0,\infty[$ be a positive function, recall that the surface of revolution of $f$ about $v$ is defined by:
$$
\Sigma_{v,f} = \left\{ R_{v,\theta}(tv + f(t)w)\ |\ \theta,t\in\Bbb{R}\right\},
$$
where $w\in\Bbb{R}^3$ is any unit vector orthogonal to $v$.

\begin{proposition}\label{PropUniquenessOfRotationalSurfaces}
For every unit vector $v\in\mathbb{R}^3$, the unique (unoriented) properly embedded free boundary minimal surfaces in $(B,\delta)$ which are  invariant under rotation about $v$ are:
\begin{itemize}
\item[(1)] the disk obtained by intersecting $B$ with the equatorial plane normal to $v$; and
\item[(2)] the annulus obtained by intersecting $B$ with the catenoid $\Sigma_{v,f}$, where $f(t)=r_0^{-1}\cosh(r_0t)$, $r_0=t_0\cosh(t_0)$ and $t_0>0$ is the unique positive solution of $t_0=\coth(t_0)$.
\end{itemize}
\end{proposition}
\begin{remark} An elementary calculation shows that $r_0>t_0>1$.
\end{remark}
\begin{proof} Consider the foliation of $\Bbb{R}^3$ by lines parallel to $v$. Let $\Sigma\subseteq (B,\delta)$ be a properly embedded free boundary minimal surface. If $\Sigma$ is normal to this foliation at every point, then $\Sigma$ is the intersection of $B$ with a plane normal to $v$. Since $\Sigma$ meets $\partial B$ orthogonally along $\partial\Sigma$, it follows that $\Sigma$ coincides with the intersection of the equatorial plane normal to $v$ with $B$, which yields Case $(1)$. Otherwise, it follows by Example $5$ of Section $3.5$ of \cite{doCarmo} that $\Sigma$ is the surface of revolution about $v$ of the function $f(t)=a^{-1}\cosh(at + b)$ for some $a>0$ and for some $b\in\Bbb{R}$. Since $\Sigma$ meets $\partial B$ orthogonally along $\partial\Sigma$,  an elementary calculation shows that $\alpha=r_0$ and $b=0$, as desired.
\end{proof}

\begin{proposition}\label{PropExtremalDisk}
\noindent For all $t\neq 0$ and for every vector $v\in\Bbb{R}^3$, the unique (unoriented) properly embedded free boundary minimal disk in $(B,g_t)$ which is invariant under rotation about $v$ is the disk obtained by intersecting with the equatorial Euclidean plane normal to $v$.
\end{proposition}
\begin{proof} Choose $t\neq 0$ and let $\Sigma$ be an properly embedded free boundary minimal disk in $(B,g_t)$. Suppose that $\Sigma$ is invariant under rotation about $v$. It follows from this that $
\partial\Sigma$ is equal to $\mathcal{C}_s$ for some $s\in(-1,1)$, where $\{\mathcal{C}_s\}_{s\in(-1,1)}$ is the standard  foliation of $\partial B\setminus\{v,-v\}$ by spherical geodesic circles (c.f. Remark \ref{foliation}). Now consider the standard foliation $\{\mathcal{D}_{s,t}\}_{s\in(-1,1)}$ of $B\setminus\{v,-v\}$ by totally geodesic disk with respect to metric $g_t$ (c.f. Remark \ref{foliation}). There exists a leaf of this foliation which is an exterior tangent to $\Sigma$ at some point. By the geometric maximum principle, $\Sigma$ coincides with this leaf and since $\Sigma$ meets $\partial B$ orthogonally along $\partial\Sigma$ we conclude that $\Sigma=\mathcal{D}_{0,t}$, which is precisely the disk obtained by intersecting $B$ with the equatorial Euclidean plane normal to $v$.
\end{proof}

\begin{proposition}\label{PropExtremalCatenoid}
\noindent There exists $\delta>0$ such that, for all $t\in(-\delta,\delta)$ and for every vector $v\in\Bbb{R}^3$, there exists a unique (unoriented) properly embedded free boundary minimal surface in $(B,g_t)$ which is diffeomorphic to the annulus \linebreak$\Sph^1\times[0,1]$ and invariant under rotation about $v$.
\end{proposition}
\begin{proof} We first study the transversality properties of rotationally symmetric minimal surfaces in Euclidean space. We define $F:]0,\infty[\times\Bbb{R}\times\Bbb{R}\rightarrow\Bbb{R}$ by:
$$
F(a,b,s) = a^{-1}\opCosh(as + b).
$$
\noindent We note that for all $a,b$, the surface of revolution of $F_{a,b}:=F(a,b,\cdot)$ about $v$ is a catenoid which is a properly embedded minimal surface. We denote $\widehat{F}(a,b,s):=(s,F(a,b,s))$ and $\widehat{F}_{a,b}:=\widehat{F}(a,b,\cdot)$. We verify that $\widehat{F}$ is a submersion into $\Bbb{R}^2$. Let $\Sph^1$ be the unit circle in $\Bbb{R}^2$. Observe that $\hat{F}(r_0,0,\pm r_0t_0)\in \Sph^1$. By definition of $r_0$, the curve $\widehat{F}(r_0,0,\cdot)$ meets $\Sph^1$ orthogonally. In particular, it is transverse to $\Sph^1$. Let $\Omega$ be a neighbourhood of $(r_0,0)$ in $]0,\infty[\times\Bbb{R}$. By the implicit function theorem, for $\Omega$ sufficiently small, there exist smooth functions $G_\pm:\Omega\rightarrow\Bbb{R}$ such that for all $(a,b)\in\Omega$, $\widehat{F}(a,b,G_\pm(a,b))$ is an element of $C$.

Let $\nu$ be the outward-pointing unit normal vector field over $\Sph^1$. Let $N:\Omega\times\Bbb{R}\rightarrow\Bbb{R}^2$ be such that, for all $(a,b)\in\Omega$, $N_{a,b}:=N(a,b,\cdot)$ is a unit, normal vector field over the curve $\widehat{F}_{a,b}(\Bbb{R})$. We define $\Theta_\pm:\Omega\rightarrow\Bbb{R}$ such that, for all $(a,b)\in\Omega$, $\Theta_\pm$ is the angle that $\nu$ makes with $N_{a,b}$ at the point $F(a,b,G_\pm(a,b))$. Observe that $\partial_a\Theta_-(r_0,0)$ and $\partial_a\Theta_+(r_0,0)$ are both non-zero with the same sign, but that $\partial_b\Theta_-(r_0,0)$ and $\partial_b\Theta_+(r_0,0)$ are both non-zero with opposite signs. In particular, $\nabla\Theta_\pm(r_0,0)\neq 0$ and $\nabla\Theta_-(r_0,0)\neq \nabla\Theta_+(r_0,0)$. Thus, upon reducing $\Omega$ if necessary, $\Theta^{-1}_+(\left\{0\right\})$ and $\Theta^{-1}_-(\{0\})$ define smooth embedded  curves in $\Omega$ which intersect transversally at $(r_0,0)$.

We now return to metrics of non-zero curvature. Choose $\delta>0$ small and define $\widetilde{F}:\Omega\times(-\delta,\delta)\times\Bbb{R}\rightarrow\Bbb{R}$ such that, for all $(a,b,t)$, the surface of revolution of $\widetilde{F}_{a,b,t}:=\widetilde{F}(a,b,t,\cdot)$ about $v$ is minimal with respect to the metric $g_t$, and, moreover, $\widetilde{F}_{a,b,t}(-b/a)=a^{-1}$, $\widetilde{F}'_{a,b,t}(-b/a)=0$. Observe that, for all $(a,b,t)$, $\widetilde{F}_{a,b,t}$ is uniquely defined by a second-order nonlinear ODE. In particular, $\widetilde{F}_{a,b,0}=F_{a,b}$ for all $(a,b)\in\Omega$. It now follows by transversality that, upon reducing $\Omega$ and $\delta$ if necessary, for all $t\in(-\delta,\delta)$, there exists a unique point $(a(t),b(t))\in\Omega$ such that the curve $\widetilde{F}_{a(t),b(t),t}$ intersects $\Sph^1$ orthogonally with respect to the metric $g_t$. In particular, the surface of revolution of $\tilde{F}_{a(t),b(t),t}$ about $v$ is a properly embedded free boundary minimal annulus with respect to this metric, thus proving existence for sufficiently small $\delta$.

We now prove uniqueness. Indeed, suppose the contrary. Observe first that, by the uniqueness part of above discussion, if $\Sigma$ is a properly embedded minimal annulus in $(B,g_t)$ which is  invariant under rotation about $v$, and if $\Sigma$ is sufficiently close to the surface of revolution of $F_{r_0,0}$ about $v$ in the $C^1$ sense, then $\Sigma$ coincides with the surface of revolution of $\widetilde{F}_{a(t),b(t),t}$ about $v$. Now suppose there exists a sequence $(t_m)_\minn$ converging to $0$, and, for all $m$, two distinct (unoriented) properly embedded free boundary minimal annuli $\Sigma_m$ and $\Sigma'_m$ in $(B,g_{t_m})$ which are invariant under rotations about $v$. By Theorem \ref{ThmFraserLi}, we may suppose that $(\Sigma_m)_\minn$ and $(\Sigma'_m)_\minn$ both converge to $\Sigma_\infty$ and $\Sigma'_\infty$ respectively. By Proposition \ref{PropUniquenessOfRotationalSurfaces}, $\Sigma_\infty=\Sigma'_\infty$ is the surface of revolution of $F_{r_0,0,0}$ about $v$, and so, by the preceeding observation, for sufficiently large $m$, $\Sigma_m$ and $\Sigma'_m$ both coincide with the surface of revolution of $F_{a(t_m),b(t_m),t_m}$ about $v$. This is absurd, and uniqueness follows.
\end{proof}

\begin{proposition}\label{PropNoOtherNonDegenerateFamilies}
\noindent If $\Sigma$ is neither diffeomorphic to the disk $D$ nor to the annulus $\Sph^1\times[0,1]$ then there exists $\delta>0$ such that for all $t\in(-\delta,\delta)$, there exists no properly embedded free boundary minimal surface $(B,g_t)$ which is diffeomorphic to $\Sigma$ and invariant under rotation about $v$.
\end{proposition}
\begin{proof} Indeed, suppose the contrary. There exists a sequence $(t_m)_\minn$ converging to $0$, and, for all $m$, a properly embedded free boundary minimal surface $\Sigma_m$ in $(B,g_t)$ which is diffeomorphic to $S$ and invariant under rotation about some unit vector, $v_m$, say. Upon extracting a subsequence, we may suppose that $(v_m)_\minn$ converges to $v_\infty\in \Sph^2$, say. By Theorem \ref{ThmFraserLi}, upon extracting a further subsequence, we may suppose that $(\Sigma_m)_\minn$ converges to an embedded surface $\Sigma_\infty$ say. $\Sigma_\infty$ is a properly embedded free boundary minimal surface in $(B,\delta)$ which is diffeomorphic to $\Sigma$ and invariant under rotation about $v_\infty$. It thus follows from Proposition \ref{PropUniquenessOfRotationalSurfaces} that $\Sigma$ is diffeomorphic either to the disk $D$ or to the annulus $\Sph^1\times[0,1]$. This is absurd and the result follows.
\end{proof}

We henceforth refer to the embeddings constructed in Propositions \ref{PropExtremalDisk} and \ref{PropExtremalCatenoid} respectively as the {\it critical disk} and the {\it critical catenoid} of the metric $g_t$ with axis $v$.

\subsection{Non-degenerate families of disks}\label{SectionNonDegenerateFamiliesOfDisks} Let $e_1,e_2,e_3$ be the canonical basis of $\Bbb{R}^3$. We parametrise the critical disk of the Euclidean metric by the mapping $e_\opdisk:D\longrightarrow B$ given by:
$$
e_\opdisk(x,y) = (x,y,0).
$$
Let $\text{\rm J}_\opdisk:=(\text{\rm J}^h_{\opdisk},\text{\rm J}^\theta_{\opdisk})$ be the Jacobi operator of $e_\opdisk$ with respect to this metric.

\begin{proposition}
For all $\varphi\in C^\infty(D)$:
$$
\text{\rm J}^h_{\opdisk}\varphi = -\Delta\varphi,
$$
where $\Delta$ is the standard Laplacian of $\Bbb{R}^2$, and:
$$
\text{\rm J}^\theta_{\opdisk}\varphi = \varphi\circ\epsilon - \partial_\nu\varphi,
$$
where $\epsilon:\partial D\rightarrow D$ is the canonical embedding, and $\partial_\nu$ is the partial derivative in the outward-pointing, normal direction over $\partial D$.
\end{proposition}
\begin{proof} Observe that $e_\opdisk$ is a totally geodesic isometric embedding, and the result now follows by Propositions \ref{PropFormulaForJacobiOperatorOfMeanCurvature} and \ref{PropFormulaForJacobiOperatorOfBoundaryAngle}.
\end{proof}

\begin{proposition}\label{PropNullityOfExtremalDisk}
$\opKer(\text{\rm J}_{\opdisk})$ is $2$-dimensional.
\end{proposition}
\begin{proof} Choose $\varphi\in\opKer(\text{\rm J}_\opdisk)$. In particular, $\Delta\varphi=0$, and $\varphi$ is therefore the real part of a holomorphic function defined over $\overline{D}$. There therefore exists a sequence $(a_n)_\ninn\in\Bbb{C}$ such that for all $z\in D$:
$$
\varphi(z) = \opRe\left(\sum_{n=0}^\infty a_nz^n\right).
$$

By elliptic regularity, $\varphi\in C^\infty(\overline{D})$, and so, by classical Fourier analysis, the Taylor series of $\varphi$ and all its derivatives converge uniformly over $\partial D$. Since $\varphi$ satisfies the Robin condition $\text{\rm J}^\theta_{\opdisk}\varphi=0$, using the Cauchy-Riemann equations, we obtain, for all $\theta$:
$$
\opRe\left(\sum_{n=0}^\infty(1-n)a_n e^{in\theta}\right) = 0,
$$
from which it follows that $a_n=0$ for all $n\neq 1$. Consequently:
$$
\varphi(z) = \opRe(a_1 z) = \alpha x + \beta y,
$$
where $a_1 = \alpha -  i\beta$, and we conclude that $\opKer(\text{\rm J}_\opdisk)$ is $2$-dimensional, as desired.
\end{proof}

\begin{proposition}\label{NonDegenerateFamiliesOfDisks}
If $\Sigma=D$ is the disk, then there exists $\delta>0$ such that for all $t\in(-\delta,\delta)$, the family of embeddings $[e]\in\Cal{Z}(\left\{g_t\right\})$ which are invariant under rotation about some unit vector in $\Bbb{R}^3$ constitutes a non-degenerate family diffeomorphic to $\Sph^2$.
\end{proposition}
\begin{proof} We define $\Cal{I}_t:\Sph^2\rightarrow\Cal{Z}(\left\{g_t\right\})$ such that, for all $v\in \Sph^2$, $\Cal{I}_t(v)$ is the critical disk of the metric $g_t$ with axis $v$, oriented such that its normal coincides with $v$. We see that $\Cal{I}_t$ is a smooth embedding. By Proposition \ref{PropExtremalDisk}, $\Cal{I}_t(\Sph^2)$ accounts for all free boundary minimal embeddings in $\Cal{Z}(\left\{g_t\right\})$ which are invariant under rotation. By Proposition \ref{PropNullityOfExtremalDisk}, when $t=0$, the nullity of the Jacobi operator of $\Cal{I}(v)$ with respect to the metric $g_0$ is equal to $2$ for all $v\in \Sph^2$. Since the nullity is upper-semicontinuous, there exists $\delta>0$ such that for all $\left|t\right|<\delta$ and for all $v\in \Sph^2$, the nullity of the Jacobi operator of $\Cal{I}_t(v)$ with respect to the metric $g_t$ is at most $2$. Since $\Cal{I}_t$ is an embedding, by Proposition \ref{lowerboundondimensionofkernel}, the nullity of the Jacobi operator of $\Cal{I}_t(v)$ is also bounded below by the dimension of $\Sph^2$. It follows that the nullity of $\Cal{I}_t(v)$ with respect to the metric $g_t$ is equal to $2$, and we conclude that $\Cal{I}_t(\Sph^2)$ is a non-degenerate family, as desired.\end{proof}

\subsection{Non-degenerate families of catenoids}\label{NonDegenerateFamiliesOfCatenoids}  Let $t_0$ be as in Proposition \ref{PropUniquenessOfRotationalSurfaces}. We parametrise the critical catenoid with axis $e_3$ by the mapping $e_\opcat:[-t_0,t_0]\times \Sph^1\rightarrow \Bbb{R}^3$ given by:
$$
e_\opcat(t,\theta) = (r_0^{-1}\opCosh(t)\opCos(\theta),r_0^{-1}\opCosh(t)\opSin(\theta),r_0^{-1}t).
$$
Let $\text{\rm J}_\opcat=(\text{\rm J}^h_{\opcat},\text{\rm J}^\theta_{\opcat})$ be the Jacobi operator of $e_\opcat$ with respect to the Euclidean metric.

\begin{proposition}\label{PropFormulaForJacobiOperatorOfExtremalAnnulus}
For all $\varphi\in C^\infty([-t_0,t_0]\times \Sph^1)$ and for all $(t,\theta)\in\Bbb{R}\times \Sph^1$:
$$
(\text{\rm J}^h_{\opcat}\varphi)(t,\theta) = -\frac{2r_0^2}{\opCosh^4(t)}\varphi(t,\theta)-\frac{r_0^2}{\opCosh^2(t)}(\Delta\varphi)(t,\theta),
$$
where $\Delta$ is the standard Laplacian of $\Bbb{R}\times \Sph^1$, and, for all $\theta\in \Sph^1$:
$$
(\text{\rm J}^\theta_{\opcat}\varphi)(\pm t_0,\theta) = \varphi(\pm t_0,\theta) \mp t_0(\partial_t\varphi)(\pm t_0,\theta)
$$
\end{proposition}
\begin{proof} Observe that the parametrisation $e_\opcat$ is conformal and that, for all $(t,\theta)\in\Bbb{R}\times \Sph^1$:
$$
(e_\opcat^*g_0)(t,\theta) = r_0^{-2}\opCosh^2(t)(dt^2 + d\theta^2).
$$
Thus if $\Delta_\opcat$ denotes the Laplacian operator of the metric $e_\opcat^*\delta$, then:
$$
\Delta_\opcat = \frac{r_0^{2}}{\opCosh^2(t)}\Delta.
$$
Let $I$ be an interval, and let $f:I\rightarrow]0,\infty[$ be a smooth, positive function. We recall that the principle curvature vectors of the surface of revolution of $f$ lie in the directions parallel and normal to the direction of revolution. Moreover, the principle curvature in the direction of revolution (with respect to the outward-pointing normal) is equal to $1/(f\sqrt{1+(f')^2})$. When this surface is minimal, the principle curvature in the other direction is then equal to $-1/(f\sqrt{1+(f')^2})$. Thus, if $A$ denotes the shape operator of $e_\opcat$, then:
$$
\|A\|^2 = \frac{2r_0^2}{\opCosh^4(t)}.
$$
Thus by Lemma \ref{PropFormulaForJacobiOperatorOfMeanCurvature}:
$$
(\text{\rm J}^h_{\opcat}\varphi)(t,\theta) = -\frac{2r_0^2}{\opCosh^4(t)}\varphi(t,\theta)-\frac{r_0^2}{\opCosh^2(t)}(\Delta\varphi)(t,\theta),
$$
as desired. Finally, by Proposition \ref{PropFormulaForJacobiOperatorOfBoundaryAngle}, bearing in mind that the shape operator of the unit sphere in $\Bbb{R}^3$ coincides with $\opId$:
$$
(\text{\rm J}^\theta\varphi)(\pm t_0,\theta) = \varphi(\pm t_0,\theta) \mp t_0(\partial_t\varphi)(\pm t_0,\theta),
$$
and this completes the proof.\end{proof}

For any function $\varphi\in C^\infty([-t_0,t_0]\times \Sph^1)$, we consider the Fourier transform of $\varphi$ in the $\theta$ direction. For all $(t,\theta)\in\Bbb{R}\times \Sph^1$, we write:
$$
\varphi(t,\theta) = \sum_{n\in\Bbb{Z}}\varphi_n(t)e^{in\theta},
$$
where, for all $n\in\Bbb{Z}$, $\varphi_n$ is the $n$'th Fourier mode of $\varphi$.

\begin{proposition}
A function $\varphi\in C^\infty([-t_0,t_0]\times \Sph^1)$ is an element of $\opKer(\text{\rm J}_\opcat)$ if and only if, for all $n\in\Bbb{Z}$:
\begin{eqnarray}\label{EqnFourierTransformedJacobiOperator}
\varphi_n'' + (\frac{2}{\opCosh^2(t)} - n^2)\varphi_n \hfill&=& 0, \nonumber \\
\varphi_n(\pm t_0)\mp t_0\varphi_n'(\pm t_0)\hfill&=&0.
\end{eqnarray}
\end{proposition}
\begin{proof} Since $\varphi$ is smooth, its Fourier series converges in the $C^\infty$ sense. Since, in addition, the operator $\text{\rm J}_\opcat=(\text{\rm J}^h_{\opcat},\text{\rm J}^\theta_{\opcat})$ is linear, it follows that $\text{\rm J}_\opcat\varphi=0$ if and only if $\text{\rm J}_\opcat\varphi_n=0$ for all $n\in\Bbb{Z}$, and the result follows by Proposition \ref{PropFormulaForJacobiOperatorOfExtremalAnnulus}.\end{proof}

\begin{proposition}\label{PropCaseNEqualsZero}
\noindent There exists no non-trivial solution $\varphi_0\in C^\infty([-t_0,t_0])$ to \eqref{EqnFourierTransformedJacobiOperator} with $n=0$.
\end{proposition}
\begin{remark} The functions constructed in the proof of this result are obtained by considering the normal perturbations of $e_\opcat$ arising from dilatations and from translations in the $e_3$ direction.\end{remark}
\begin{proof} The solution space to any second-order, linear ODE (ignoring boundary conditions) is $2$-dimensional. By inspection, we verify that the solution space to \eqref{EqnFourierTransformedJacobiOperator} with $n=0$ is spanned by $u$ and $v$, where:
$$
u(t)=1-t\tanh t,\qquad v(t)=\tanh t.
$$
By inspection, we verify that no linear combination of these solutions satisfies the boundary conditions, and it follows that there exists no non-trivial solution to \eqref{EqnFourierTransformedJacobiOperator} with $n=0$, as desired.\end{proof}

\begin{proposition}\label{PropCaseNGreaterThanTwo}
There exists no non-trivial solution $\varphi_n\in C^\infty([-t_0,t_0])$ to \eqref{EqnFourierTransformedJacobiOperator} with $\left|n\right|\geqslant 2$.
\end{proposition}
\begin{proof} Choose $|n|\geqslant 2$ and define $f_n:[-t_0,t_0]\longrightarrow \Bbb{R}$ by
$$
f_n(t)=\frac{2}{\cosh^2t}-n^2.
$$
Since $|n|\geqslant 2$, we have that $f_n(t)\leqslant -2$. We now argue by contradiction. Suppose there exists a non-trivial solution, $\varphi_n$ to \eqref{EqnFourierTransformedJacobiOperator} with $\left|n\right|\geqslant 2$. Since \eqref{EqnFourierTransformedJacobiOperator} is linear, upon multiplying by $-1$ if necessary, we may assume that $\varphi_n(0)\geqslant 0$. Since \eqref{EqnFourierTransformedJacobiOperator} is even, upon replacing $\varphi_n(t)$ with $\varphi_n(-t)$ if necessary, we may assume that $\varphi_n'(0)\geqslant 0$. Since $\varphi_n$ is non-trivial, $\varphi_n(0)$ and $\varphi_n'(0)$ cannot both be equal to $0$. Observe that if $\varphi_n>0$ over an interval $I$, then $\varphi_n''=-f_n\varphi_n\geqslant2\varphi_n >0$ over $I$, and so $\varphi_n$ is strictly convex over $I$. We deduce that $\varphi_n(t),\varphi_n'(t)>0$ for all $t\in]0,t_0]$, and we therefore define $\gamma:]0,t_0]\rightarrow\Bbb{R}$ by:
$$
\gamma(t) = \frac{\varphi_n'(t)}{\varphi_n(t)}.
$$
Observe that, for all $t$:
$$
\gamma'(t) = -f_n(t) - \gamma(t)^2 \geqslant 2 - \gamma(t)^2.
$$
Moreover, since $\gamma(t)>0$ for all $t>0$, it follows that:
$$
\mliminf_{t\rightarrow 0}\gamma(t)\geqslant 0.
$$
Observe that $\beta(t):=\sqrt{2}\opTanh(\sqrt{2}t)$ satisfies:
$$
\beta'(t) = 2 - \beta(t)^2,
$$
with initial condition $\beta(0)=0$, and it follows that $\gamma(t)\geqslant\beta(t)=\sqrt{2}\opTanh(\sqrt{2}t)$ for all $t\in]0,t_0]$. In particular, bearing in mind that $t_0>1$:
$$
\gamma(t_0) \geqslant \beta(t_0) = \sqrt{2}\opTanh(\sqrt{2}t_0) > \sqrt{2}\opTanh(\sqrt{2}) > 1 > t_0^{-1}.
$$
However, the boundary condition implies that $\gamma(t_0)=t_0^{-1}$, which is absurd, and there therefore exists no solution to \eqref{EqnFourierTransformedJacobiOperator} with $\left|n\right|\geqslant 2$ as desired.\end{proof}

\begin{proposition}\label{PropCaseNEqualsOne}
The only non-trivial solutions to \eqref{EqnFourierTransformedJacobiOperator} with $n=\pm 1$ are given by:
$$
\varphi_{\pm 1}(t)=a\left(\opSinh(t)+\frac{t}{\opCosh(t)}\right),
$$
for some $a\in \Bbb{C}$.
\end{proposition}
\begin{remark} The functions constructed in the proof of this result are obtained by considering the normal perturbations of $e_\opcat$ arising from rotations about the axes $e_1$ and $e_2$ and from translations in the $e_1$ and $e_2$ directions.\end{remark}
\begin{proof} The solution space to any second-order ODE (ignoring boundary conditions) is $2$-dimensional. By inspection, we verify that the solution space to \eqref{EqnFourierTransformedJacobiOperator} with $n=\pm 1$ is spanned by $u$ and $v$, where:
$$
u = \sinh t+\frac{t}{\cosh t},\qquad v=\frac{1}{\cosh t}.
$$
By inspection $au+bv$ satisfies the boundary condition if and only if $b=0$, and this completes the proof.\end{proof}

\begin{proposition}\label{PropNullityOfExtremalAnnulus}
$\opKer(J)$ is $2$-dimensional.
\end{proposition}
\begin{proof} Choose $\varphi\in\opKer(\text{\rm J}^h,\text{\rm J}^\theta)$. By Proposition \ref{PropEllipticRegularityOfJacobiOperatorII}, $\varphi\in C^\infty([-t_0,t_0]\times \Sph^1)$, and so its Fourier series converges in the $C^\infty$ sense. For $n\in\Bbb{Z}$, let $\varphi_n\in C^\infty([-t_0,t_0])$ be the $n$'th Fourier mode of $\varphi$. By Proposition \ref{PropCaseNEqualsZero}, $\varphi_0=0$, by Proposition \ref{PropCaseNGreaterThanTwo}, $\varphi_n=0$ for all $\left|n\right|\geqslant 2$, and by Proposition \ref{PropCaseNEqualsOne}:
$$
\varphi_{\pm 1} = a\left(\opSinh(t)+\frac{t}{\opCosh(t)}\right),
$$
for some $a\in\Bbb{C}$. Thus:
$$
\varphi = \left(\opSinh(t)+\frac{t}{\opCosh(t)}\right)(a\opCos(\theta) + b\opSin(\theta)),
$$
for some $a,b\in\Bbb{R}$. The space of all such functions is $2$-dimensional, and this completes the proof.\end{proof}

\begin{proposition}\label{PropNonDegenerateFamiliesOfCatenoids}
If $S=\Sph^1\times[0,1]$ is the annulus, then there exists $\delta>0$ such that for all $t\in(-\delta,\delta)$, the family of embeddings $[e]\in\Cal{Z}(\left\{g_t\right\})$ which are invariant under rotation about some vector constitutes a non-degenerate family diffeomorphic to two disjoint copies of $\Bbb{R}\Bbb{P}^2$.
\end{proposition}
\begin{proof} We define $\Cal{I}_{t,+}:\Sph^2\rightarrow\Cal{Z}(\left\{g_t\right\})$ such that, for all $v\in \Sph^2$, $\Cal{I}_{t,+}(v)$ is the extremal catenoid of the metric $g_t$ with axis $v$, oriented such that its normal points towards the axis of rotation. We define $\Cal{I}_{t,-}:\Sph^2\rightarrow\Cal{Z}(\left\{g_t\right\})$ such that for all $v\in \Sph^2$, $\Cal{I}_{t,-}(v)=\Cal{I}_{t,+}(v)$ with the reverse orientation. We see that $\Cal{I}_{t,\pm}$ quotients down to a smooth embedding of $\Bbb{R}\Bbb{P}^2$ into $\Cal{E}$. By Proposition \ref{PropExtremalCatenoid}, $\Cal{I}_{t,\pm}(\Bbb{R}\Bbb{P}^2)$ accounts for all free boundary minimal embeddings in $\Cal{Z}(\left\{g_t\right\})$ which are invariant under rotation. By Proposition \ref{PropNullityOfExtremalAnnulus}, when $t=0$, the nullity of the Jacobi operator of $\Cal{I}_{0,\pm}(v)$ with respect to the metric $g_0$ is equal to $2$ for all $v\in \Bbb{R}\Bbb{P}^2$. Since the nullity is upper-semicontinuous, there exists $\delta>0$ such that for all $\left|t\right|<\delta$ and for all $v\in \Sph^2$, the nullity of the Jacobi operator of $\Cal{I}_{t,\pm}(v)$ with respect to the metric $g_t$ is at most $2$. Since $\Cal{I}_{t,\pm}$ is an embedding, by Proposition \ref{lowerboundondimensionofkernel}, the nullity of the Jacobi operator of $\Cal{I}_{t,\pm}(v)$ is also bounded below by the dimension of $\Bbb{R}\Bbb{P}^2$. It follows that the nullity of $\Cal{I}_{t,\pm}(v)$ with respect to the metric $g_t$ is equal to $2$, and we conclude that $\Cal{I}_{t,\pm}(\Bbb{R}\Bbb{P}^2)$ is a non-degenerate family, as desired.\end{proof}

\subsection{Calculating the degree}\label{CalculatingTheDegree} Let $\Sigma$ be a compact surface with boundary. Let $\delta$ be a positive real number chosen as in Proposition \ref{NonDegenerateFamiliesOfDisks} if $\Sigma$ is diffeomorphic to the disk, $D$; as in Proposition \ref{PropNonDegenerateFamiliesOfCatenoids} if $\Sigma$ is diffeomorphic to the annulus, $\Sph^1\times[0,1]$; and as in Proposition \ref{PropNoOtherNonDegenerateFamilies} otherwise. We have (c.f. \cite{WhiteII}):
\begin{proposition}\label{PropFaithfulActionsOfFiniteGroups}
For all $t\in(-\delta,\delta)$, there exists $N\in\Bbb{N}$ such that if $S \subseteq B$ is an embedded surface in $B$ which is diffeomorphic to $\Sigma$ and free boundary minimal with respect to $g_t$, then either:
\begin{itemize}
\item[(1)] $S$ is invariant by rotation about some unit vector $v$; or
\item[(2)] for all unit vectors $v\in \Sph^2$, and for all $k\geqslant N$, $R_{v,2\pi/k}(S)\neq S$.
\end{itemize}
\end{proposition}
\begin{proof} Suppose the contrary. There exists a sequence $(k_m)_\minn$ in $\Bbb{N}$ converging to $\infty$, a sequence $(v_m)_\minn$ of unit vectors in $\Bbb{R}^3$ and a sequence $(S_m)_\minn$ of embedded surfaces in $B$ diffeomorphic to $\Sigma$ such that for all $m$, $S_m$ is free boundary minimal with respect to $g_t$, is not invariant under rotation about any vector, but satisfies $R_{v_m,2\pi/k_m}(S_m)=S_m$. Upon extracting a subsequence, we may suppose that $(v_m)_\minn$ converges to a unit vector $v_\infty$ in $\Bbb{R}^3$, say. By Theorem \ref{ThmFraserLi}, upon extracting a further subsequence, we may suppose that $(S_m)_\minn$ converges to an embedded submanifold $S_\infty$ which is also diffeomorphic to $\Sigma$ and free boundary minimal with respect to $g_t$. We claim that $S_\infty$ is invariant under rotation about $v_\infty$. Indeed, choose $\theta\in\Bbb{R}$. Since $(k_m)_\minn$ converges to $\infty$, there exists a sequence $(l_m)_\minn\in\Bbb{Z}$ such that $(2\pi l_m/k_m)_\minn$ converges to $\theta$. However, for all $m$:
$$
R_{v_m,2\pi l_m/k_m}(S_m) = (R_{v_m,2\pi/k_m})^{l_m}(S_m) = S_m.
$$
\noindent Thus, upon taking limits, we find that $R_{v_\infty,\theta}(S_\infty)=S_\infty$, and since $\theta\in\Bbb{R}$ is arbitrary, it follows that $S_\infty$ is invariant under rotation about $v_\infty$, as asserted. If $\Sigma$ is diffeomorphic to the disk, $D$, then by Proposition \ref{PropExtremalDisk}, $S_\infty$ is the critical disk of the metric $g_t$ with axis $v$. By Proposition \ref{NonDegenerateFamiliesOfDisks}, the family of critical disks of the metric $g_t$ is non-degenerate. In particular, by Theorem \ref{ThmNonDegenerateFamilies}, this family is isolated in $\Cal{Z}(\left\{g_t\right\})$. Thus, for sufficiently large $m$, $S_m$ is also a critical disk of $g_t$. In particular, $S_m$ is invariant under rotation about some vector, which is absurd. If $\Sigma$ is diffeomorphic to the annulus, $\Sph^1\times[0,1]$, then, by Proposition \ref{PropExtremalCatenoid}, $S_\infty$ is the critical catenoid of the metric $g_t$ with axis $v$. By Proposition \ref{PropNonDegenerateFamiliesOfCatenoids}, the family of critical annuli of the metric $g_t$ is non-degenerate. In particular, by Theorem \ref{ThmNonDegenerateFamilies}, this family is isolated in $\Cal{Z}(\left\{g_t\right\})$. Thus, for sufficiently large $m$, $S_m$ is also a critical annulus of $g_t$. In particular, $S_m$ is invariant under rotation about some vector, which is absurd. It follows that $S_\infty$ is not diffeomorphic, either to the disk, $D$, or to the annulus, $\Sph^1\times[0,1]$. However, this is absurd by Proposition \ref{PropNoOtherNonDegenerateFamilies}, and the result follows.\end{proof}

\begin{theorem}\label{ThmCalculatingDegree}
$$
\opDeg(\Pi)=
\left\{\begin{array}{cl}\pm 2\hfill&\text{if}\ \Sigma\ \text{is diffeomorphic to}\ D;\hfill \\
\pm 2\hfill&\text{if}\ \Sigma\ \text{is diffeomorphic to}\ \Sph^1\times [0,1];\ \text{and}\\ \hfill
0\hfill&\text{otherwise.}\hfill\end{array}\right.
$$
\end{theorem}
\begin{remark}We recall that the degree theory constructed in this paper has been designed to count {\it oriented} surfaces. In the present case, this means that every free boundary minimal surface will be counted twice, once for each orientation, so that the degree will always be even.
\end{remark}
%\begin{remark} The sign of the degree is determined by calculating the Morse indices of $\text{\rm J}_\opdisk$ and $\text{\rm J}_\opcat$, and can be calculated using techniques similar to those employed in the proofs of Propositions \ref{PropNullityOfExtremalDisk} and \ref{PropNullityOfExtremalAnnulus} respectively.\end{remark}
\begin{proof} Let $Z_0\subseteq\Cal{Z}(\left\{g_t\right\})$ be the set of embeddings which are free boundary minimal with respect to $g_t$ and invariant under rotation with respect to some vector. By Propositions \ref{PropNoOtherNonDegenerateFamilies}, \ref{NonDegenerateFamiliesOfDisks} and \ref{PropNonDegenerateFamiliesOfCatenoids}, $Z_0$ constitutes a non-degenerate family. By Theorem \ref{ThmNonDegenerateFamilies}, there exists a neighbourhood $\Omega$ of $Z_0$ in $\Cal{E}$ such that:
$$
\Cal{Z}(\left\{g_t\right\})\minter\overline{\Omega} = Z_0.
$$
\noindent Upon reducing $\Omega$ if necessary, we may suppose that $\Omega$ is also invariant under the action of $\opSO(3)$. We first calculate the contribution to the degree from embeddings in $\overline{\Omega}^c$. Let $N$ be as in Proposition \ref{PropFaithfulActionsOfFiniteGroups} and let $v$ be a unit vector in $\Bbb{R}^3$. Pick $[e]\in\Cal{Z}(\left\{g_t\right\}|\overline{\Omega}^c)$. By definition of $N$, for all $p\geqslant N$, $
R_{v,2\pi/p} \circ e (\Sigma) \neq e(\Sigma)$. If, in addition, $p$ is prime, then for all $1\leqslant k<p$ we also have $R_{v,2\pi k/p}\circ e(\Sigma)\neq e(\Sigma)$. Since $e$ is minimal, there exists an open, dense subset $V$ of $\Sigma$ such that, for all $1\leqslant k<p$, $ R_{v,2\pi k/p}\circ e(V)\minter e(V) = \emptyset$. Choose $q\in V$ and let $U$ be a neighbourhood of $e(q)$ in $B$ such that for all $1\leqslant k<p$:
$$
R_{v,2\pi k/p}(U)\minter U=\emptyset,\qquad R_{v,2\pi k/p}(U)\minter e(\Sigma) = \emptyset.
$$
Let $X_0=\left\{x_0\right\}$ be the manifold consisting of a single point. Denote $g_{x_0}:=g_t$. We define the mapping ${g}:C^\infty(M)\times X_0\times M\rightarrow\opSym^+(TM)$ such that for all $f\in C^\infty(M)$:
$$
{g}_f := {g}(f,x_0,\cdot) = e^f g_t.
$$
Let $E$ be a finite-dimensional, linear subspace of $C^\infty(M)$ and for $r>0$, let $E_r$ be the closed ball of radius $r$ about $0$ in $E$ with respect to some metric. Extend $X_0$ to ${X}=E_r\times \{x_0\}$. Let $G\subseteq\opSO(3)$ be the subgroup generated by $R_{v,2\pi/p}$. Let $f_1,...,f_m$ be a basis of $\opKer(\text{\rm J}_{g_{x_0},e})$, let $\varphi_1,...,\varphi_m$ be as in Proposition \ref{PropFindingTheRightPerturbationSpace} with $U$ as above, and for $1\leqslant k\leqslant m$, define $\varphi'_k$ by:
$$
\varphi'_k = \sum_{l=1}^p \varphi_k\circ R_{v,2\pi l/p}.
$$
By definition, $\varphi'_k$ is invariant under the action of $G$. Let $E_{t,e}\subseteq C^\infty(M)$ be the linear span of $\varphi'_1,...,\varphi'_m$. As in the proof of Theorem \ref{PropSurjectivityForExtensions}, we show that if $E$ contains $E_{t,e}$, then $\text{\rm P}_{x_0,e}+\text{\rm J}_{x_0,e}$ defines a surjective map from $T_{x_0}{X}\times C^\infty(\Sigma)$ into $C^\infty(\Sigma)\times C^\infty(\partial \Sigma)$. Proceeding as in the proof of Theorem \ref{PropSurjectivityForExtensions}, we show that $E$ and $r$ may be chosen such that $g_x$ is invariant under the action of $G$ for all $x\in{X}$, $\partial_\omega\Cal{Z}(X|\Omega)=\partial_\omega \Cal{Z}(X|\overline{\Omega}^c)=\emptyset$, and $\text{\rm P}_{x,e} + \text{\rm J}_{x,e}$ defines a surjective map from $T_{x}{X}\times C^\infty(\Sigma)$ into $C^\infty(\Sigma)\times C^\infty(\partial \Sigma)$ for all $[e]\in\Cal{Z}(\left\{x\right\}|\overline{\Omega}^c)$. Thus, by Theorem \ref{ThmSolutionSpaceIsSmooth}, $\Cal{Z}(X|\overline{\Omega}^c)$ is a smooth manifold of finite dimension equal to $\opDim(X)$ and $\Pi(\partial(\Cal{Z}(X|\overline{\Omega}^c))\subseteq\partial{X}$.

Now, we let $x\in X$ be a regular value of the restriction of $\Pi$ to $\Cal{Z}(X|\overline{\Omega}^c)$. Since $g_x$ and $\Omega^c$ are both invariant under the action of $G$, it follows that $\Cal{Z}(\left\{x\right\}|\overline{\Omega}^c)$ decomposes into disjoint orbits of $G$. By Proposition \ref{PropFaithfulActionsOfFiniteGroups}, none of the orbits of $G$ in $\Cal{Z}(\left\{x_0\right\}|\Omega^c)$ is trivial, and so, by Theorem \ref{ThmFraserLi}, for $x$ sufficiently close to $x_0$, none of the orbits of $G$ in $\Cal{Z}(\left\{x\right\}|\overline{\Omega}^c)$ is trivial either. However, since $p$ is prime, all of the non-trivial orbits of $G$ have order $p$, so:
$$
\opDeg(\Pi|\overline{\Omega}^c)=\sum_{[e]\in\Cal{Z}(\left\{x\right\}|\overline{\Omega}^c)}\opSig(\text{\rm J}_{x,e}) = 0\ \text{mod}\ p.
$$

We now account for the embeddings in $\Omega$. By Theorem \ref{ThmIntegerValuedDegree}, there exists an extension $\tilde{X}$ of ${X}$ such that $\partial_\omega\Cal{Z}(\tilde{X}|\Omega)=\emptyset$ and $\Cal{Z}(\tilde{X})$ is a smooth manifold of finite dimension equal to $\opDim(\tilde{X})$. By Theorem \ref{ThmNonDegenerateFamilies}:
$$
\opDeg(\Pi|\Omega) = (-1)^{\opInd(Z_0)}\chi(Z_0).
$$
Combining these relations yields:
$$
\opDeg(\Pi) = (-1)^{\opInd(Z_0)}\chi(Z_0)\ \text{mod}\ p,
$$
and since $p>0$ is arbitrary, it follows that:
$$
\opDeg(\Pi) = (-1)^{\opInd(Z_0)}\chi(Z_0),
$$
and the result now follows by Propositions \ref{PropNoOtherNonDegenerateFamilies}, \ref{NonDegenerateFamiliesOfDisks} and \ref{PropNonDegenerateFamiliesOfCatenoids}.\end{proof}

\subsection{Proof of Theorem \ref{ThmMainTheorem}}We now complete the proof of Theorem \ref{ThmMainTheorem}. For $s\in\Bbb{R}$, denote ${g}_s:=e^{-2sf} g$, and let $\Rc^s$ be the Ricci-curvature tensor of this metric. Then:
$$
\left.\frac{\partial}{\partial_s}\right|_{\tiny s=0}\hspace{-15pt }{\Rc}^s = (n-2)\opHess f + \Delta f g > 0.
$$
Thus, for sufficiently small, positive $s$, ${g}_s$ has positive Ricci curvature. Trivially, for $s$ sufficiently small, $f$ is still strictly convex with respect to ${g}_s$. We now use Theorem \ref{ThmCalculatingDegree} to prove existence. Indeed, let $t_m$ be any sequence of positive numbers converging to $0$. Fix $m$ and let $X=\{g_{t_m}\}$ be the manifold consisting of a single point. By Theorem \ref{ThmIntegerValuedDegree}, there exists an extension $\tilde{X}$ of $X$ such that $\mathcal{Z}(\tilde{X})$ has the structure of a differential manifold of finite dimension equal to $\opDim(\tilde{X})$ and the canonical projection $\Pi:\mathcal{Z}(\tilde{X})$ has a well-defined integer valued degree. By Theorem \ref{ThmCalculatingDegree}, $\opDeg(\Pi)=\pm 2$, and, in particular, for any regular value $x$ of $\Pi$ in $\tilde{X}$, there exists an embedding $e_m:\Sph^1\times[0,1]\rightarrow B$ which is free boundary minimal with respect to $g_x$. Moreover, by Sard's Theorem, $g_m:=g_x$ may be chosen as close to $g_{t_m}$ as we wish, and we may therefore suppose that $(g_m)_\minn$ also converges to $g$. It now follows by Theorem \ref{ThmFraserLi} that there exists an embedded submanifold $\Sigma_\infty\subseteq B$ towards which $(\Sigma_m)_\minn$ converges. In particular, $\Sigma_\infty$ is diffeomorphic to $\Sph^1\times[0,1]$ and is free boundary minimal with respect to $g$, and this completes the proof.

\begin{remark}\label{RmkFinalRemark}
Observe that Theorem \ref{ThmCalculatingDegree} and the same argument as above also recovers the result \cite{GrueterJost} of Grueter and Jost.
\end{remark}

%%%%%%%%%%%%%%%%%%%%%%%%%%%%%%%%%%%%%%%%%%
%%%%%%%%%%%%Appendix A: Some formulas%%%%
%%%%%%%%%%%%%%%%%%%%%%%%%%%%%%%%%%%%%%%%%

%%%%%%%%%%%%%%%%%%%%%%%%%%%%%%%%%%%%%%
%%%%%%%%%%%%%Appendix A%%%%%%%%%%%%%%%%%%
%%%%%%%%%%%%%%%%%%%%%%%%%%%%%%%%%%%%%%

%%%%%%%%%%%%%%%%%%%%%%%%%%%%%
%%%Bibliography%%%%%%%%%%%%%%%
%%%%%%%%%%%%%%%%%%%%%%%%%%%%%%%%
\bibliographystyle{ams}

\begin{thebibliography}{99}

\bibitem{Ambrozio}
{L. Ambrozio.}
\newblock{\it Rigidity of area minimizing free boundary surfaces in mean convex three-manifolds.} {J. Geom. Anal.} to appear.

\bibitem{Aronszajn}
{N. Aronszajn.}
\newblock{\it A unique continuation theorem for solutions of elliptic partial differential equations or inequalities of second order.} {J. Math. Pures Appl.} {\bf 36}, (1957), 235--249

\bibitem{doCarmo}
{M. P. do Carmo.}
\newblock{\it Geometria riemanniana.} Projeto Euclides, {\bf 10}, IMPA, Rio de Janeiro (1979)

%\bibitem{ChoiSchoen}
%{ H. I. Choi and R. Schoen.}
%\newblock{\it The space of minimal embeddings of a surface into a three-dimensional manifold of positive Ricci curvature}, {Invent. Math.}, {\bf 81}, (1985), no. 3, 387--394

%\bibitem{ChowKnopf}
%{B. Chow and D. Knopf.}
%{\it The Ricci flow: an introduction}, Mathematical Surveys and Monographs {\bf 110} American Mathematical Society, Providence, RI, (2004)

%\bibitem{ColdingMinicozzi}
%{T. H. Colding and W. P. Minicozzi II.} {\it Minimal surfaces.} Courant Lecture Notes in Mathematics {\bf 4} New York University, Courant Institute of Mathematical Sciences, New York

\bibitem{Cou}
{R. Courant.}
{\it The existence of minimal surfaces of given topological structure under prescribed
boundary conditions.} Acta Math. {\bf 72} (1940), 51--98.

\bibitem{DLP}
{C. De Lellis and F. Pellandini.}
{\it Genus bounds for minimal surfaces arising from min-max constructions.}
J. Reine Angew. Math. {\bf 644} (2010), 47--99.

\bibitem{DHT}
{U. Dierkes, S. Hildebrandt and A. J. Tromba.}  {\it Regularity of minimal surfaces.} Revised and enlarged second edition. With assistance and contributions by A. K\"uster. Grundlehren der Mathematischen Wissenschaften, 340. Springer, Heidelberg, 2010. xviii+623 pp. ISBN: 978-3-642-11699-5

%\bibitem{EelsSalamon}
%{J. Eells and S. Salamon.}
%{\it Twistorial construction of harmonic maps of surfaces into four-manifolds,} {Ann. Scuola Norm. Sup. Pisa Cl. Sci. 4} {\bf 12} (1985), no. 4, 589--640

%\bibitem{ElworthyTromba}
%{K. D. Elworthy and A. J. Tromba.}
%{Degree theory on Banach manifolds, in {\it Nonlinear Functional Analysis} (Proc. Sympos. Pure Math.), Vol. XVIII, Part 1, Chicago, Ill., (1968), 86--94, Amer. Math. Soc., Providence, R.I.}

\bibitem{FraserLi}
{A. Fraser and M. Li.}
\newblock{\it Compactness of the space of embedded minimal surfaces with free boundary in three-manifolds with nonnegative Ricci curvature and convex boundary.} J. Differential Geom. to appear.

\bibitem{FrSc1}
{A. Fraser and R. Schoen.}
\newblock {\it The first Steklov eigenvalue, conformal geometry, and minimal surfaces.} Adv. Math. {\bf 226} (2011), no. 5, 4011--4030.

\bibitem{FrSc2}
{A. Fraser and R. Schoen.}
\newblock{\it Sharp eigenvalue bounds and minimal surfaces in the ball.}
\newblock Preprint (2013) arXiv:1209.3789.

\bibitem{GilbTrud}
{D. Gilbarg and N. S. Trudinger.}
\newblock{\it Elliptic partical differential equations of second order.} Die Grundlehren der mathemathischen Wissenschagten {\bf 224}, Springer-Verlag, Berlin, New York (1977)

\bibitem{Huisken}
{G. Huisken.}
{\it Contracting convex hypersurfaces in Riemannian manifolds by their mean curvature.} {Invent. Math.} {\bf 84} (1986), no. 3, 463--480

\bibitem{GrueterJost}
{M. Gr\"uter and J. Jost.}
\newblock{\it On embedded minimal disks in convex bodies.} {Ann. Inst. H. Poincar\'e, Anal. Non Lin\'eaire} {\bf 3} (1986), no. 5, 345--390

\bibitem{GuillemanPollack}
{V. Guillemin and A. Pollack.} {\it Differential Topology.} Prentice-Hall, Englewood Cliffs, N.J., (1974)

\bibitem{Kato}
{T. Kato.} {\it Perturbation theory for linear operators.} Grundlehren der Mathematischen Wissenschaften {\bf 132} Springer-Verlag, Berlin, New York, (1976)

\bibitem{Lew}
{H. Lewy.}
{\it On minimal surfaces with partially free boundary.} Comm. Pure Appl. Math. {\bf 4} (1951), 1--13.

\bibitem{Li}
{M. Li.}
\newblock{\it A General Existence Theorem for Embedded Minimal Surfaces with Free Boundary.} Comm. Pure Appl. Math. to appear.

%\bibitem{Mantegazza}
%{C. Mantegazza.}  {\it Lecture Notes on Mean Curvature Flow}, Progress in Mathematics, {\bf 290}, Birkh\"auser Verlag, Basil, (2011)

\bibitem{MY}
{W. Meeks and S.T. Yau.}
{\it Topology of three-dimensional manifolds and the embedding problems in minimal surface theory.} Ann. of Math. (2) {\bf 112} (1980), no. 3, 441Ð484.

\bibitem{Ni1}
{J. Nitsche.}
{\it Stationary partitioning of convex bodies.} Arch. Rational Mech. Anal. {\bf 89} (1985), no. 1, 1--19.

\bibitem{Ros1}
{A. Ros.}
{\it Stability of minimal and constant mean curvature surfaces with free boundary.} Mat. Contemp. {\bf 35} (2008), 221--240.

\bibitem{SmiRos}
{H. Rosenberg and G. Smith.}
{\it  Degree Theory of Immersed Hypersurfaces,} arXiv:1010.1879

\bibitem{Smy}
{B. Smyth.}
{\it Stationary minimal surfaces with boundary on a simplex.} Invent. Math. {\bf 76} (1984), no. 3, 411Ð420.

\bibitem{Stru}
{Struwe, M.}
{\it On a free boundary problem for minimal surfaces.} Invent. Math. {\bf 75} (1984), no. 3, 547--560.

\bibitem{TaylorI}
{M. E. Taylor.} {\it Partial differential equations I. Basic theory.} Applied Mathematical Sciences {\bf 115} Springer, New York, (2011)

%\bibitem{Tromba}
%{A. J. Tromba.} {\it The Euler characteristic of vector fields on Banach manifolds and a globalization of Leray-Schauder degree.} { Adv. in Math.} {\bf 28}, (1978), no. 2, 148--173

\bibitem{WhiteI}
{B. White.}
\newblock{\it The space of m-dimensional surfaces that are stationary for a parametric elliptic functional.} {Indiana Univ. Math. J.} {\bf 36} (1987), no. 3, 567--602

\bibitem{WhiteII}
{B. White.}
\newblock{\it The space of minimal submanifolds for varying Riemannian metrics.} Indiana Math. Journal \textbf{40} (1991), no.1, 161-200.





\end{thebibliography}

\end{document}